\documentclass[oneside,english,11pt]{amsart}
\usepackage[T1]{fontenc}
\usepackage[latin9]{inputenc}
\usepackage{amsthm}
\usepackage{amssymb}
\usepackage{pstricks, pst-node, pst-tree}
\usepackage{graphs}
\usepackage{float}

\newcommand{\wg}{\omega}
\newcommand{\Aff }{\mathrm{Aff}}

\newcommand{\Z}{\mathbb{Z}}
\newcommand{\D}{\mathbb{D}}
\newcommand{\ZN}{\mathbb{Z}_N}

\newcommand{\Sym}{\mathbb{S}}

\newcommand{\charf }{\mathrm{char}\,\fie }
\newcommand{\Alt}{\mathbb{A}}

\newcommand{\N}{\mathbb{N}}
\newcommand{\NA}{\mathfrak{B}}
\newcommand{\F}{\mathbb{F}}

\newcommand{\fie}{\Bbbk}
\newcommand{\Hilb}{\mathcal{H}}
\newcommand{\id}{\mathrm{id}}
\newcommand{\imm}{\mathrm{imm}}
\newcommand{\Inn}{\mathrm{Inn}}

\newcommand{\ndN }{\mathbb{N}}

\newcommand{\ndZ }{\mathbb{Z}}

\renewcommand{\pmod}[1]{\,\,(\mathrm{mod}\,#1)}

\newcommand{\SG}[1]{\mathbb{S}_{#1}}
\newcommand{\supp }{\mathrm{supp}\,}
\newcommand{\toba}{\mathfrak{B}}
\newcommand{\trid}{\triangleright}

\newcommand{\ydG}{ {}_{\fie G}^{\fie G}\mathcal{YD}}
\newcommand{\ydH}{ {}_{\fie H}^{\fie H}\mathcal{YD}}
\newcommand{\yedge}{\graphlinedash{3 4} \graphlinewidth{.03}}

\makeatletter
\numberwithin{equation}{section}
\numberwithin{figure}{section}
\numberwithin{table}{section}
\theoremstyle{plain}
\newtheorem{thm}{Theorem}[section]
\theoremstyle{plain}
\newtheorem{lem}[thm]{Lemma}
\newtheorem{cor}[thm]{Corollary}
\theoremstyle{plain}
\newtheorem{pro}[thm]{Proposition}
\theoremstyle{remark}
\newtheorem{rem}[thm]{Remark}
\theoremstyle{plain}
\newtheorem{exa}[thm]{Example}
\theoremstyle{plain}
\newtheorem{defn}[thm]{Definition}
\theoremstyle{plain}
\newtheorem{problem}[thm]{Problem}
\newtheorem{conjecture}[thm]{Conjecture}

\makeatother

\usepackage{babel}

\begin{document}

\title[Nichols algebras with many cubic relations]{Nichols algebras with many cubic relations}
\author{I. Heckenberger}
\author{A. Lochmann}
\author{L. Vendramin}
\address{Philipps-Universit\"at Marburg\\ 
FB Mathematik und Informatik \\
Hans-Meerwein-Stra\ss e\\
35032 Marburg, Germany}
\email{heckenberger@mathematik.uni-marburg.de}
\email{lochmann@mathematik.uni-marburg.de}
\email{lvendramin@dm.uba.ar}

\begin{abstract}
 Nichols algebras of group type
 with many cubic relations are classified under a technical
 assumption on the structure of Hurwitz orbits of the third power of the
 underlying indecomposable rack. All such Nichols algebras are
 finite-dimensional and their Hilbert series have a factorization into quantum
 integers. Also, all known finite-dimensional elementary Nichols algebras
 turn out to have many cubic relations. The technical assumption
 of our theorem can be removed if a conjecture in the theory of cellular
 automata can be proven.
\end{abstract}

\maketitle

\setcounter{tocdepth}{1}
\tableofcontents{}

\section{Introduction}
\label{section:intro}

Let $(V,c)$ be a braided vector space. A fundamental question in the theory of
Hopf algebras is whether the Nichols algebra $\NA (V)$ (which heavily depends
on $c$) is finite-dimensional.
If $V$ is of diagonal type, then $\NA (V)$ is known to have a restricted PBW
basis and the PBW generators are parametrized by Lyndon words. The multidegrees
of these Lyndon words can be regarded as positive roots in a generalized root
system, which has the symmetry of a Weyl groupoid. Moreover, the existence of the
restricted PBW basis implies that the Hilbert series of $\NA (V)$ is rational.

If the braiding $c$ is of group type, then the structure of $\NA (V)$ is much less
understood. Moreover, only a few finite-dimensional examples are known which are
not of diagonal type. One common feature of them is the factorization of the
Hilbert series of $\NA (V)$ as
\[ \Hilb _{\NA (V)}(t)=\prod _{i=1}^{k_1} (a_i)_t \prod _{i=1}^{k_2}(b_i)_{t^2}\]
for some $k_1,k_2\ge 0$, $a_1,\dots ,a_{k_1},b_1,\dots ,b_{k_2}\ge 2$, where
$(a)_{t^b}=1+t^b+t^{2b}+\cdots +t^{(a-1)b}$ for all $a,b\ge 1$. We say that
$\Hilb _{\NA (V)}(t)$ is $t$-\emph{integral of depth two}.
The $t$-integrality of all known Hilbert series of Nichols algebras of group type
was the starting point of a classification program in \cite{MR2803792} and
\cite{MR2891215}.
It turned out that under additional technical restrictions the $t$-integrality
of the Hilbert series is equivalent to an inequality on the dimensions of the
homogeneous components of low degree of the Nichols algebra, and that a complete
list can be given.

Let $G$ be a group, $\fie $ be a field and $V\in \ydG $.
Then the Nichols algebra $\NA (V)$ is called of group type.
We say that the Nichols algebra $\NA (V)$ is \emph{elementary}
if $V$ is finite-dimensional, absolutely irreducible and if its support
$$\supp V=\{x\in G\,\,|\,V_x\not=0\}$$
generates the group $G$. A fundamental problem in the theory of Nichols
algebras is the classification of finite-dimensional elementary Nichols
algebras. In fact, often it is not important to know $G$. Let $H$
be a group. We say
that two Yetter-Drinfeld modules $V\in \ydG $, $W\in \ydH $ are
\emph{bg-equivalent}\footnote{g refers to the grading by the group and b
refers to the braiding}
if there exists a bijection
$\varphi :\supp V\to \supp W$ and a
linear isomorphism $\psi :V\to W$ such that
$$\psi (V_g)=W_{\varphi (g)},\quad
%
%
\psi (gv)=\varphi(g)\psi (v)$$
for all $g,x\in \supp V$, $v\in V$. Two Nichols algebras of group type are
called \emph{bg-equivalent} if their degree one parts are bg-equivalent.
Then it is more convenient to ask for
all finite-dimensional elementary Nichols algebras up to bg-equivalence.
The answer to this problem is unknown to a large
extent. However, there are indications that the following conjecture is
possibly true:

\begin{conjecture} \label{conj:elemNA}
	\footnote{This conjecture was posed in the Oberwolfach mini-workshop
  ``Nichols algebras and Weyl groupoids'' in October 2012.}
All finite-dimensional elementary Nichols algebras are
$t$-integral of depth two. In particular, any such Nichols algebra is
bg-equivalent to one of
those listed in Table~\ref{tab:nichols}.
\end{conjecture}

In this paper we extend the results in \cite{MR2803792} and \cite{MR2891215}.
More precisely, we reduce the classification problem of
elementary Nichols algebras
with $t$-integral Hilbert series of depth two to a problem for cellular
automata on braid group orbits.

Let us discuss some details of our approach.
Assume first that the Hilbert series of $\NA (V)$
is $t$-integral of depth two. In \cite[Section
2.1]{MR2891215} it was shown that then the inequality
$$ \dim \ker \,(1+c_{12}+c_{12}c_{23})\ge \frac{1}{3}\dim V( (\dim
V)^2-1)$$
holds. Now, if this inequality holds for $V$, we say that $\NA(V)$ has many
cubic relations. Our intention is to classify all Nichols algebras with many
cubic relations and to prove that their Hilbert series are $t$-integral of depth
two. If $\supp V$ is a braided rack, then this claim was proven in
\cite{MR2891215}. Here we attack the general problem.

Recall that the braid group $\mathbb{B}_{3}$ can be presented by
gererators $\sigma _1$ and $\sigma _2$
and relation $\sigma _1\sigma _2\sigma _1=\sigma _2\sigma _1\sigma _2$.
The group $\mathbb{B}_{3}$ acts
on $(\supp V)^{3}$ by $\sigma _1\cdot(x,y,z)=(x\triangleright y,x,z)$ and
$\sigma _2\cdot(x,y,z)=(x,y\triangleright z,y)$ for all $x,y,z\in \supp V$,
where $\triangleright $ means conjugation. The orbits of the action of
$\mathbb{B}_3$ are called Hurwitz orbits. 
We assume that the
quotients of the Hurwitz orbits of $(\supp V)^3$ by the action of the center
of the braid group have only $\sigma _1$-cycles of length $\le 4$. (The
braidedness condition on racks
means that the Hurwitz orbits of $(\supp V)^3$
themselves have only $\sigma _1$-cycles of length $\le 3$.) This way we cover the
braided case and the
cases $\supp V=\Aff (5,i)$ with $i=2,3$, for which finite-dimensional Nichols
algebras are known to exist. On the other hand, we have to deal with
infinitely many Hurwitz orbits in contrast to \cite{MR2891215}.

Looking at homogeneous components, $\dim \ker \,(1+c_{12}+c_{12}c_{23})$
can be estimated from above by a purely combinatorial way.
Choose a subset $Y$ of $(\supp V)^3$ and take an element
$\overline{\alpha }\in \oplus _{(x,y,z)\in Y}V_x\otimes V_y\otimes V_z$.
Let $k(Y,\overline{\alpha})$ be the set of all $\alpha \in \ker
\,(1+c_{12}+c_{12}c_{23})$ with projection $\overline{\alpha }$ to
its homogeneous parts with degree in $Y$.
If $x,y,z\in \supp V$ and two of
\[
\{(x,y,z),\sigma _2\cdot(x,y,z),\sigma _1\sigma _2\cdot(x,y,z)\}
\]
are in $Y$ then for any $\alpha \in k(Y,\overline{\alpha })$ the summand with
the third degree is uniquely determined. Thus
$k(Y,\overline{\alpha })\subseteq k(Y',\overline{\alpha}')$, where $Y'$ is the
union of $Y$ and the third degree and $\overline{\alpha}'$ is the extension of
$\overline{\alpha }$.
This procedure of enlarging $Y$ can be regarded as a
cellular automaton on $(\supp V)^3$.
If $Y=(\supp V)^3$ then
$k(Y,\overline{\alpha})=\overline{\alpha}$ or
$k(Y,\overline{\alpha})=\emptyset $.
Hence, if a given subset $Y\subseteq (\supp V)^3$
can be enlarged this way to $(\supp V)^3$, then
the projection of $\ker \,(1+c_{12}+c_{12}c_{23})$ to the sum of
homogeneous parts of degree $(x,y,z)\in Y$ is injective.
Thus an
important question is the following: for a given Hurwitz orbit $\mathcal{O}$,
provide (the size of) a smallest subset $Y$ which can be enlarged by the above
process to $\mathcal{O}$. The size of such a $Y$ yields surprizingly often
a sharp upper bound for
$\dim \ker \,(1+c_{12}+c_{12}c_{23})$.
We call the quotient $|Y|/|\mathcal{O}|$ the \emph{immunity} of $\mathcal{O}$.

Our classification of finite-dimensional elementary Nichols algebras
with many cubic relations has three steps. First we solve the evolution
problem on the Hurwitz orbits in $X^3$ for all indecomposable racks $X$ such
that the quotient by the center has only $\sigma _1$-cycles of length $\le 4$.
This is one of our
main results. Then we conclude a small upper bound for the size of possible
racks. Using additional information on the cycle structure of racks and
the classification of indecomposable racks of size $\le 35$, we finally obtain
all racks satisfying the inequality on the immunity of $X^3$. For those racks
we then study the possible cocycles coming from Yetter-Drinfeld structures of
the enveloping group $G_X$ and find all elementary Nichols algebras
with many cubic relations, under the standing assumption on the Hurwitz
orbit structure.
This finishes our classification.

The paper is organized as follows. In Section \ref{section:racks} we review the
basic definitions concerning racks and we prove our first main result about the
structure of indecomposable racks. In Section~\ref{section:cycles} we provide an
obstruction on the cycle structure of indecomposable crossed sets.
Section \ref{section:hurwitz} is devoted to
the theory of coverings of Hurwitz orbits. In Section~\ref{section:automaton}
we define cellular automata on $G$-spaces for arbitrary groups $G$ and study
particular examples.
In Section \ref{section:immunities}
we prove our second main result formulated in Section~\ref{section:plagues}:
an explicit upper bound for the immunity of a class of Hurwitz orbits
in terms of local data. Finally, in Sections~\ref{section:manycubic}
and \ref{section:2cocycles}
we classify elementary Nichols algebras
with many cubic relations under
our standing assumption on Hurwitz orbits.

\section{Racks} 
\label{section:racks}

We recall basic notions and facts about racks. For additional information we
refer to \cite{MR1994219}.  A \emph{rack} is a pair $(X,\trid)$, where $X$ is
a non-empty set and $\trid:X\times X\to X$ is a map (considered as a binary
operation on $X$) such that
\begin{enumerate}
  \item the map $\varphi_i:X\to X$, where $x\mapsto i\trid x$, is bijective for all $i\in X$, and
  \item $i\trid(j\trid k)=(i\trid j)\trid(i\trid k)$ for all $i,j,k\in X$
    (i.\,e.~$\trid $ is \emph{self-distributive}).
\end{enumerate}

A rack $(X,\trid )$, or shortly $X$, is a \emph{quandle} if $i\trid i=i$ for
all $i\in X$.  A \emph{crossed set} $X$ is a quandle such that
$x\triangleright y=y$ if and only if $y\triangleright x=x$ for all $x,y\in X$.
A \emph{subrack} of a rack $X$ is a non-empty subset $Y\subseteq X$ such that
$(Y,\trid)$ is also a rack.  The \emph{inner group} of a rack $X$ is the
group generated by the permutations $\varphi_i$ of $X$, where $i\in X$. We
write $\Inn(X)$ for the inner group of $X$.
For any rack $X$, the \emph{enveloping group of} $X$
is the group $G_X$ given by the generating set $X$ and the relations
$xy=(x\trid y)x$
for all $x,y\in X$. A rack $X$ is called \emph{injective} if the canonical map
$X\to G_X$ is injective.


We say that a rack $X$ is \emph{indecomposable} if the inner group $\Inn(X)$
acts transitively on $X$. Also, $X$ is \emph{decomposable} if it is not
indecomposable.
The \emph{profile} of an indecomposable rack $X$ is the cycle structure
of the permutation $\varphi_x$ for some $x\in X$. For example, a profile
$1^a2^b3^c$ means that for any $x\in X$, $\varphi _x$ is a product of $b$
disjoint transpositions, $c$ disjoint $3$-cycles and $|X|=a+b+c$.

\begin{exa}
\label{exa:racks}
A group $G$ is a rack with $x\trid y=xyx^{-1}$ for all $x,y\in G$.  If a subset
$X\subseteq G$ is stable under conjugation by the elements of $X$,
then it is a subrack of $G$.
In particular, the conjugacy class $g^G$ of any $g\in G$ is a rack.
\end{exa}

\begin{rem}
For a union $X$ of conjugacy classes of a group $G$,
the subgroup of $G$ generated
by $X$ is a quotient of the enveloping group of $X$.
This implies that a rack is injective if and only if it is
isomorphic to a union of conjugacy classes of a group.
\end{rem}

\begin{exa}
\label{exa:affine}
Let $A$ be an abelian group, and $X=A$. For any $g\in\mathrm{Aut}(A)$ we have a
rack structure on $X$ given by 
\[
x\trid y=(1-g)x+gy
\]
for all $x,y\in X$.  This rack is called the \emph{affine rack} associated to
the pair $(A,g)$ and will be denoted by $\mathrm{Aff}(A,g)$. In particular, let
$p$ be a prime number, $q$ a power of $p$ and $\alpha\in\F_q\setminus\{0,1\}$.
We write $\mathrm{Aff}(\mathbb{F}_q, \alpha)$, or simply
$\mathrm{Aff}(q,\alpha)$, for the affine rack $\mathrm{Aff}(A,g)$, where
$A=\mathbb{F}_q$ and $g$ is the automorphism given by $x\mapsto\alpha x$ for
all $x\in\F_q$. 
\end{exa}

\begin{exa}
\label{exa:dihedral}
Let $p$ be a prime number. The affine rack $\Aff(p,-1)$ is called
\emph{dihedral rack} and it will be denoted by $\D_p$.
\end{exa}


Let $X$ be a finite indecomposable injective rack and let $x\in X$. In
\cite[Section 2.3]{MR2803792} integers $k_n$ for $n\in\N_{\geq2}$ were defined
by 
\begin{align*}
k_n=|\{y\in X\mid & \underbrace{x\trid(y\trid(x\trid(y\trid\cdots)))}_{n\text{
elements}}=y,\\
&\underbrace{x\trid(y\trid(x\trid(y\trid\cdots)))}_{j\text{
elements}}\not=y \text{ for all $j\in \{1,2,\dots ,n-1\}$}\}|.
\end{align*}
Since $X$ is indecomposable, the numbers $k_n$ do not depend on $x$.
Recall that for all $x\in X$ the permutation $\varphi_x$ has precisely $1+k_2$
fixed points.
Furthermore, $1+k_2+k_3+\cdots$ is just the cardinality of $X$.
We frequently use the notation \[
	k_m'=k_{m+1}+k_{m+2}+k_{m+3}+\cdots
\]
where $m\in\N$. The number of points moved by $\varphi_x$ is then
$k_2'=k_3+k_3'$.

\begin{lem} \label{le:S}
  Let $X$ be a crossed set and let $u\in X$. Then the subset
  $S=\{x\in X\,|\,u\trid x=x,\,x\not=u\}$ is a subquandle of $X$.
\end{lem}

\begin{proof}
  Let $x,y\in X$ with $u\trid x=x$, $u\trid y=y$. Then
  $$u\trid (x\trid y)= (u\trid x)\trid (u\trid y)=x\trid y.$$
  Further, $x\trid u=u$ since $X$ is a
  crossed set and hence $\varphi _x(S)\subseteq S$. Clearly, $\varphi _x|_S$ is
  injective. Finally, $\varphi _x^{-1}(S\cup \{u\})\subseteq S\cup \{u\}$ since
  $$u\trid (\varphi _x^{-1}(y))=\varphi _{u\trid x}^{-1}(u\trid y)=\varphi
  _x^{-1}(y)$$
  and $\varphi _x^{-1}(u)=u$. Thus $\varphi _x|_S$ is invertible with inverse
  $\varphi _x^{-1}|_S$.
\end{proof}

\begin{pro} \label{pro:TmovesT}
  Let $X$ be an indecomposable crossed set and let $u\in X$.
  Let $T=\{x\in X\,|\,u\trid x\not=x\}$.
  Then for all $t\in T$ there exists $t'\in T$ with $t\trid t'\not=t'$.
\end{pro}

\begin{proof}
  Let $S=X\setminus (T\cup \{u\})$ and
  $$T_0=\{t\in T\,|\,\supp \varphi _t\subseteq S\cup \{u\}\}.$$
  We have to show that $T_0=\emptyset$.

  Assume that $T_0\not=\emptyset$. Since $X$ is
  indecomposable, there exist $x\in T_0$, $y\in X\setminus T_0$, $z\in X$ such
  that $z\trid x=y$. As $x\trid z\not=z$ and $x\in T_0$,
  we conclude that $z\in S\cup \{u\}$.
  As $S\cup \{u\}$ is a subquandle of $X$ by Lemma~\ref{le:S}
  and $x=\varphi _z^{-1}(y)$, it follows that $y\in T\setminus T_0$. But
  $\supp \varphi _{z\trid x}=\varphi _z(\supp \varphi
  _x)\subseteq S\cup \{u\}$, a contradiction to $z\trid x=y\in T\setminus T_0$.
  This proves the proposition.
\end{proof}

\begin{thm}
	\label{thm:size_of_racks}
	Let $X$ be an indecomposable crossed set. Assume that the profile of $X$ is
	$1^{a_1}\;2^{a_2}\;\ldots\;k^{a_k}$. Then 
	\begin{equation*}
		k_2\leq\sum_{j\geq 2}a_j (k'_2-2)\quad\text{and}\quad |X|\leq\sum_{j\geq 2}a_j(k'_2-2) + k'_2 + 1.
	\end{equation*}
\end{thm}

\begin{proof}
  Let $u\in X$, $T=\{x\in X\,|\,u\trid x\not=x\}$
  and $S=X\setminus (T\cup\{u\})$.
  Then $S$ is a
  subquandle of $X$ by Lemma~\ref{le:S}.
  Let $S_1,\ldots,S_n$ be the $\Inn (S)$ orbits of $S$.
  Then $S\trid S_j=S_j$ for all $j$, and $S\trid T=T$ because $S$ is
  a subquandle and $S\trid \{u\}=\{u\}$.
  Since $X$ is indecomposable, there can only be elements of $T$
	to connect the $\Inn (S)$ orbits $S_j$ with each other, with $\{u\}$ and with $T$.

	Let $x\in T$, $s\in S_j$ and $y\in X\setminus S_j$ be such that $x\trid y=s$.
	Applying $\Inn (S)$ one can see that for all $s'\in S_j$ there must be elements
	$x'\in T$, $y'\in X\setminus S_j$ with $x'\trid y'=s'$. Hence, if one element
	of $S_j$ is connected to $X\setminus S_j$ via a single element of $T$, then
	every element of $S_j$ is. We conclude that each of the $k_2$ elements of $S$
	has to be moved by at least one $\varphi_x$ for an $x\in T$.

	Each $\varphi_x$, $x\in T$, moves $k'_2$
	elements within $X$. One of these elements must be $u$, and by
  Proposition~\ref{pro:TmovesT} another one must be an element of $T$.
  So there are at most $k'_2-2$ elements which may be moved by $\varphi _x$
  from $S$ to $X$ for each $x\in T$, and $k'_2(k'_2-2)$ altogether.
  Let
  $$Y=\{(x,s)\,|\,x\in T,\,s\in S,\,x\trid s\not= s\}.$$
  For any $(x,s)\in Y$,
	$\varphi_u (x)\trid s=\varphi _u(x)\trid \varphi _u(s)=\varphi _u(x\trid s)
  \not=\varphi _u(s)=s$,
	and $\varphi _u(x)\not=x$. So each $s\in S$ appears not only once as a
  second component in some $(x,s)\in Y$, but
  at least $j$ times, where $j=\min \{m\in \ndN \,|\, \varphi _u^m(x)=x\}$.
	For each $j$, there are $ja_j$ such $x\in T$, which
	provide (since $x\trid u\not=u$)
  at most $ja_j(k'_2-2)$ pairs $(x,s)\in Y$.
	In these pairs we may have up to $a_j(k'_2-2)$ different elements
  from $S$ as a second component, and hence $k_2\le \sum _{j\ge
  2}a_j(k'_2-2)$.
	This proves the first formula of the claim. The second now follows
  immediately from $|X|=1+k_2+k'_2$.
\end{proof}

\begin{cor}
	Under the assumptions of Theorem \ref{thm:size_of_racks}, 
	$k_2\leq\frac{1}{2}\,k'_2(k'_2-2)$ and $|X|\leq\frac{1}{2}k^{'2}_2+1$.
\end{cor}

\begin{proof}
  Use the inequality $\sum _{j\ge 2}a_j\le \frac{1}{2}\sum _{j\ge
  2}ja_j=k'_2/2$.
\end{proof}

\section{On the cycle structure of racks}
\label{section:cycles}

\begin{lem} 
	\label{le:abdec}
  Let $d\in \ndN $, $a,b\in \ndN _{\ge 2}$ with $\gcd (a,b)=1$
  and let $X=\{1,\dots ,d\}$
  be an indecomposable crossed set.
  Assume that $\varphi _1^a(x)=x$ or $\varphi _1^b(x)=x$ for all $x\in X$.
  Let
  \begin{align*}
    S=&\,\{x\in X\,|\,1\trid x=x,x\not=1\},\\
    T_a=&\,\{x\in X\,|\,\varphi _1^a(x)=x,\,1\trid x\not=x\},\\
    T_b=&\,\{x\in X\,|\,\varphi _1^b(x)=x,\,1\trid x\not=x\}.
  \end{align*}
  Then $X$ is the disjoint union of $\{1\},S,T_a$ and $T_b$ and we have
  \begin{align*}
    1\trid 1=&\,1,& 1\trid S=&\,S,& 1\trid T_a=&\,T_a,& 1\trid T_b=&\,T_b,\\
    S\trid 1=&\,1,& S\trid S=&\,S,& S\trid T_a=&\,T_a,& S\trid T_b=&\,T_b,\\
    & & T_a\trid T_b=&\,T_b,& T_b\trid T_a=&\,T_a.
  \end{align*}
\end{lem}

\begin{proof}
  The action of $1$ is clear by definition of $S$, $T_a$ and $T_b$.
  $X$ is a crossed set, therefore $s\trid 1=1$ holds for each $s\in S$.
  Moreover, $S\trid S=S$ by Lemma~\ref{le:S}.
  Let $s\in S$ and $t\in T_a$ be arbitrary.
  Then $s\trid t=s\trid\varphi_1^a(t)=\varphi_1^a(s\trid t)$ because
  $1$ and $s$ commute. Hence, $s\trid t\in T_a\cup S\cup\{1\}$. But
  $\varphi_s$ is bijective, so $s\trid t\in T_a$; similar for $t\in T_b$.
  Now let $x\in T_a$ and $y\in T_b$ be arbitrary. Assume $x\trid y\in T_a\cup S\cup \{1\}$.
  Then apply $\varphi_1^a$ to get $x\trid \varphi_1^a(y)=x\trid y$,
  hence $\varphi_1^a(y)=y$, which is a contradiction; so $T_a\trid T_b=T_b$
  and similarly $T_b\trid T_a=T_a$.
\end{proof}

\begin{pro} \label{pro:incommensurable_profiles}
	Let $r,d\in \N_{\ge 1}$, $a_1,\dots ,a_r,b\in \N_{\ge 2}$ with $\gcd
	(a_1\cdots a_r,b)=1$ and $d\ge a_1+\cdots +a_r+b+1$.  For all $i\in \{1,\dots
	,r\}$ let $\sigma _i$ be an $a_i$-cycle and let $\tau $ be a $b$-cycle
  in $\SG d$
	with pairwise disjoint supports. Then there is no indecomposable crossed set
	$X=\{1,\dots ,d\}$ with $\varphi _1=\sigma _1\cdots \sigma _r\tau $.
\end{pro}

\begin{proof}
  Assume to the contrary that $X$ is an indecomposable crossed set with
  $\varphi _1=\sigma _1\cdots \sigma _r\tau $.
  Then $1\notin \supp \sigma _i$ for all $i$
  and $1\notin \supp \tau $. Let
  $y_1,\dots ,y_b\in X$ with $1\trid y_i=y_{i+1}$ for all $1\le i<b$ and
  $1\trid y_b=y_1$. Then
  $$\tau =(y_1\,\cdots \,y_b).$$
  Let
  $$T=\bigcup _{i=1}^r \supp \sigma _i,\qquad S=X\setminus
  \left(\{1\}\cup T\cup \supp \tau \right).$$
  Then Lemma~\ref{le:abdec} applies with $a=a_1\cdots a_r$,
  $T_a=T$, $T_b=\supp \tau $.
  We consider four cases.

  \textbf{Case 1.} $y_1\trid 1\in S$. Let $s=y_1\trid 1$. Then
  $$ \varphi _s=\varphi _{y_1\trid 1}=\varphi _{y_1}\trid \varphi _1
  =\varphi _{y_1}\sigma _1\cdots \sigma _r\varphi _{y_1}^{-1}
  \cdot ( y_1 \, y_1\trid y_2\cdots \,y_1\trid y_b).$$
  As $s\in S$, we know from Lemma~\ref{le:abdec}
  that $s\trid \supp \tau =\supp \tau $
  and hence $y_1\in \supp \tau $ implies that $y_1\trid y_i\in \supp \tau $
  for all $i\in \{1,\dots ,b\}$. Thus $y_1\trid \supp \tau =\supp \tau $.
  Applying $\varphi _1$ and using the transitivity of the action of $\varphi
  _1$ on $\supp \tau $ we conclude that $y_i\trid \supp \tau =\supp
  \tau $ for all $i\in \{1,\dots ,b\}$. Thus $\supp \tau $ is $X$-stable by
  Lemma~\ref{le:abdec} which
  contradicts the indecomposability of $X$.

  \textbf{Case 2.} $y_1\trid 1\in T$. Since $y_1\in \supp \tau $,
  Lemma~\ref{le:abdec} yields that $y_1\trid T=T$ and hence $\varphi
  _{y_1}^{-1}(T)=T$, a contradiction to $1\in \varphi _{y_1}^{-1}(T)$.

  \textbf{Case 3.} $y_1\trid 1\in \supp \tau $,
  $\varphi _{y_1}^{a_1\cdots a_r}(1)=1$. Let $a=a_1\cdots a_r$.
  Then $\varphi _{y_1}^a\varphi _1\varphi _{y_1}^{-a}=\varphi _1$,
  and hence $\varphi _{y_1}^a\tau \varphi _{y_1}^{-a}=\tau $. Since $\varphi
  _{y_1}(y_1)=y_1$, we conclude that $\varphi _{y_1}^a(y)=y$ for all $y\in
  \supp \tau $. Conjugation by $\varphi _1$ yields
  \begin{align}\label{eq:phiy}
    \varphi _y^a(y')=y'\quad \text{for all $y,y'\in \supp \tau $.}
  \end{align}

  By assumption, $y_1\trid 1\in \supp \tau $. Further,
  $$ \varphi _{y_1\trid 1}=\varphi _{y_1}\varphi _1\varphi _{y_1}^{-1}
  =\varphi _{y_1}\sigma _1\cdots \sigma _r\varphi _{y_1}^{-1}\cdot
  (y_1 \,y_1\trid y_2\,\cdots y_1\trid y_b). $$
  Since $b\nmid a$, we conclude that
  $\varphi _{y_1\trid 1}^a(y_1)\not=y_1$ which is a contradiction to
  Equation~\eqref{eq:phiy}.

  \textbf{Case 4.} $y_1\trid 1\in \supp \tau $,
  $\varphi _{y_1}^b(1)=1$. By applying $\varphi _1$ and using
  Lemma~\ref{le:abdec} we obtain that
  $y\trid 1\in \supp \tau $ and $\varphi _y^b(1)=1$ for all $y\in \supp \tau
  $.
  In particular, for all $y\in \supp \tau $ there exist $\sigma _{y,1},\dots
  ,\sigma _{y,r},\tau _y\in \SG d$ with $1\in \supp \tau _y$ such that $\varphi
  _y=\sigma _{y,1}\cdots \sigma _{y,r}\tau _y$. Since
  \begin{align} \label{eq:phiy11}
    \varphi _{y_1\trid 1}=\prod _{i=1}^r(\varphi _{y_1}\sigma _i\varphi
    _{y_1}^{-1})\cdot (y_1\,y_1\trid y_2 \cdots y_1\trid y_b)
  \end{align}
  and $y_1\trid T=T$ by Lemma~\ref{le:abdec}, we conclude that $\supp \sigma
  _{y_1\trid 1 ,i}\subseteq T$ for all $i\in \{1,\dots ,r\}$. By conjugation
  of Equation~\eqref{eq:phiy11} with
  $\varphi _1$ and using $1\trid T=T$ and
  the transitivity of $\varphi _1$ on $\supp \tau $ it
  follows that
  \begin{align}
    \supp \sigma _{y,1}\cdots \sigma _{y,r}=T \quad
    \text{for all $y\in \supp \tau $.}
  \end{align}
  Let $x\in T$, $y\in \supp \tau $, $z\in X$ with $x\trid y=z$.
  Then $z\in \supp \tau $ by Lemma~\ref{le:abdec}. Further,
  $$\sigma _{z,1}\cdots \sigma _{z,r}\tau _z=\varphi _z
  =\varphi _{x\trid y}=\varphi _x (\sigma _{y,1}\cdots \sigma _{y,r})\varphi
  _x^{-1}\cdot \varphi _x\tau _y\varphi _x^{-1}
  $$
  and hence $\sigma _{z,1}\cdots \sigma _{z,r}=
  \varphi _x (\sigma _{y,1}\cdots \sigma _{y,r})\varphi _x^{-1}$. Thus $\varphi
  _x(T)=T$. We conclude that $T\trid T=T$ and hence $X\trid T=T$ by
  Lemma~\ref{le:abdec}, a contradiction to the indecomposability of $X$.
\end{proof}

\section{Quotients of Hurwitz orbits}
\label{section:hurwitz}

For the general theory of braid groups we refer to \cite{MR2435235}.  Recall
that the braid group in three strands can be presented as 
\[
\mathbb{B}_3=\langle\sigma_1,\sigma_2\mid\sigma_1\sigma_2\sigma_1
=\sigma_2\sigma_1\sigma_2\rangle.
\]
Let $X$
be a finite rack and $(x,y,z)\in X^3$. Then the braid group $\mathbb{B}_3$ acts
on $X^3$ via the Hurwitz action: $$ \sigma _1\cdot (x,y,z)=(x\trid y,x,z),\quad
\sigma _2\cdot (x,y,z)=(x,y\trid z,y)$$ for all $x,y,z\in X$. The orbits of
this action are called \emph{Hurwitz orbits}.  We write
$\mathcal{O}=\mathcal{O}(x,y,z)$ for the Hurwitz orbit of $(x,y,z)$.  According
to Brieskorn \cite{MR975077}, this action of $\mathbb{B}_3$ on $X^3$ was
implicitly considered by Hurwitz in \cite{MR1510692}.

\begin{lem} \label{le:OtoInnX}
  Let $X$ be a rack and $\mathcal{O}\subseteq X^3$ a Hurwitz orbit.
  Then the map $\mathcal{O}\to G_X$, $(x,y,z)\mapsto xyz$ is constant.
\end{lem}

\begin{proof}
  According to the definition of the Hurwitz action,
  it suffices to show that $xy=(x\trid y)x$ for all $x,y\in X$.
  The latter holds by the definition of $G_X$.
\end{proof}

\begin{lem}
\label{lem:black_and_red_cycles}
Let $X$ be an injective rack, $\mathcal{O}\subseteq X^3$ a Hurwitz orbit and let
$(x,y,z)\in\mathcal{O}$. Then 
$|(x,y,z)^{\langle\sigma_1\rangle}\cap
(x,y,z)^{\langle\sigma_2\rangle}|=1$.
\end{lem}

\begin{proof}
  The elements of $(x,y,z)^{\langle\sigma_1\rangle}\cap
  (x,y,z)^{\langle\sigma_2\rangle}$ are of the form $(x,y',z)$ with
  $y'\in X$. For all such triples, $xyz=xy'z$ in $G_X$ by Lemma~\ref{le:OtoInnX}.
  Thus $y=y'$ in $G_X$ and hence $y=y'$ in $X$ since $X$ is injective.
\end{proof}

Let $(x,y,z),(x',y',z')\in X^3$. We define the following relation on $X^3$:
\begin{equation}
	\label{eq:delta_equivalence}
	(x,y,z)\sim(x',y',z')\Leftrightarrow\Delta^m\cdot (x,y,z)=(x',y',z')\text{ for some }m\in\Z,
\end{equation}
where $\Delta=(\sigma_1\sigma_2)^3$.  Clearly, $\sim$ is an equivalence
relation. We write $\overline{\mathcal{O}}(x,y,z)$ for the equivalence class
containing $(x,y,z)$.
The set $\overline{\mathcal{O}}$ of equivalence classes of $\mathcal{O}$ is
called a \emph{Hurwitz orbit quotient}.

Let $x=\sigma_2^{-1}\sigma_1^{-1}Z(\mathbb{B}_3)$ and
$y=\sigma_1\sigma_2\sigma_1 Z(\mathbb{B}_3)$.  Then, by construction, the group 
\[
\langle x,y\mid x^3=y^2=1\rangle\simeq \mathbf{PSL}(2,\Z)\simeq\mathbb{B}_3\slash\langle\Delta\rangle
\]
acts on
$\overline{\mathcal{O}}$, see for example \cite[Appendix A]{MR2435235}. Since
$\mathbb{B}_3$ acts transitively on $\mathcal{O}$, the action of
$\mathbf{PSL}(2,\Z)$ on $\overline{\mathcal{O}}$ is transitive, that is,
$\overline{\mathcal{O}}$ is a homogeneous space.

An \emph{$xy$-cycle} in a $\mathbf{PSL}(2,\Z)$-space (in
particular, in a Hurwitz orbit quotient) is a minimal non-empty subset $C$ such
that $xy\cdot v\in C$ and $(xy)^{-1}\cdot v\in C$ for all $v\in C$.
We write $C_{xy}(v)$ for the $xy$-cycle containing a fixed element $v$.
Similarly, one defines \emph{$yx$-cycles} $C_{yx}(v)$.

We intend to determine all finite homogeneous
$\mathbf{PSL}(2,\Z)$-spaces (up to isomorphism) such that any
$xy$-cycle has at most $4$ elements.  Finite homogeneous
$\mathbf{PSL}(2,\Z)$-spaces up to isomorphism are known to be in
bijection with conjugacy classes of finite index subgroups of the modular group
$\mathbf{PSL}(2,\Z)$, which are intensively studied, see
e.\,g.~\cite{MR0498390}.  We present such
sets in terms of their Schreier graphs with respect to the generators $x$ and
$y$ of $\mathbf{PSL}(2,\Z)$. Here a Schreier graph of a subgroup
$H\subseteq \mathbf{PSL}(2,\Z)$
will be an oriented labeled graph with vertices corresponding to the left $H$-cosets.
In the interpretation as a $\mathbf{PSL}(2,\Z )$-space, the vertices correspond
to the points of the space.
In the Schreier graph, an $x$-arrow points from any coset $gH$ (equivalently, 
point of the $\mathbf{PSL}(2,\Z )$-space) to the coset $xgH$,
and a $y$-arrow points from any coset $gH$ to the coset $ygH$. Instead of a double
arrow labeled by $y$ we display a single dashed line. We then omit the label $x$
on the $x$-arrows. Later on we will add further labels on the graph which is
used to study Schreier graphs of finite index subgroups of coverings of
$\mathbf{PSL}(2,\Z )$.

\begin{pro}
	\label{pro:homogeneous_spaces}
  Let $M$ be a finite homogeneous $\mathbf{PSL}(2,\Z)$-space such
  that any $xy$-cycle in $M$ has at most $4$ elements.
  Then the Schreier graph of $M$ is one of the graphs in
  Figures~\ref{fig:Hoq1}--\ref{fig:Hoq18}.
\end{pro}

We denote the homogeneous $\mathbf{PSL}(2,\Z)$-spaces
and their Schreier graphs appearing in
Proposition~\ref{pro:homogeneous_spaces} by $\Sigma _{nX}$, where $n$ denotes
the number of vertices of the graph, and $X$ is a capital letter
which serves as a further distinction.

\begin{proof}
  The calculations are somewhat lengthy but elementary. It is reasonable to
  start with the classification of all homogeneous spaces with a point fixed by
  $x$, that is, with Schreier graphs containing an oriented loop. The
  restriction on the cycles yields that only the Schreier graphs in
  Figures~\ref{fig:Hoq1},
  \ref{fig:Hoq2},
  \ref{fig:Hoq5},
  \ref{fig:Hoq6},
  \ref{fig:Hoq11}, and
  \ref{fig:Hoq12}
  are possible. We explain this step in detail.

  Let $v_1$ be a point of $M$ with $xv_1=v_1$. If $yv_1=v_1$,
  then the Schreier graph of $M$ is in Figure~\ref{fig:Hoq1}.
  Otherwise, let $v_2=yv_1$. If $xv_2=v_2$,
  then the Schreier graph of $M$ is in Figure~\ref{fig:Hoq2}.
  Assume now that $xv_2\not=v_2$. Since $x^3v_2=v_2$, we have at least
  two more points $v_3=xv_2$ and $v_4=xv_3$, and then $xv_4=v_2$.
  Recall that each vertex of the Schreier graph
  has precisely one incoming and one outgoing black arrow
  and one outgoing dashed line, which implies that $v_3,v_4$ can neither be equal
  to an existing point nor to each other.
  If $yv_3=v_4$,
  then the Schreier graph of $M$ is in Figure~\ref{fig:Hoq5}.
  If $yv_4=v_4$, then let $v=yv_3$. The $xy$-cycle through $v$
  containst the vertices $v$, $xyv=v_4$,
  $xyv_4=v_2$, $xyv_2=v_1$, and $xyv_1=v_3$. Since it can contain at most $4$
  vertices, we conclude that $v=v_3$. Similarly, $yv_3=v_3$ implies that $yv_4=v_4$.
  Then the Schreier graph of $M$ is in Figure~\ref{fig:Hoq6}.
  Assume now that $yv_3\not=v_3$ and $yv_4\not=v_4$.
  Then we have at least two more points, $v_5=yv_3$ and $v_6=yv_4$.
  Since $(xy)^mv_6=v_6$
  for some $m\in \{1,2,3,4\}$ and since $xyv_6=v_2$, $(xy)^2v_6=v_1$,
  $(xy)^3v_6=v_3$, we obtain that $v_6=xyv_3=xv_5$. Since $x^3v_5=v_5$,
  there is a point $v_7=xv_6$ of $M$, and $xv_7=v_5$. If $yv_7=v_7$,
  then the Schreier graph of $M$ is in Figure~\ref{fig:Hoq11}.
  Finally, let $v_8=yv_7$. Then $v_8$, $v_5=xyv_8$,
  $v_4=xyv_5$, and $v_7=xyv_4$ are pairwise distinct, and hence $xyv_7=v_8$
  by assumption on the length of $xy$-cycles. Thus $xv_8=v_8$. The 
  Schreier graph of $M$ is then in Figure~\ref{fig:Hoq12}.

  Next, one proceeds with the homogeneous spaces having a point
  fixed by $xy$ but no point fixed by $x$. The corresponding Schreier graphs are
  Figures~\ref{fig:Hoq3} and \ref{fig:Hoq7}. Next, one can determine those
  homogeneous spaces which have a point fixed by $y$, but no points fixed by $x$
  or $xy$. The corresponding Schreier graphs are Figures~\ref{fig:Hoq4},
  \ref{fig:Hoq8}, \ref{fig:Hoq9}, \ref{fig:Hoq13}, and \ref{fig:Hoq14}. Finally,
  only Schreier graphs are left which have no oriented loops and no $y$-edges
  starting and ending in the same oriented triangle. Such graphs can be
  classified for example by looking at the number of $y$-edges between
  two oriented triangles. Under the restriction on the length of $xy$-cycles,
  one gets the Figures~\ref{fig:Hoq10},
  \ref{fig:Hoq15}, \ref{fig:Hoq16}, \ref{fig:Hoq17}, and \ref{fig:Hoq18}.
\end{proof}


For our analysis, finite homogeneous $\mathbb{B}_3$-spaces are relevant. We
obtain such spaces as coverings of homogeneous
$\mathbf{PSL}(2,\Z)$-spaces.

\begin{defn}
	\label{defn:covering}
	Let $T$ be a $\mathbf{PSL}(2,\Z)$-space.
	Consider $T$ as a $\mathbb{B}_3$-space on which
	$Z(\mathbb{B}_3)$ acts trivially.
	A \emph{covering} of $T$ is a triple $(p,S,T)$, where $S$ is a
	$\mathbb{B}_3$-space, and $p:S\to T$ is a surjective
	$\mathbb{B}_3$-equivariant map
	such that $p(s_1)=p(s_2)$ for $s_1,s_2\in S$
        implies that $s_1=\Delta ^ms_2$ for some $m\in \Z $. 
\end{defn}

\begin{rem}
	Let $T$ be a $\mathbf{PSL}(2,\Z)$-space
	and let $(p,S,T)$ be a covering of $T$.
 	If the action on $T$ is transitive then the action on $S$ is transitive.
\end{rem}

\begin{defn}
	A covering $(p,S,T)$ is said to be \emph{trivial} if $p:S\to T$ is bijective.
	We say that a covering $p:S\to T$ is \emph{finite} if $S$ (and hence
	$T$) is a finite set.
\end{defn}

Since $\Delta $ is contained in (and in fact generates)
the center $Z(\mathbb{B}_3)$, the following lemma is straightforward.

\begin{lem} \label{lem:N}
  Let $(p,S,T)$ be a covering of a homogeneous $\mathbf{PSL}(2,\Z)$-space $T$.
  Then $|p^{-1}(v)|=|p^{-1}(w)|$ for all points $v,w$ of $T$.
\end{lem}

Let $(p,S,T)$ be a covering of a homogeneous $\mathbf{PSL}(2,\Z)$-space $T$.
We will always write $N$ for the number of elements of the fibers.
 
\begin{rem} \label{rem:covgraph}
  Let $T$ be a homogeneous $\mathbf{PSL}(2,\Z)$-space.
   Coverings $(p,S,T)$ of $T$
  can be displayed as labeled Schreier graphs with respect to the generators
  $\sigma_2^{-1}\sigma _1^{-1}$
  and $\sigma _1\sigma _2\sigma _1$ of $\mathbb{B}_3$. Hence, by definition of
  $x$ and $y$, the generators
  $\sigma_2^{-1}\sigma _1^{-1}$
  and $\sigma _1\sigma _2\sigma _1$ correspond to labeled $x$- and $y$-arrows,
  respectively, in the labeled Schreier graph.
  Since the covering is a homogeneous
  space and the sequence
  $$ Z(\mathbb{B}_3) \to \mathbb{B}_3 \to \mathbf{PSL}(2,\Z) $$
  is exact, the fiber over any $v\in T$ consists of a $\langle \Delta \rangle
  $-orbit.
  We may fix a point $v[0]\in S$ in the fiber over a point $v\in T$ and 
  enumerate all other points of the fiber by $v[i]=\Delta ^i v[0]$ for all
  $i\in \{0,1,\dots ,N-1\}$, where $N$ is the size of the fiber.
  We also write $v[*]$ for the complete fiber $p^{-1}(v)$ over $v$.
  By choosing a spanning tree of the Schreier graph of $T$ and the images of
  $v[0]$ along the arrows of the spanning tree, one obtains the images of $v[i]$
  for all $i$ since $\Delta $ is central. The remaining arrows $v_i\to v_j$
  (those not on the spanning tree)
  in the graph of $T$ then have to obtain labels indicating the index shift in
  the fiber: a label $s$ tells that $v_i[k]$ is mapped to
  $v_j[k+s \pmod N]$ for all $k$. Then, up to the choice of the
  spanning tree, any covering of $T$ is uniquely determined by the labels 
  of the $x$- and $y$-edges.

	Observe that, since $\Delta =(\sigma _1\sigma _2)^3=(\sigma _1\sigma _2\sigma
	_1)^2$, the sum of the labels in any $x$-triangle is $-1$ and the sum of the
	two labels of a $y$-edge is $1$. The $y$-edges we interpret as double arrows
	and put the label of the arrow close to its destination.

  For any $xy$-cycle (or $yx$-cycle) $C$ in $T$,
  the label of $C$ is the sum of the labels of
  $x$- and $y$-edges of the cycle.
\end{rem}

\begin{defn} 
	\label{defn:graph_covering}
	Let $\mathcal{O}$ be a covering of a
  $\mathbf{PSL}(2,\Z)$-space.
  The \emph{graph of the covering} is the
	labeled Schreier graph as explained in
	Remark~\ref{rem:covgraph}. The arrows correspond to the action of
  $\sigma _2^{-1}\sigma_1^{-1}$
	and $\sigma _1\sigma_2\sigma _1$.
\end{defn}

\begin{defn}
	\label{defn:simpleadmisible_covering}
 	For $i\in \{1,2\}$, a $\sigma _i$-\emph{cycle} of a homogeneous
	$\mathbb{B}_3$-space is a minimal non-empty subset which is closed under the
	action of $\sigma _i$.
	We say that a $\mathbb{B}_3$-space $S$ (or a covering $(p,S,T)$ of a
	$\mathbf{PSL}(2,\Z)$-space $T$) has \emph{simply intersecting
	cycles} if any given $\sigma _1$-cycle $c_1$ and $\sigma _2$-cycle $c_2$
	in $S$ intersect in at most one point.
\end{defn}

\begin{exa}
	By Lemma \ref{lem:black_and_red_cycles}, if $\mathcal{O}\subseteq X^3$
  is a Hurwitz orbit, where $X$ is an injective rack,
  then $\mathcal{O}$ has simply intersecting cycles.
\end{exa}

\begin{rem}
	For any covering $(p,S,T)$ of a $\mathbf{PSL}(2,\Z)$-space $T$, the image of
	a $\sigma _1$-cycle in $S$ is an $xy$-cycle in $T$, and the image of a
	$\sigma _2$-cycle in $S$ is a $yx$-cycle in $T$.
\end{rem}

\begin{lem}
	\label{lem:loops}
	Let $T$ be a $\mathbf{PSL}(2,\Z)$-space
  and let $(p,S,T)$ be a covering of $T$ with simply intersecting cycles. 
	Let $v$ be a vertex of the graph of $T$. 
	\begin{enumerate}
		\item If there exists an $x$-loop on $v$ with label $a$ then
			$3a\equiv-1\pmod{N}$.
		\item If there exists a $y$-loop on $v$ with label $a$ then
			$2a\equiv1\pmod{N}$.
	\end{enumerate}
\end{lem}

\begin{proof}
	It follows from Remark \ref{rem:covgraph}.
\end{proof}

The following four lemmas are easy consequences of the definition of simply
intersecting cycles.

\begin{lem} \label{lem:trivial_covering}
  Let $T$ be a $\mathbf{PSL}(2,\Z)$-space
  and let $(p,S,T)$ be a covering of $T$ with simply intersecting cycles. 
  Let $v\in T$ and $w\in C_{yx}(v)\cap C_{xy}(v)$. If $w\not=v$,
  then $(p,S,T)$ is not trivial.
\end{lem}

\begin{lem}
	\label{lem:xy_and_yx_cycles}
	Let $T$ be a $\mathbf{PSL}(2,\Z)$-space
  and let $(p,S,T)$ be a covering of $T$ with simply intersecting cycles. 
  Let $v\in T$ and $N=|p^{-1}(v)|$.
	Let $\lambda \in \ZN $ and $\mu \in \ZN $
  be the labels of the $xy$- and $yx$-cycle containing $v$, respectively.
  Then $\langle\lambda\rangle\cap\langle\mu\rangle=0$ as subgroups of $\ZN $.
\end{lem}

\begin{lem}
	\label{lem:xy_and_yx_cycles_with_loops}
	Let $T$ be a $\mathbf{PSL}(2,\Z)$-space
  and let $(p,S,T)$ be a covering of $T$ with simply intersecting cycles. 
  Let $v\in T$ and assume that $xv=v$ or $yv=v$ and that
  $\mathbf{PSL}(2,\Z)v\not=\{v\}$.
  Then the labels of the $xy$- and $yx$-cycles containing $v$ are $0$.
\end{lem}



\begin{lem}
  \label{lem:from_v_to_w}
  Let $T$ be a $\mathbf{PSL}(2,\Z)$-space
  and let $(p,S,T)$ be a covering of $T$ with simply intersecting cycles. 
  Let $v,w\in T$, $N=|p^{-1}(v)|$
  and assume that $v\not=w$ and that $v,w$ are on the same $xy$-
  and the same $yx$-cycle. Let $\lambda$ and $\mu$ be the labels of the
  $xy$- and $yx$-path from $v$ to $w$, respectively.
  Then $\lambda\not\equiv\mu\pmod{N}$. 
\end{lem}

\begin{cor} \label{cor:two_y_loops}
  Let $T$ be a $\mathbf{PSL}(2,\Z)$-space.
  Let $v,w\in T$ and assume that $v\not=w$, $yv=v$, $yw=w$ and that
  $v,w$ are on the same $xy$-cycle. Then $T$ has no coverings with
  simply intersecting cycles.
\end{cor}

\begin{proof} Assume to the contrary that $(p,S,T)$ is a covering of $T$
  with simply intersecting cycles. Let $N=|p^{-1}(v)|$ and let $a$ and $b$ be
  the labels of the $y$-loops at $v$ and $w$, respectively. By
  Lemma~\ref{lem:loops}(2) we obtain that $2a\equiv 1\pmod{N}$ and $2b\equiv
  1\pmod{N}$ and hence $a\equiv b\pmod{N}$ and $N$ is odd. Since $yv=v$
  and $yw=w$, $v$ and $w$ are on the same $yx$-cycle and the
  $xy$- and $yx$-paths from $v$ to $w$ have the same labels. This is a
  contradiction to Lemma~\ref{lem:from_v_to_w}.
\end{proof}


%
%
%
%

\section{Cellular automata over homogeneous spaces}
\label{section:automaton}


As explained in the introduction, for the classification of Nichols algebras with many
cubic relations we are going to determine upper bounds for immunities of Hurwitz orbits.
The principle of the method to calculate these upper bounds is well-known from
the theory of cellular automata. Therefore we intend to formulate our techniques
in the language of cellular automata.

The following definition is inspired by \cite[Definition 1.4.1]{MR2683112}, where the very
similar definition of a cellular automaton over a group $G$ was given.
Let $G$ be a group acting transitively on a set $\Omega$ and let $A$ be a set.
Let $A^\Omega $ be the set of all functions from $\Omega $ to $A$.

\begin{defn}
  Let $S$ be a set, let $(g_s)_{s\in S}$ be a family of elements in $G$,
  and let $\mu :A^S\to A$
  be a map. Then the map $\tau :A^\Omega \to A^\Omega $ such that
  \begin{equation}
    \tau(f)(w)=\mu( (f(g_s\cdot w))_{s\in S})
  \end{equation}
  for all $f\in A^\Omega $, $w\in\Omega$, is called a
  \emph{cellular automaton} over $(G,\Omega)$ with alphabet $A$.
  The infinite sequence $(\tau ^n(f))_{n\ge 0}$ is called the \emph{evolution of} $f$.
\end{defn}

\begin{rem}
  A good interpretation of a cellular automaton over $(G,\Omega )$ is the following.
  For any $w\in \Omega $,
  consider the family of points $(g_s\cdot w)_{s\in S}$ as the neighborhood of $w$.
  Then for any function $f\in A^\Omega $, the value of $\tau (f)$ at $w$ is obtained
  from the values of $f$ in the neighborhood of $w$ according to the rule determined
  by $\mu $.
\end{rem}

We will only consider cellular automata over $(G,\Omega )$ with the alphabet
$A=\Z_2$. For any function $f\in \Z _2^\Omega $ let
$$ \supp \,f=\{w\in \Omega \,|\,f(w)=1\}.$$
Conversely, the characteristic function of a set $I\subseteq \Omega $ is
\begin{equation} \label{eq:f}
  \chi _I\in \Z_2^\Omega ,\quad 
	x\mapsto \begin{cases}
		1&\text{if $x\in I$},\\
		0&\text{otherwise}.
	\end{cases}
\end{equation}

\begin{defn}
  Let $\tau $ be a cellular automaton over $(G,\Omega )$ with alphabet
  $\Z_2$. We say that $\tau $ is \emph{monotonic} if
  \begin{enumerate}
    \item $\supp f\subseteq \supp \tau (f)$ for all $f\in \Z_2^\Omega $, and
    \item $\supp \tau (f)\subseteq \supp \tau (g)$
      for all $f,g\in \Z_2^\Omega $ with $\supp f\subseteq \supp g$.
  \end{enumerate}
\end{defn}

In what follows we will only study monotonic cellular automata over
homogeneous spaces.

\begin{defn}
  Let $\tau $ be a monotonic cellular automaton over $(G,\Omega )$
  with alphabet $\Z_2$.
  For any two subsets $I,J\subseteq \Omega $ with $I\subseteq J$
  we say that $I$ \emph{spreads to} $J$, if
  $J\subseteq \supp \tau ^n(\chi _I)$ for some $n\in \N $.
  A subset $I\subseteq \Omega $ is a \emph{quarantine}
  if $\tau (\chi _I)=\chi _I$.
  A subset $I\subseteq \Omega $ is a \emph{plague}
  if the smallest quarantine containing $I$ is $\Omega $.
  The cardinality of a plague $I$ is also called its \emph{size}.
\end{defn}

\begin{rem}
Let $\tau $ be a monotonic cellular automaton over $(G,\Omega )$
with alphabet $\Z_2$. If a subset $I$ spreads to another subset $J$ of $\Omega $,
then any subset $I'\subseteq \Omega $ with $I\subseteq I'$ spreads to $J$.

Assume that $\Omega $ has only finitely many points.
Then a subset $I$ of $\Omega $ is
a plague if and only if it spreads to $\Omega $. In this case, any subset
of $\Omega $ containing $I$ is a plague.
\end{rem}

Plagues of Hurwitz orbits were already introduced in
\cite[Definition~3]{MR2891215}. That definition is a special case
of a plague of the monotonic cellular automaton in the following example,
which is the main example of our interest.

\begin{exa}
	\label{exa:B3}
  Let $G=\mathbb{B}_3$ and let $\Omega $ be a finite homogeneous $G$-space.
  For example, if $X$ is a finite rack, then $G$ acts on $X^3$ by
	$\sigma _1\cdot (x,y,z)=(x\triangleright y,x,z)$ and $\sigma _2\cdot
	(x,y,z)=(x,y\triangleright z,y)$, and $\Omega $ can be taken as an orbit of
  this action.
  Let
  $$(g_s)_{s\in \{1,\dots ,7\}}=(1,\sigma _2,\sigma _1\sigma _2,
  \sigma _2^{-1}\sigma _1^{-1},\sigma _1^{-1},\sigma _2^{-1},\sigma _1)
  \in \mathbb{B}_3^7.$$
  As an alphabet take $A=\Z_2$.
  We consider the configuration given in Figure \ref{fig:neighbors}.
	(In Figure \ref{fig:neighbors} the black arrow indicates the action of
  $\sigma_1$,
	and the dashed arrow the action of $\sigma_2$.)
  Then $x_s=g_s\cdot x_1$ for all $s\in \{1,2,\dots ,7\}$.
  Define
	$\mu:A^7\to A$ by 
	\begin{align*}
	\mu(f_1,f_2,\dots,f_7)=&\;f_1\vee f_2f_3 \vee f_4f_5 \vee f_6f_7\\
        =&\;1-(1-f_1)(1-f_2f_3)(1-f_4f_5)(1-f_6f_7),
	\end{align*}
        where $f_1,\dots ,f_7\in A$, and $\vee $ denotes logical or.
  Then the map $\tau$
  defined by $\mu $ and $(g_s)_{s\in \{1,\dots ,7\}}$ is a monotonic
  cellular automaton over $(\mathbb{B}_3,\Omega )$.
  A plague of this automaton is literally the
  same as a plague in the sense of \cite[Definition~3]{MR2891215}.
  \begin{figure}[h]
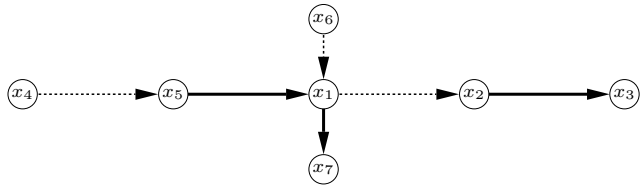

    \begin{graph}(7,2)
			\graphnodesize{.4}
      \roundnode{n1}(4,1)[\graphlinewidth{.01}\graphnodecolour{1}]
			\autonodetext{n1}{\tiny $x_1$}
			\autonodetext{n2}{\tiny $x_2$}
			\autonodetext{n3}{\tiny $x_3$}
			\autonodetext{n4}{\tiny $x_4$}
			\autonodetext{n5}{\tiny $x_5$}
			\autonodetext{n6}{\tiny $x_6$}
			\autonodetext{n7}{\tiny $x_7$}
      \roundnode{n2}(6,1)[\graphlinewidth{.01}\graphnodecolour{1}]
      \roundnode{n3}(8,1)[\graphlinewidth{.01}\graphnodecolour{1}]
      \roundnode{n4}(0,1)[\graphlinewidth{.01}\graphnodecolour{1}]
      \roundnode{n5}(2,1)[\graphlinewidth{.01}\graphnodecolour{1}]
      \roundnode{n6}(4,2)[\graphlinewidth{.01}\graphnodecolour{1}]
      \roundnode{n7}(4,0)[\graphlinewidth{.01}\graphnodecolour{1}]
			\diredge{n1}{n2}[\graphlinedash{1}]
			\diredge{n2}{n3}[\graphlinewidth{.04}]
  		\diredge{n4}{n5}[\graphlinedash{1}]
      \diredge{n5}{n1}[\graphlinewidth{.04}]
  		\diredge{n6}{n1}[\graphlinedash{1}]
  		\diredge{n1}{n7}[\graphlinewidth{.04}]
    \end{graph}
		\caption{Neighbors of $x_1$}
		\label{fig:neighbors}
  \end{figure}
\end{exa}

\begin{problem}
	\label{problem:percolation}
	Find a plague of minimal size for the cellular automata of Example
	\ref{exa:B3}. 
\end{problem}

\begin{exa}
  A cellular automaton in the very traditional sense is just a cellular
  automaton over $(\Z ^n,\Z ^n)$ with $n\in \{1,2\}$,
  where the action is given by the usual free
  action of $\Z ^n$ on itself and in $(g_s)_{s\in S}$ only $n\in \{1,-1\}$
  ($(n_1,n_2)\in \Z ^2$ with $|n_1|,|n_2|\le 1$, respectively,) appear.
\end{exa}

\subsection{Examples over $(\Z,\Z_m)$}

In this subsection we present a family of examples of automata useful for
studying the following process. 

Let $m\in\N_{\geq2}$, $f\in \Z_2^{\Z_m }$, $r\in \N $,
and $a_1,\dots,a_r\in\Z_m\setminus \{0\}$.
The group $G=\Z$ acts transitively on $\Omega=\Z_m$. Let
$A=\Z _2$, $S=\{0,1,\dots ,r\}$, and
$(g_s)_{s\in S}=(0,-a_1,-a_2,\dots,-a_r)\in G^S$.
Define $\mu:A^S\to A$ by 
\[
	\mu(f_0,f_1,\dots,f_r)=\begin{cases}
		1 & \text{if $f_0=1$ or $f_1=f_2=\cdots=f_r=1$,}\\
		0 & \text{otherwise.}
	\end{cases}
\]
The map $\tau:\Z_2^{\Z_m}\to \Z_2^{\Z_m}$ defined by $\mu $ and $(g_s)_{s\in
S}$
is then a monotonic cellular automaton. By definition,
\begin{align} \label{eq:rule_ZN}
  \supp \tau (f)=\supp f\cup \{w\in \Omega \,|\,
  f(x-a_1)=\cdots =f(x-a_r)=1\}
\end{align}
for all $f\in \Z_2^{\Z _m}$.
We study now plagues for special cases of this cellular automaton.

\begin{exa}
	\label{exa:9A}
	Let $r=1$ and $a_1=\lambda $. The cellular automaton is determined by
  the rule
  \begin{equation}
		\label{eq:rule9A}
    \supp \tau (f)=\supp f\cup \{x\in \Z_m \,|\,f(x-\lambda)=1\}.
	\end{equation}
	Let $\Gamma=\langle\lambda\rangle$ and let $I$ be a set of representatives for
	$\Omega /\Gamma$. Then $I$ is a plague.
\end{exa}

\begin{exa}
	\label{exa:6D}
	Let $\lambda\in\Z_m$, $r=3$, $a_1=1$, $a_2=\lambda +1$, and $a_3=-\lambda $.
	Let $\Gamma=\langle\lambda\rangle$ and let $I$ be the union of a set of
	representatives for $\Omega /\Gamma$ with $\Gamma$. For example
	$I=\langle\lambda\rangle\cup\{1,2,\dots,\lambda-1\}$.
  If $\supp f$ contains a coset $a+\Gamma $, where $a+1\in I$,
  then $\supp f$ spreads to $a+1+\Gamma $.
  Thus $I$ is a plague.
\end{exa}

\begin{exa}
	\label{exa:game12C}
	Let $\lambda\in\Omega\setminus\{0,1\}$, $r=2$, $a_1=\lambda $,
  $a_2=\lambda -1$.
  Let $I=\{0,1,\dots,(m-1)/2\}$ if $m$ is odd and $\{0,1,\dots ,m/2-1\}$ if
  $m$ is even. Then $I$ is a plague of size $\le (m+1)/2$.
  It is in general not minimal, for example
  for $m\ge 3$, $\lambda =2$ the set $\{0,1\}$ is a plague.
\end{exa}

\section{Plagues on Hurwitz orbits}
\label{section:plagues}

Example~\ref{exa:B3} describes cellular automata over braid group orbits.
In order to prove Theorem~\ref{thm:percolation}, we have to find small plagues
for these automata.

By \cite[\S 1.4]{MR2891215}, the \emph{immunity} of a Hurwitz orbit $\Sigma $
  is the ratio $s/|\Sigma |$, where $s$ is the size of a smallest
  plague for $\Sigma $ with respect to the cellular automaton
  in Example~\ref{exa:B3}. We write $\imm
  (\Sigma )$ for the immunity of $\Sigma $.

  Let us formulate the cellular automaton in Example~\ref{exa:B3}
in terms of the generators
$\sigma _2^{-1}\sigma _1^{-1}$ and $\sigma _1\sigma _2\sigma _1$
of the group $\mathbb{B}_3$. Note that
$x=\sigma _2^{-1}\sigma _1^{-1}\langle \Delta \rangle $ and
$y=\sigma _1\sigma _2\sigma _1\langle \Delta \rangle $ in
$\mathbf{PSL}(2,\Z)$.
Let $f$ be a $\Z _2$-valued function on the braid group
orbit $\Omega $ and let $P=\supp f$.
To obtain the support of $\tau (f)$, proceed as
follows.

We use the notation regarding
$\Omega $ and its Schreier graph introduced in Remark~\ref{rem:covgraph}.
Let $p$ be a point in the braid group orbit quotient $\Omega /\langle \Delta
\rangle $.
Let $I$ be a subset of $\ZN $ and let
$p[I]=\{p[i]\,|\,i\in I\}$ be the corresponding subset of the fiber over $p$.
Consider the three neighboring subsets
$$(\sigma _2^{-1}\sigma _1^{-1})^{-1} \cdot p[I],\quad
\sigma _1\sigma _2\sigma _1\cdot p[I],\quad\text{and}\quad
\sigma _2^{-1}\sigma _1^{-1}\cdot p[I+1]$$
of $p[I]$. They are displayed in Figure~\ref{fig:percolation} and are denoted
by $x_1[I-c]$, $x_2[I+a]$ and $x_3[I+b+1]$, respectively, where,
for example, $I+a$ means the translation of $I$ by $a$,
i.e., $I+a=\{i+a\mid i\in I\}$. In this setting we call $p[I]$ a \emph{pivot}.

By Example~\ref{exa:B3}, if $P$ contains the subsets
$x_1[I-c]=\sigma_1\sigma_2\cdot p[I]$
and $\sigma _2\cdot x_1[I-c]=x_2[I+a]$, then $\supp \tau (f)$
contains
$$\sigma _1\cdot x_2[I+a]=\sigma _1\sigma _2\sigma _1\sigma _2\cdot p[I]
=\sigma _2^{-1}\sigma _1^{-1}\Delta \cdot p[I]=x_3[I+b+1].$$
Similarly, if any two of the neighboring subsets
$x_1[I-c]$, $x_2[I+a]$, and $x_3[I+b+1]$ of $p[I]$
are contained in $P$,
then the third is a subset of $\supp \tau (f)$. Moreover, $\supp \tau (f)$
is the smallest subset of $\Omega $ containing $\supp f$ and all sets
constructed this way for some point $p$ and some subset $I\subseteq \ZN$.

\begin{figure}[h]
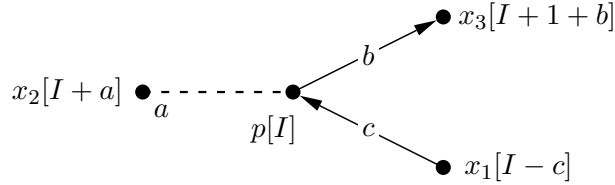

  \begin{graph}(6,2)
    \roundnode{n1}(1,1)
    \roundnode{n2}(3,1)
    \roundnode{n3}(5,2)
    \roundnode{n4}(5,0)
    \edge{n1}{n2}[\yedge ]
    \diredge{n2}{n3}
    \diredge{n4}{n2}
		\edgetext{n2}{n3}{$b$}
		\edgetext{n4}{n2}{$c$}
		\nodetext{n1}(-1,0){$x_2[I+a]$}
		\nodetext{n2}(-.25,-.50){$p[I]$}
		\nodetext{n3}(1.25,0){$x_3[I+1+b]$}
		\nodetext{n4}(1,0){$x_1[I-c]$}
		\nodetext{n1}(.25,-.25){$a$}
  \end{graph}
	\caption{The cellular automaton on braid group orbits}
  \label{fig:percolation}
\end{figure}

We intend to obtain an upper bound for the immunity of each Hurwitz orbit
using the cycle structure at each vertex. For that purpose, we define  
\begin{equation}
  (\omega'_{ij})_{i,j\ge 1}=
	\left(\begin{array}{ccccccc}
		1   & 1/3 &  11/24 &  1/2 & 1/2 & \cdots \\
		1/3 & 1/3 &  1/3 &  1/3 & 1/3 &  \cdots \\
		11/24 & 1/3 &  7/24 &  7/24 & 7/24 & \cdots \\
		1/2 & 1/3 & 7/24 & 1/4 & 1/4 & \cdots \\
		1/2 & 1/3 & 7/24 & 1/4 & 1/4 & \cdots \\
		\vdots & \vdots & \vdots & \vdots & \vdots & \ddots
	\end{array}\right)
	\label{eq:weights}
\end{equation}

The notation for $\mathbf{PSL}(2,\Z)$-spaces was fixed below
Proposition~\ref{pro:homogeneous_spaces}.
For the coverings of a
$\mathbf{PSL}(2,\Z)$-space $\Sigma _*$, where $*$ stands for any possible
index,
we use the notation $\Sigma _*^{N;a,b,\dots}$. Here,
$N$ is the size of the fiber over a point of $\Sigma _*$,
and $a,b,\dots $ are the values of the individual labels determining
the covering. These labels are chosen by the method described in
Remark~\ref{rem:covgraph} and can be read off from the corresponding figure.
We again use the same notation for the homogeneous space and its Schreier
graph.

Let $\Sigma$ be a homogeneous $\mathbb{B}_3$-space
and let $v\in\Sigma$. Assume that $v$ belongs
to a $\sigma_1$-cycle of length $i$, and also to a $\sigma_2$-cycle of length
$j$. Then we write $c(v)=(i,j)$.  Let
$\omega:\Sigma\to\mathbb{Q}$ be the map defined by 
\begin{equation} \label{eq:omega}
	\omega(v)=\begin{cases}
		\omega'_{ij}+\frac1{30}=\frac{13}{40}&\text{if $\Sigma=\Sigma^{5;3,2}_{4A}$ and $v\in v_1[*]$,}\\
		\omega'_{ij}+\frac1{12}=\frac13&\text{if $\Sigma=\Sigma^{4;2,2}_{6A}$ and $v\in v_3[*]$,}\\
		\omega'_{ij}+\frac1{24}=\frac13&\text{if $\Sigma$ is the trivial covering of $\Sigma_{12C}$,}\\
		\omega'_{ij}&\text{otherwise,}
	\end{cases}
\end{equation}
where $c(v)=(i,j)$, and $v_i[*]$ is defined in Remark~\ref{rem:covgraph}.

\begin{rem}
  If we would not add the three exceptional coverings in the definition of
  $\omega $, then Theorem~\ref{thm:percolation} below would not hold. We could
  compensate this by changing the infinite matrix $(\omega '_{ij})_{i,j\ge
  1}$, but then we would have only a weak upper bound for the immunity of a
  Hurwitz orbit. In that case we would not have enough information to reduce
  the proof of Theorem~\ref{thm:main} to the study of few small racks.
\end{rem}

\begin{defn}
  Let $\Sigma$ be a homogeneous $\mathbb{B}_3$-space.
The \emph{weight} of $\Sigma$ is defined as
\begin{equation}
	\label{eq:weight_of_sigma}
	\omega(\Sigma)=\frac1{|\Sigma|}\sum_{v\in\Sigma}\omega(v).
\end{equation}
\end{defn}

The main goal of this section is to formulate a good upper bound for the immunity
of Hurwitz orbits.

\begin{thm}
	\label{thm:percolation}
	Let $\Sigma$ be a covering with simply intersecting cycles of a finite
	homogeneous $\mathbf{PSL}(2,\Z)$-space $\overline{\Sigma}$. Assume that any
	$xy$-cycle of $\overline{\Sigma}$ has at most four elements. Then
	$\imm(\Sigma)\leq\omega(\Sigma)$.
\end{thm}

We prove Theorem \ref{thm:percolation} in Section \ref{section:proof}.



\section{Proof of Theorem \ref{thm:percolation}}

This section contains a case-by-case analysis of the coverings, immunities and
weights for the Schreier graphs of finite homogeneous
$\mathbf{PSL}(2,\Z)$-spaces $\overline\Sigma$ such that any $xy$-cycle of
$\overline\Sigma$ has at most four elements.
The main goal of this section is to prove 
Theorem~\ref{thm:percolation}.

\label{section:immunities}
\label{section:proof}

\subsection{The graph $\Sigma_{1A}$}
\label{1A}

\begin{lem}
	\label{lem:Hoq1} 
	Every covering of $\Sigma_{1A}$ in Figure \ref{fig:Hoq1} with simply
	intersecting cycles is trivial.
\end{lem}

\begin{proof}
	Lemma \ref{lem:loops} implies that $3a\equiv-1\pmod{N}$ and
	$2b\equiv1\pmod{N}$. By Lemma \ref{lem:xy_and_yx_cycles_with_loops},
	$a+b\equiv0\pmod{N}$ and hence $N=1$. 
\end{proof}

\begin{figure}[h]
  \begin{graph}(2,2)
    \roundnode{n1}(1,1)
    \dirloopedge{n1}(-.5,-.5)(-.5,.5)
    \loopedge{n1}(.5,.5)(.5,-.5)[\yedge ]
		\freetext(-.15,1){$a$}
		\freetext(2.15,1){$b$}
  \end{graph}
	\caption{Schreier graph $\Sigma _{1A}$ and its coverings}
  \label{fig:Hoq1}
\end{figure}

The following lemma is trivial, but we state it for completeness.

\begin{lem} \label{lem:1A_imm}
  Let $\Sigma $ be the trivial covering of
  the homogeneous $\mathbf{PSL}(2,\Z)$-space $\Sigma _{1A}$.
  Then $\imm(\Sigma)=1=\wg(\Sigma)$.
\end{lem}

\subsection{The graph $\Sigma_{2A}$}
\label{2A}
\begin{lem}
	\label{lem:Hoq2}
	The graph $\Sigma_{2A}$ in Figure \ref{fig:Hoq2} has no coverings with simply
	intersecting cycles.
\end{lem}

\begin{proof}
	Lemma \ref{lem:loops} implies that $3a\equiv3b\equiv-1\pmod{N}$.
	Further, from Lemma \ref{lem:xy_and_yx_cycles_with_loops}
        on $v$ we obtain
	$a+b+1\equiv0\pmod{N}$ and the claim follows. 
\end{proof}

\begin{figure}[h]
  \begin{graph}(4,2)
    \roundnode{n1}(1,1)
		\nodetext{n1}(.25,.25){$v$}
		\nodetext{n1}(.25,-.25){$1$}
		\nodetext{n2}(-.25,.25){$0$}
    \roundnode{n2}(3,1)
    \dirloopedge{n1}(-.5,-.5)(-.5,.5)
    \edge{n1}{n2}[\yedge ]
    \dirloopedge{n2}(.5,.5)(.5,-.5)
		\freetext(-.15,1){$a$}
		\freetext(4.15,1){$b$}
  \end{graph}
	\caption{Schreier graph $\Sigma_{2A}$ and its coverings}
  \label{fig:Hoq2}
\end{figure}

\subsection{The graph $\Sigma_{3A}$}
\label{3A}

\begin{lem}
\label{lem:Hoq3}
	Any covering of $\Sigma_{3A}$ in Figure \ref{fig:Hoq3} with simply intersecting
	cycles satisfies $2a\equiv1\pmod{N}$, $a\equiv b\pmod{N}$.
\end{lem}

\begin{proof}
	Lemma \ref{lem:loops} on
	$v_1$ implies that $2a\equiv1\pmod{N}$.
  Since $v_1$ belongs to cycles with label $a-b$,
	$a\equiv b\pmod{N}$
  by Lemma \ref{lem:xy_and_yx_cycles_with_loops}.
\end{proof}

\begin{exa}
	The trivial covering of $\Sigma_{3A}$ in Figure \ref{fig:Hoq3} is isomorphic to
	the Hurwitz orbit with three elements of \cite[Figure 8]{MR2891215}.
\end{exa}

\begin{figure}[h]
  \begin{graph}(4,4)
    \roundnode{n1}(1,2)
    \loopedge{n1}(-.5,-.5)(-.5,.5)[\yedge]
    \roundnode{n2}(3,3)
    \roundnode{n3}(3,1)
    \diredge{n1}{n2}
    \diredge{n2}{n3}
    \diredge{n3}{n1}
		\edgetext{n1}{n2}{$-1$}
    \bow{n2}{n3}{.2}[\yedge]
		\nodetext{n1}(0,.5){$v_1$}
		\nodetext{n2}(-.30,.25){$v_2$}
		\nodetext{n3}(-.30,-.25){$v_3$}
		\freetext(-.15,2){$a$}
		\freetext(3.35,3){$b$}
		\freetext(3.75,1){$1-b$}
  \end{graph}
	\caption{Schreier graph $\Sigma _{3A}$ and its coverings}
  \label{fig:Hoq3}
\end{figure}

\begin{lem}
	\label{lem:3A_imm}
	Let $\Sigma$ be a covering of $\Sigma_{3A}$ with simply intersecting cycles.
	Then $\imm(\Sigma)\leq1/3=\wg(\Sigma)$.
\end{lem}

\begin{proof}
	First we prove that $P=v_2[*]$ is a plague.
  With $v_3[*]$ as a pivot we see that $P$ spreads to $v_1[*]$.
  Then, with $v_1[*]$ as a pivot, it also spreads to
  $v_3[*]$ and hence $\imm(\Sigma)\leq 1/3$. 
	Now we prove $\wg(\Sigma)=1/3$.  By Lemma \ref{lem:Hoq3}, the cycle structure
	on each vertex of $\Sigma$ is the following: 
	\begin{align*}
		v_1[i]&:\text{ two $2$-cycles,}\\
		v_2[i]\text{ and }v_3[i]&:\text{ one $2$-cycle and a cycle of length
    $|\langle b\rangle|$,}
	\end{align*}
	for all $i\in\ZN$. Then $\wg(\Sigma)=1/3$ and this proves the claim.
\end{proof}

\subsection{The graph $\Sigma_{3B}$}
\label{3B}

\begin{lem}
	\label{lem:Hoq4}
	The graph $\Sigma_{3B}$ in Figure \ref{fig:Hoq4} has no coverings with simply
	intersecting cycles.
\end{lem}

\begin{proof}
  This follows from Corollary~\ref{cor:two_y_loops}.
\end{proof}

\begin{figure}[h]
  \begin{graph}(4,4)
    \roundnode{n1}(1,2)
    \roundnode{n2}(3,3)
    \roundnode{n3}(3,1)
    \diredge{n1}{n2}
    \diredge{n2}{n3}
    \diredge{n3}{n1}
    \loopedge{n1}(-.5,-.5)(-.5,.5)[\yedge]
    \loopedge{n2}(0,.5)(.5,0)[\yedge]
    \loopedge{n3}(.5,0)(0,-.5)[\yedge]
		\edgetext{n1}{n2}{$-1$}
		\nodetext{n1}(0,.5){$v_1$}
		\nodetext{n2}(-.30,.25){$v_2$}
		\nodetext{n3}(-.30,-.25){$v_3$}
		\freetext(-.15,2){$a$}
		\freetext(3.75,3.5){$b$}
		\freetext(3.75,.5){$c$}
  \end{graph}
	\caption{Schreier graph $\Sigma_{3B}$ and its coverings}
  \label{fig:Hoq4}
\end{figure}

\subsection{The graph $\Sigma_{4A}$}
\label{4A}

\begin{figure}[h]
  \begin{graph}(6,4)
    \roundnode{n1}(1,2)
    \roundnode{n2}(3,2)
    \roundnode{n3}(5,3)
    \roundnode{n4}(5,1)
    \dirloopedge{n1}(-.5,-.5)(-.5,.5)
    \edge{n1}{n2}[\yedge ]
    \diredge{n2}{n3}
    \diredge{n3}{n4}
    \diredge{n4}{n2}
		\edgetext{n2}{n3}{$-1$}
    \bow{n3}{n4}{.2}[\yedge]
		\freetext(-.15,2){$a$}
		\nodetext{n1}(.25,.30){$v_1$}
		\nodetext{n2}(-.25,-.30){$v_2$}
		\nodetext{n3}(-.25,.30){$v_3$}
		\nodetext{n4}(-.25,-.30){$v_4$}
		\nodetext{n1}(.25,-.25){$0$}
		\nodetext{n2}(-.25,.25){$1$}
		\freetext(6,3){$1-b$}
		\freetext(5.5,1){$b$}
  \end{graph}
	\caption{Schreier graph $\Sigma_{4A}$ and its coverings}
  \label{fig:Hoq5}
\end{figure}

\begin{lem}
	\label{lem:Hoq5}
	Any covering of $\Sigma_{4A}$ in Figure \ref{fig:Hoq5} with simply intersecting
	cycles satisfies $3a\equiv-1\pmod{N}$, $a+b\equiv0\pmod{N}$ and $N>1$.
\end{lem}

\begin{proof}
  Since $v_2\in C_{xy}(v_1)\cap C_{yx}(v_1)$,
	Lemma \ref{lem:trivial_covering} implies that $N>1$.
  Since $v_1$ is on an $xy$-cycle with label $a+b$, 
  by Lemmas \ref{lem:loops}(1) and
	\ref{lem:xy_and_yx_cycles_with_loops} on $v_1$ the claim follows.
\end{proof}

\begin{exa}
	Let $\Sigma$ be the covering of $\Sigma_{4A}$ with $N=2$ and $a=b=1$.  Then
	$\Sigma$ can be presented as 
	\begin{gather*}
		\sigma _1 = (v_1\;v_3\;\Delta v_2)(\Delta v_1\;\Delta v_3\;v_2), \quad
		\sigma _2 = (v_1\;v_2\;v_4)(\Delta v_1\;\Delta v_2\;\Delta v_4),
	\end{gather*}
	and $\Sigma$ is isomorphic to the Hurwitz orbit of eight elements of 
	\cite[Figure 10]{MR2891215}. The isomorphism is given by the
	bijection
	\[
	\left(\begin{array}{cccccccc}
		v_1 & \Delta v_1 & v_2 & \Delta v_2 & v_3 & \Delta v_3 & v_4 & \Delta v_4\\
		E & D & B & G & H & A & C & F
	\end{array}\right).
	\]
\end{exa}

\begin{lem}
	\label{lem:Hoq5_imm}
	Let $\Sigma$ be a covering of $\Sigma_{4A}$ with simply intersecting cycles.
	Then $\imm(\Sigma)\leq\frac{N+1}{4N}$. 
\end{lem}

\begin{proof}
	Let $I=\{0\}\subseteq \ZN $. We claim that $P=v_3[*]\cup
	v_1[I]$ is a plague. With $v_4[*]$ as a pivot, $P$ spreads to
  $v_2[*]$.
	Since the neighbors of $v_1[\lambda]$ are
	$v_1[\lambda-a]$, $v_1[\lambda+a+1]$ and
	$v_2[\lambda+1]$, pivoting from $v_1[I+a]$ implies that
  $P$ spreads to $v_1[I+2a+1]$.
  We describe this situation with the following table:
	\begin{center}
	\begin{tabular}{c|cc}
		pivot & $v_{4}[*]$ & $v_{1}[I+a]$\tabularnewline
		\hline
		 & $v_{2}[*]$ & $v_{1}[I+2a+1]$\tabularnewline
	 \end{tabular}
	 \end{center}
	 From Lemma \ref{lem:Hoq5} we obtain that
         $\langle 2a+1\rangle=\langle-a\rangle=\ZN$ and hence
         Example~\ref{exa:9A} implies
	 that $P$ spreads to $v_1[*]$. 
	 Now with $v_2[*]$ as a pivot we see that $P$ spreads to
         $\Sigma $ and we are done.
	 Thus the claim follows. 
\end{proof}

\begin{lem}
	\label{lem:4A_imm}
	Let $\Sigma$ be a covering of $\Sigma_{4A}$ with simply intersecting cycles.
	Then $\imm(\Sigma)\leq\omega(\Sigma)$.
\end{lem}

\begin{proof}
	By Lemma \ref{lem:Hoq5}, the cycle structure on each vertex of $\Sigma$ is
	the following: 
	\begin{align}
		\label{eq:Hoq5_v1v2}
		v_1[i]\text{ and }v_2[i]&:\text{ two $3$-cycles,}\\
		\label{eq:Hoq5_v3v4}
		v_3[i]\text{ and }v_4[i]&:\text{ one $3$-cycle and a cycle of length $|\langle1-b\rangle|$,}
	\end{align}
	for all $i\in\ZN$. Assume first that $N=3k+1$. Then $a=k$, $b=2k+1$ and $k\geq1$.  
	Since $3(b-1)\equiv -2\pmod{N}$ and $b-1$ is even, we get
	\begin{equation*}
	|\langle1-b\rangle|=|\langle 2\rangle |=
	\begin{cases}
		N & \text{ if $k$ is even, }\\
		\frac{N}2 & \text{ if $k$ is odd.}
		\end{cases}
	\end{equation*}
	By \eqref{eq:Hoq5_v3v4}, $\omega(v_3[i])=\omega(v_4[i])=1/3$	if $k=1$, 
	and $\omega(v_3[i])=\omega(v_4[i])=7/24$ if $k>1$, for all $i\in\ZN$. 
	Therefore 
	\begin{equation*}
		\wg(\Sigma)=
		\begin{cases}
			 5/16 & \text{ if $k=1$, }\\
			 7/24 & \text{ if $k>1$.}
		\end{cases}
	\end{equation*}
	If $k=1$ then $N=4$ and $\imm(\Sigma)\leq\frac{N+1}{4N}=5/16=\wg(\Sigma)$.	
	Otherwise, if $k>1$, then $N>6$ and hence  
	\[
	\imm(\Sigma)\leq\frac{N+1}{4N}<\frac7{24}=\wg(\Sigma).
	\]

	Now assume that $N=3k+2$. Then $a=2k+1$, $b=k+1$ and $k\geq0$. Further,
  $3(b-1)\equiv -2\pmod{N}$ and hence
	\begin{equation*}
	|\langle1-b\rangle|=
	\begin{cases}
		\frac{N}2 & \text{ if $k$ is even, }\\
		N & \text{ if $k$ is odd.}
		\end{cases}
	\end{equation*}
	By \eqref{eq:Hoq5_v3v4}, 
	\begin{equation*}
		\omega(v_3[i])=
	\begin{cases}
		11/24 & \text{ if $k=0$, }\\
		7/24 & \text{ if $k>0$,}
		\end{cases}
	\end{equation*}
	for all $i\in\ZN$, and therefore 
	\begin{equation*}
		\wg(\Sigma)=
		\begin{cases}
			3/8& \text{ if $k=0$, }\\
			7/24 & \text{ if $k>0$.}
		\end{cases}
	\end{equation*}
	If $k=0$ then $N=2$ and $\imm(\Sigma)\leq\frac{N+1}{4N}=3/8=\wg(\Sigma)$.
        If $k=1$ then $N=5$,
        $$\imm(\Sigma)\le \frac{N+1}{4N}=3/10=(4\cdot 7/24+1/30)/4=\wg(\Sigma). $$
	Finally, if $k>1$, then $\imm(\Sigma)\leq\frac{N+1}{4N}\leq7/24=\wg(\Sigma)$
	for $N\geq6$.  
\end{proof}

\begin{exa}
	Let $x_1=(1\,2\,3\,4\,5)$, $x_2=(1\,2\,4\,5\,3)$ and $x_3=(1\,4\,3\,5\,2)$ be
	$5$-cycles in $\Alt_5$.  Let $X$ be the rack associated to the conjugacy
	class of $x_1$ in $\Alt_5$.  The covering $\Sigma_{4A}^{5;3,2}$ can be
	realized as the Hurwitz orbit of $(x_1,x_2,x_3)\in X^3$.
\end{exa}

\begin{lem}
	\label{lem:exception_4A}
	Let $X$ be an injective indecomposable rack. Assume that $X^3$ contains a
	Hurwitz orbit isomorphic to the covering $\Sigma_{4A}^{5;3,2}$.  If
	$(x,y,z)\in v_1[*]$ then $(x\triangleright y)\triangleright y=z$.
\end{lem}

\begin{proof}
	By inspection,
	\begin{align}
		\sigma _1 v_{1}[i]&=v_{3}[i],\quad & \sigma _2 v_{1}[i]&=v_{2}[i+4],\label{eq:v1_4A}\\
		\sigma _1 v_{2}[i]&=v_{1}[i+3],\quad & \sigma _2 v_{2}[i]&=v_{4}[i+1],\label{eq:v2_4A}\\
		\sigma _1 v_{3}[i]&=v_{2}[i+2],\quad & \sigma _2 v_{3}[i]&=v_{3}[i+4],\label{eq:v3_4A}\\
		\sigma _1 v_{4}[i]&=v_{4}[i+4],\quad & \sigma _2 v_{4}[i]&=v_{1}[i].\label{eq:v4_4A}
	\end{align}
	Let $1,2,3\in X$ with $v_{1}[0]=(1,2,3)\in \Sigma _{4A}^{5;3,2}\subseteq X^{3}$.
	Set $1\triangleright2=4$, $1\triangleright3=5$ and
	$1\triangleright4=6$. (The elements $1,2,3,4,5,6\in X$ are not necessarily pairwise distinct.)
        We claim that $(1\triangleright2)\triangleright1=2$.
	Indeed, using \eqref{eq:v1_4A} we obtain
	$v_{3}[0]=\sigma _1 v_{1}[0]=(1\triangleright2,1,3)$.  By \eqref{eq:v3_4A},
	$v_{2}[2]=\sigma _1v_{3}[0]=((1\triangleright2)\triangleright1,1\triangleright2,3)$.  Since
	$v_{1}[0]=\sigma _1v_{2}[2]$, we obtain $(1\triangleright2)\triangleright1=2$.
	Similarly, using formulas \eqref{eq:v1_4A}--\eqref{eq:v4_4A}, straightforward
	computations show that: 
  \begin{center}
	\begin{tabular}{c|c|c|c|c}
		$i$ & $v_1[i]$ & $v_{2}[i]$ & $v_{3}[i]$ & $v_{4}[i]$\tabularnewline
		\hline
		$0$ & $(1,2,3)$ & $(3,4,5)$ & $(4,1,3)$ & $(1,3,4)$\tabularnewline
		$1$ & $(3,5,6)$ & $(6,4,1)$ & $(4,3,6)$ & $(3,6,4)$\tabularnewline
		$2$ & $(6,1,2)$ & $(2,4,3)$ & $(4,6,2)$ & $(6,2,4)$\tabularnewline
		$3$ & $(2,3,5)$ & $(5,4,6)$ & $(4,2,5)$ & $(2,5,4)$\tabularnewline
		$4$ & $(5,6,1)$ & $(1,4,2)$ & $(4,5,1)$ & $(5,1,4)$\tabularnewline
	\end{tabular}
	\end{center}
        Since $\sigma _2\cdot (1,4,2)=\sigma _2 v_2[4]=v_4[0]=(1,3,4)$, we get
        $(1\trid 2)\trid 2=4\trid 2=3$. Since $\Delta $ acts transitively on $v_1[*]$,
        the lemma follows.
\end{proof}

\subsection{The graph $\Sigma_{4B}$}
\label{4B}

\begin{lem}
	\label{lem:Hoq6}
	The graph $\Sigma_{4B}$ in Figure \ref{fig:Hoq6} has no coverings with simply
	intersecting cycles.
\end{lem}

\begin{proof}
  This follows from Corollary~\ref{cor:two_y_loops}.
\end{proof}

\begin{figure}[h]
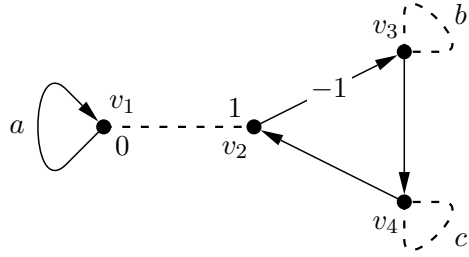

  \begin{graph}(6,4)
    \roundnode{n1}(1,2)
    \roundnode{n2}(3,2)
    \roundnode{n3}(5,3)
    \roundnode{n4}(5,1)
    \dirloopedge{n1}(-.5,-.5)(-.5,.5)
    \edge{n1}{n2}[\yedge ]
    \diredge{n2}{n3}
    \diredge{n3}{n4}
    \diredge{n4}{n2}
		\edgetext{n2}{n3}{$-1$}
    \loopedge{n3}(0,.5)(.5,0)[\yedge]
    \loopedge{n4}(0,-.5)(.5,0)[\yedge]
		\freetext(-.15,2){$a$}
		\freetext(5.75,3.5){$b$}
		\freetext(5.75,.5){$c$}

		\nodetext{n1}(.25,.30){$v_1$}
		\nodetext{n2}(-.25,-.30){$v_2$}
		\nodetext{n3}(-.25,.30){$v_3$}
		\nodetext{n4}(-.25,-.30){$v_4$}
		\nodetext{n1}(.25,-.25){$0$}
		\nodetext{n2}(-.25,.25){$1$}
  \end{graph}
	\caption{Schreier graph $\Sigma _{4B}$ and its coverings}
  \label{fig:Hoq6}
\end{figure}

\subsection{The graph $\Sigma_{6A}$}
\label{6A}

\begin{lem}
	\label{lem:Hoq7}
	Any covering of $\Sigma_{6A}$ in Figure \ref{fig:Hoq7} with simply intersecting
	cycles satisfies $N\ne1$, $a+b\equiv0\pmod{N}$ and
	$2b\not\equiv1\pmod{N}$.
\end{lem}

\begin{proof}
	The $xy$-cycles and their labels are: $(v_1\;v_3\;v_5\;v_4)$ with $a+b$,
	$(v_2)$ with $-a$ and $(v_6)$ with $-b+1$. The $yx$-cycles are:
	$(v_2\;v_4\;v_6\;v_3)$ with $a+b$, $(v_5)$ with $-b+1$ and $(v_1)$ with $-a$.
	Lemma \ref{lem:xy_and_yx_cycles} on $v_3$ implies $a+b\equiv 0\pmod{N}$.  Since
	$v_3$ and $v_4$ are on the same cycle, $N>1$ by Lemma
	\ref{lem:trivial_covering} and $a+1\not\equiv b\pmod{N}$ by Lemma
	\ref{lem:from_v_to_w}.
\end{proof}

\begin{figure}[h]
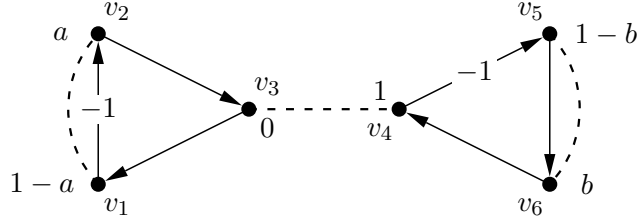

  \begin{graph}(8,4)
    \roundnode{n1}(1,1)
    \roundnode{n2}(1,3)
    \roundnode{n3}(3,2)
    \roundnode{n4}(5,2)
    \roundnode{n5}(7,3)
    \roundnode{n6}(7,1)
    \diredge{n1}{n2}
    \diredge{n2}{n3}
    \diredge{n3}{n1}
    \diredge{n4}{n5}
    \diredge{n5}{n6}
    \diredge{n6}{n4}
    \bow{n1}{n2}{.2}[\yedge]
    \bow{n5}{n6}{.2}[\yedge]
    \edge{n3}{n4}[\yedge]
		\freetext(0.5,3){$a$}
		\freetext(0.25,1){$1-a$}
		\freetext(7.75,3){$1-b$}
		\freetext(7.5,1){$b$}
		\edgetext{n1}{n2}{$-1$}
		\edgetext{n4}{n5}{$-1$}
		\nodetext{n1}(.25,-.30){$v_1$}
		\nodetext{n2}(.25,.30){$v_2$}
		\nodetext{n3}(.25,.30){$v_3$}
		\nodetext{n4}(-.25,-.30){$v_4$}
		\nodetext{n5}(-.25,.30){$v_5$}
		\nodetext{n6}(-.25,-.30){$v_6$}
	  \nodetext{n3}(.25,-.25){$0$}
		\nodetext{n4}(-.25,.25){$1$}
  \end{graph}
	\caption{Schreier graph $\Sigma _{6A}$ and its coverings}
  \label{fig:Hoq7}
\end{figure}

\begin{lem}
	\label{lem:Hoq7_imm}
	Let $\Sigma$ be a covering of $\Sigma_{6A}$ with simply intersecting cycles,
  and let 
	\[
		m=\min\left\{\frac{|\langle1+a\rangle|-1}{6N},\frac1{6|\langle-a\rangle|}\right\}.
	\]
	Then
	\[
	\imm(\Sigma)\leq\begin{cases}
		1/3& \text{if $a\in\{0,-1\}$, }\\
		\frac1{6|\langle a+1\rangle|}+\frac16+m & \text{otherwise. }
	\end{cases}
	\]
\end{lem}

\begin{proof}
        Let $\Sigma =\Sigma _{6A}^{N;a,b}$ as in Figure~\ref{fig:Hoq7}.
        There exists a unique isomorphism from $\Sigma $ to $\Sigma _{6A}^{N;b-1,1+a}$ with
	\begin{align*}
		 &v_1\mapsto v_5,&&v_3\mapsto v_4[1],&&v_5\mapsto v_1[1],\\ 
		 &v_2\mapsto v_6[1],&&v_4\mapsto v_3,&&v_6\mapsto v_2.
	\end{align*}
	The labels $-a$ and $1-b$ are permuted under this isomorphism.
        Hence without loss of generality
	we may assume $|\langle-a\rangle|\leq|\langle1-b\rangle|$.  
	By Lemma \ref{lem:Hoq7} we have $b\equiv -a\pmod{N}$ and the cycle structure of $\Sigma$ is the following:
	\begin{align}
		\label{eq:Hoq7_v1v2}
		v_1[i]\text{ and }v_2[i]&:\text{ one $4$-cycle and a cycle of length $|\langle-a\rangle|$,}\\
		\label{eq:Hoq7_v3v4}
		v_3[i]\text{ and }v_4[i]&:\text{ two $4$-cycles,}\\
		\label{eq:Hoq7_v5v6}
		v_5[i]\text{ and }v_6[i]&:\text{ one $4$-cycle and a cycle of length $|\langle1+a\rangle|$}
	\end{align}
	for all $i\in\ZN$. 
	First assume that $a\in\{0,-1\}$.  A straightforward computation shows that
	$v_1[*]\cup v_5[*]$ is a plague of size $2N$. In this case,
	$\imm(\Sigma)\leq1/3$.

	 Now assume that $a\notin\{0,-1\}$.	Let $I$ (resp. $J$) be a set of
	 representatives for $\ZN/\langle -a\rangle$ (resp. $\ZN/\langle
	 a+1\rangle$). We claim that $P=v_1[*]\cup v_2[I]\cup v_5[J]$ is
	 a plague. First we compute
	 \begin{center}
	 \begin{tabular}{c|cc}
		 pivot & $v_{2}[*]$ & $v_{1}[I]$\tabularnewline
		 \hline
		  & $v_{3}[*]$ & $v_{2}[I+a]$\tabularnewline
		\end{tabular}
	\end{center}
   and hence $P$ spreads to
   $v_2[*]$ by Example \ref{exa:9A}.
   Now we compute
	 \begin{center}
	 \begin{tabular}{c|cc}
		 pivot & $v_{3}[*]$ & $v_{6}[J]$\tabularnewline
		 \hline
		  & $v_{4}[*]$ & $v_{5}[J+1+a]$\tabularnewline
		\end{tabular}
	\end{center}
	and hence $P$ spreads to
  $v_1[*]\cup v_2[*]\cup v_3[*]\cup v_4[*]\cup v_5[*]$ by Example \ref{exa:9A}.
  Then we conclude that $P$ is a plague.
	 
	 As an alternative, let $J$ be a set of representatives for
	 $\ZN/\langle1+a\rangle$, and let $K=\langle 1+a\rangle $ and $I=J\cup K$.
	 We prove that $P=v_1[*]\cup v_5[I]$ is a plague.
   With $v_2[*]$ as a pivot we see that $P$ spreads to
       	 $v_3[*]$, with $v_4[I]$ it spreads to $v_6[I]$,
         and with $v_5[(I-1)\cap(I+a)]$ it spreads further to
         $v_4[I\cap(I+a+1)]$.
         Finally, with $v_6[(I-1)\cap(I+a)\cap(I-a-1)]$ as a pivot,
         $P$ further spreads to
	 $v_5[(I-1)\cap(I+a)\cap(I-a-1)]$.
         By restricting our attention
         to the fiber over $v_5$, we conclude from Example~\ref{exa:6D}
         that $P$ spreads to $v_5[*]$.
         Since $v_1[*]\cup v_5[*]$ is a plague of $\Sigma $,
	 the claim of the lemma follows.
\end{proof}

\begin{lem}
	\label{lem:6A_imm}
	Let $\Sigma$ be a covering of $\Sigma_{6A}$ with simply intersecting cycles.
  Then $\imm(\Sigma)\leq\omega(\Sigma)$.
\end{lem}

\begin{proof}
  By Lemma~\ref{lem:Hoq7} we have $b\equiv -a\pmod{N}$ and $N>1$.
	As in the proof of Lemma \ref{lem:Hoq7_imm}, we may assume
	$|\langle-a\rangle|\leq|\langle1+a\rangle|$.  
	We split the proof into several cases according to the
	order of $\langle-a\rangle$.  Assume first that $|\langle-a\rangle|=1$. Then
	$a=0$ and $|\langle1+a\rangle|=N\geq2$.
  Hence
  $$\wg(\Sigma )\geq 2(1/2+1/4+1/4)/6=1/3=\imm(\Sigma )$$
  by Lemma~\ref{lem:Hoq7_imm}.

	Now assume that $|\langle-a\rangle|=2$. Then $a=N/2$, $N$ is even, and
	$N\geq4$. (If $N=2$ then $a=1$, which contradicts to
	$|\langle-a\rangle|\leq|\langle1+a\rangle|$.) As before, we need to consider
	three cases: $N=4$, $N=6$ and $N\geq8$. Then 
	\begin{center}
  	\begin{tabular}{c|c|c|c|c}
  		$N$ & $|\langle1+a\rangle|$ & $\omega(v_{1})$, $\omega(v_{2})$ &
      $\omega(v_{3})$, $\omega(v_{4})$ & $\omega(v_{5})$, $\omega(v_{6})$
      \tabularnewline
  		\hline 
  		$4$ & $4$ & $1/3$ & $1/4$ & $1/4$ \tabularnewline
  		$6$ & $3$ & $1/3$ & $1/4$ & $7/24$ \tabularnewline
  		$\geq8$ & $\geq 5$ & $1/3$ & $1/4$ & $1/4$ \tabularnewline
  	\end{tabular}
	\end{center}
	and therefore (observe that the case $N=4$ is exceptional)
	\[
	\wg(\Sigma)=\begin{cases}
		(3\cdot 1/3+3\cdot 1/4)/6=7/24 & \text{ if $N=4$, }\\
		7/24 & \text{ if $N=6$, }\\
		5/18 & \text{ if $N\geq 8$. }
	\end{cases}
	\]
	Further, $|\langle a+1\rangle |=N/2$ if $a$ is odd and $|\langle a+1\rangle
  |=N$ if $a$ is even. Thus Lemma~\ref{lem:Hoq7_imm}
  yields that $\imm (\Sigma )\le 7/24$ for $N=4$,
  $\imm (\Sigma )\le 1/4+1/6N$ for $N\ge 6$, both if $a$ is odd and if $a$ is
  even.
  This implies the claim.

	Now assume that $|\langle-a\rangle|=3$. Clearly, $3a\equiv0\pmod{N}$ and
	$3$ divides $N$. We claim that $4\leq|\langle1+a\rangle|$. Indeed, if
	$|\langle1+a\rangle|=3$ then $3(1+a)\equiv0\pmod{N}$ and $N=3$, $a=2$, $b=1$,
	by Lemma \ref{lem:Hoq7}. Hence $|\langle1+a\rangle|=0$, which is a
	contradiction. Therefore, by \eqref{eq:Hoq7_v1v2}--\eqref{eq:Hoq7_v5v6} and
	Lemma \ref{lem:Hoq7_imm}, we obtain $\imm(\Sigma)\leq19/72=\omega(\Sigma)$.

	Finally, assume that $|\langle-a\rangle|\geq4$. Using
	\eqref{eq:Hoq7_v1v2}--\eqref{eq:Hoq7_v5v6} and Lemma \ref{lem:Hoq7_imm} we
	obtain $\wg(\Sigma)=1/4\geq\imm(\Sigma)$.
\end{proof}



\begin{exa}
	Let $X=\Aff(5,2)$ or $X=\Aff(5,3)$. The Hurwitz orbit of $(1,1,2)\in X^3$ is
	isomorphic to the covering $\Sigma_{6A}^{4;2,2}$.
\end{exa}


\begin{lem}
	\label{lem:exception_6A}
	Let $X$ be an injective indecomposable rack, and $x\in X$. Assume that $X$ has
	at least one Hurwitz orbit isomorphic to the covering $\Sigma_{6A}^{4;2,2}$.
	Then 
	\begin{enumerate}
		\item $\varphi_x$ contains a $4$-cycle and $k_4\geq4$.
		\item If $(x,y,z)\in v_3[*]$ and $w=x\triangleright y$, then $w\triangleright( (w\triangleright  x)\triangleright x)=z$.
	\end{enumerate}
\end{lem}

\begin{proof}
	We may assume that $X=\{1,2,\dots,d\}$. By inspection,
	\begin{align}
		\sigma _1v_{1}[i]&=v_{3}[i+2],\quad & \sigma_2v_{1}[i]&=v_{1}[i+2],\label{eq:v1_6A}\\
		\sigma _1v_{2}[i]&=v_{2}[i+2],\quad & \sigma_2v_{2}[i]&=v_{4}[i+1],\label{eq:v2_6A}\\
		\sigma _1v_{3}[i]&=v_{5}[i],\quad & \sigma_2v_{3}[i]&=v_{2}[i+2],\label{eq:v3_6A}\\
		\sigma _1v_{4}[i]&=v_{1}[i],\quad & \sigma_2v_{4}[i]&=v_{6}[i+1],\label{eq:v4_6A}\\
		\sigma _1v_{5}[i]&=v_{4}[i+2],\quad & \sigma_2v_{5}[i]&=v_{5}[i+3],\label{eq:v5_6A}\\
		\sigma _1v_{6}[i]&=v_{6}[i+3],\quad & \sigma_2v_{6}[i]&=v_{3}[i].\label{eq:v6_6A}
	\end{align}
	Without loss of generality we may assume that $v_{1}[0]=(1,2,3)\in X^{3}$,
	where $|\{1,2,3\}|\leq3$.	Set $4=1\trid2$, $5=1\trid4$, $6=4\trid5$,
  $7=2\trid 4$, $8=7\trid 1$ and $9=\varphi _2^{-1}(8)$.
  Using formulas \eqref{eq:v1_6A}--\eqref{eq:v6_6A}, a
	straightforward computation shows that: 
  \begin{center}
	\begin{tabular}{c|c|c|c|c|c|c}
		$i$ & $v_1[i]$ & $v_{2}[i]$ & $v_{3}[i]$ & $v_{4}[i]$ & $v_5[i]$ & $v_6[i]$\tabularnewline
		\hline
		$0$ & $(1,2,3)$ & $(4,7,1)$ & $(7,1,2)$ & $(2,8,3)$ & $(8,7,2)$ & $(7,2,8)$\tabularnewline
		$1$ & $(9,4,7)$ & $(3,2,9)$ & $(2,9,4)$ & $(4,8,7)$ & $(8,2,4)$ & $(2,4,8)$\tabularnewline
		$2$ & $(1,3,2)$ & $(7,4,1)$ & $(4,1,3)$ & $(3,8,2)$ & $(8,4,3)$ & $(4,3,8)$\tabularnewline
		$3$ & $(9,7,4)$ & $(2,3,9)$ & $(3,9,7)$ & $(7,8,4)$ & $(8,3,7)$ & $(3,7,8)$\tabularnewline
	\end{tabular}
	\end{center}
	Since there are no loops, we obtain $|\{2,3,4,7\}|=4$. Furthermore, using
	\eqref{eq:v1_6A}--\eqref{eq:v6_6A} we obtain that
  $\varphi _8=(2\;7\;3\;4)\tau $ for some $\tau \in \Sym _d$,
  $\tau |_{\{2,3,4,7\}}=\id $.
  In particular, $|\{2,3,4,7,8\}|=5$. Moreover, $\lfloor 1\;3\;8\;7\rfloor$
  is a Hurwitz orbit in $X^2$, and acting with $\varphi _8$ on this we obtain
  three other Hurwitz orbits
  $\lfloor ?\;4\;8\;3\rfloor$,
  $\lfloor ?\;2\;8\;4\rfloor$,
  $\lfloor ?\;7\;8\;2\rfloor$. Hence $k_4\ge 4$. Finally, $3\trid 7=4$,
  $7\trid 1=8$,
  $8\trid ((8\trid 7)\trid 7)=8\trid (3\trid 7)=8\trid 4=2$, and hence
  $w\trid ( (w\trid x)\trid x)=z$ for $(x,y,z)=(7,1,2)$.
\end{proof}

\subsection{The graph $\Sigma_{6B}$}

\begin{lem}
	\label{lem:Hoq8}
	There is no covering of $\Sigma_{6B}$ in Figure \ref{fig:Hoq8} with simply
	intersecting cycles.
\end{lem}

\begin{proof}
	The $xy$-cycles and their labels are $(v_1\;v_2\;v_5)$ with $a-c$ and
	$(v_3\;v_6\;v_4)$ with $c+b$.  The $yx$-cycles and their labels are
  $(v_1\;v_4\;v_3)$ with
	$a-c$ and $(v_2\;v_5\;v_6)$ with $c+b$.  By Lemma \ref{lem:trivial_covering}
	on $v_3$ and $v_4$, $N>1$.  Lemma \ref{lem:loops} on $v_1$ and $v_6$ implies
	that $2a\equiv2b\equiv1\pmod{N}$ and hence $a\equiv b\pmod{N}$ and $N$ is
  odd.  By Lemma
	\ref{lem:xy_and_yx_cycles_with_loops} on $v_1$ and $v_6$, $b+c\equiv
	0\pmod{N}$ and $a-c\equiv0\pmod{N}$. Then $a\equiv -b\pmod{N}$ and
  hence $1\equiv -1\pmod{N}$. This is a contradiction to $N>2$.
\end{proof}

\begin{figure}[h]
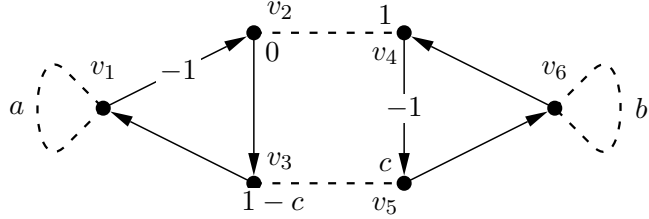

  \begin{graph}(8,4)
    \roundnode{n1}(1,2)
    \roundnode{n2}(3,3)
    \roundnode{n3}(3,1)
    \roundnode{n4}(5,3)
    \roundnode{n5}(5,1)
    \roundnode{n6}(7,2)
    \diredge{n1}{n2}
    \diredge{n2}{n3}
    \diredge{n3}{n1}
    \diredge{n4}{n5}
    \diredge{n5}{n6}
    \diredge{n6}{n4}
    \loopedge{n1}(-.5,-.5)(-.5,.5)[\yedge]
    \edge{n2}{n4}[\yedge]
    \edge{n3}{n5}[\yedge]
    \loopedge{n6}(.5,.5)(.5,-.5)[\yedge]
		\freetext(-.15,2){$a$}
		\freetext(8.15,2){$b$}
		\nodetext{n1}(0,.5){$v_1$}
		\nodetext{n6}(0,.5){$v_6$}
		\nodetext{n2}(.35,.30){$v_2$}
		\nodetext{n2}(.25,-.25){$0$}
		\nodetext{n4}(-.25,-.30){$v_4$}
		\nodetext{n4}(-.25,.25){$1$}
		\nodetext{n3}(.35,.30){$v_3$}
		\nodetext{n3}(.25,-.25){$1-c$}
		\nodetext{n5}(-.25,-.30){$v_5$}
		\nodetext{n5}(-.25,.25){$c$}
		\edgetext{n1}{n2}{$-1$}
		\edgetext{n4}{n5}{$-1$}
  \end{graph}
  \caption{Schreier graph $\Sigma _{6B}$ and its coverings}
  \label{fig:Hoq8}
\end {figure}

\subsection{The graph $\Sigma_{6C}$}
\label{6C}

\begin{lem}
	\label{lem:Hoq9}
	There is no covering of $\Sigma_{6C}$ in Figure \ref{fig:Hoq9} with simply
	intersecting cycles.
\end{lem}

\begin{proof}
  This follows from Corollary~\ref{cor:two_y_loops}.
\end{proof}

\begin{figure}[ht]
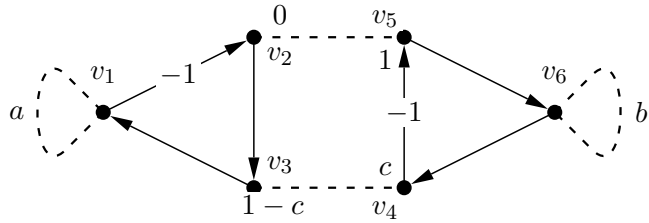

  \begin{graph}(8,4)
    \roundnode{n1}(1,2)
    \roundnode{n2}(3,3)
    \roundnode{n3}(3,1)
    \roundnode{n5}(5,3)
    \roundnode{n4}(5,1)
    \roundnode{n6}(7,2)
    \diredge{n1}{n2}
    \diredge{n2}{n3}
    \diredge{n3}{n1}
    \diredge{n4}{n5}
    \diredge{n5}{n6}
    \diredge{n6}{n4}
    \loopedge{n1}(-.5,-.5)(-.5,.5)[\yedge]
    \edge{n2}{n5}[\yedge]
    \edge{n3}{n4}[\yedge]
    \loopedge{n6}(.5,.5)(.5,-.5)[\yedge]
		\freetext(-.15,2){$a$}
		\freetext(8.15,2){$b$}
		\nodetext{n1}(0,.5){$v_1$}
		\nodetext{n6}(0,.5){$v_6$}
		\nodetext{n2}(.35,.30){$0$}
		\nodetext{n2}(.35,-.25){$v_2$}
		\nodetext{n4}(-.25,-.30){$v_4$}
		\nodetext{n4}(-.25,.25){$c$}
		\nodetext{n3}(.35,.30){$v_3$}
		\nodetext{n3}(.25,-.25){$1-c$}
		\nodetext{n5}(-.25,-.30){$1$}
		\nodetext{n5}(-.25,.25){$v_5$}
		\edgetext{n1}{n2}{$-1$}
		\edgetext{n4}{n5}{$-1$}
  \end{graph}
  \caption{Schreier graph $\Sigma _{6C}$ and its coverings}
  \label{fig:Hoq9}
\end{figure}

\subsection{The graph $\Sigma_{6D}$}
\label{6D}

\begin{lem}
	\label{lem:Hoq10}
	Let $\Sigma $ be a covering of $\Sigma_{6D}$ in Figure \ref{fig:Hoq10}
  with simply intersecting cycles. Then
	\begin{enumerate}
	  \item $\langle a+b\rangle\cap\langle -a+1\rangle=0$,
	  \item $\langle a+b\rangle\cap\langle -b\rangle=0$,
	  \item $\langle -a+1\rangle\cap\langle -b\rangle=0$.
		\end{enumerate}
\end{lem}

\begin{proof}
	The $xy$-cycles and their labels are: $(v_1\;v_4)$ with label $a+b$,
	$(v_2\;v_6)$ with $-a+1$ and $(v_3\;v_5)$ with $-b$. The
	$yx$-cycles are: $(v_1\;v_5)$ with label $-a+1$, $(v_2\;v_4)$ with $-b$ and
	$(v_3\;v_6)$ with $a+b$. Now apply Lemma \ref{lem:xy_and_yx_cycles}.
\end{proof}

\begin{exa}
	The trivial covering of $\Sigma_{6D}$ is isomorphic to the Hurwitz orbit
	of six elements of \cite[Figure 9]{MR2891215}.
\end{exa}

\begin{figure}[ht]
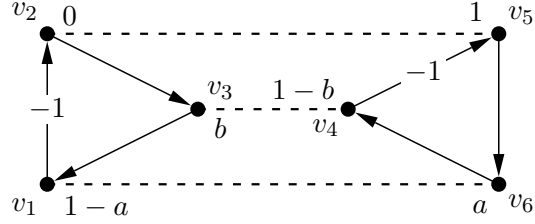

  \begin{graph}(8,4)
    \roundnode{n1}(1,1)
    \roundnode{n2}(1,3)
    \roundnode{n3}(3,2)
    \roundnode{n4}(5,2)
    \roundnode{n5}(7,3)
    \roundnode{n6}(7,1)
    \diredge{n1}{n2}
    \diredge{n2}{n3}
    \diredge{n3}{n1}
    \diredge{n4}{n5}
    \diredge{n5}{n6}
    \diredge{n6}{n4}
    \edge{n3}{n4}[\yedge]
    \edge{n1}{n6}[\yedge]
    \edge{n2}{n5}[\yedge]
		\edgetext{n1}{n2}{$-1$}
		\edgetext{n4}{n5}{$-1$}
		\nodetext{n1}(-.30,-.30){$v_1$}
		\nodetext{n2}(-.30,.30){$v_2$}
		\nodetext{n3}(.30,.25){$v_3$}
	  \nodetext{n4}(-.30,-.25){$v_4$}
		\nodetext{n5}(.30,.25){$v_5$}
		\nodetext{n6}(.30,-.25){$v_6$}
		\nodetext{n3}(.30,-.25){$b$}
		\nodetext{n4}(-.60,.25){$1-b$}
		\nodetext{n1}(.65,-.30){$1-a$}
		\nodetext{n2}(.30,.25){$0$}
		\nodetext{n5}(-.30,.30){$1$}
		\nodetext{n6}(-.25,-.30){$a$}
  \end{graph}
  \caption{Schreier graph $\Sigma _{6D}$ and its coverings}
  \label{fig:Hoq10}
\end{figure}

\begin{lem}
	\label{lem:Hoq10_imm}
	Let $\Sigma$ be a covering of $\Sigma_{6D}$ with simply intersecting cycles.
	Let $c\in \{1-a,-b,a+b\}$ and let $n=|\langle c\rangle|$. Then
	\[
		\imm(\Sigma)\leq \frac{N+N/n+n-1}{6N}.
	\]
\end{lem}

\begin{proof}
  There exists an isomorphism $\Sigma_{6D}^{N;a,b}\to \Sigma
  _{6D}^{N;1-a-b,a-1}$, $v_1\mapsto v_2$.
	This isomorphism induces the permutation
	$(1\;3\;2)$ on the triple $(1-a,-b,a+b)$.
        Therefore we may assume that $c\equiv -b \pmod N$.

	The cycle structure at each vertex of $\Sigma$ is the following:
	\begin{align}
		\label{eq:Hoq10_v1v2}v_1[i]\text{ and }v_6[i]: & \text{ cycles of length $2|\langle a+b\rangle|$ and $2|\langle -a+1\rangle|$,}\\
		\label{eq:Hoq10_v2v5}v_2[i]\text{ and }v_5[i]: & \text{ cycles of length $2|\langle -a+1\rangle|$ and $2|\langle -b\rangle|$,}\\
		\label{eq:Hoq10_v3v4}v_3[i]\text{ and }v_4[i]: & \text{ cycles of length $2|\langle a+b\rangle|$ and $2|\langle -b\rangle|$,}
	\end{align}
	for all $i\in\ZN$.
  Let $I$
	be the union of $\langle b\rangle$ with a set of representatives for
	$\ZN/\langle b\rangle$. Then $|I|=n+N/n-1$.
  We claim that $P=v_1[*]\cup v_5[I]$ is
	a plague. Using the sequence of pivots $v_{6}[I]$, $v_{3}[I+b]$,
	$v_{2}[I-1]$, $v_{5}[(I+b)\cap I]$, and $v_4[(I+1)\cap(I+b+1)\cap(I-b)]$, we
  see that $P$ spreads to
	$$ v_{4}[I+1]\cup v_{2}[I+b]\cup v_{3}[I]\cup
     v_{6}[(I+1)\cap(I+b+1)]\cup v_5[(I+1)\cap(I+b+1)\cap(I-b)].$$
  By looking at the evolution of $\chi _P$ at $v_5[*]$,
  Example \ref{exa:6D} implies that $P$ spreads to $v_5[*]$.
  Since $P\cup v_5[*]$ is a plague, also $P$ is a plague. This implies the
  lemma.
\end{proof}

\begin{lem}
	\label{lem:Hoq10_ineq}
	Let $n\in\N$ with $2\leq n\leq N/2$. Then 
	\[
		\frac{6N}{n}+6n-6\leq 3N+6. 
		\]
\end{lem}

\begin{proof}
	The inequality is equivalent to $(n-2)(2n-N)\leq 0$ and hence the claim
	follows.
\end{proof}

\begin{lem}
	\label{lem:6D_imm}
	Let $\Sigma$ be a covering of $\Sigma_{6D}$ with simply intersecting cycles.
	Then $\imm(\Sigma)\le \wg(\Sigma )$.
\end{lem}

\begin{proof}
        There is an isomorphism
        $\Sigma_{6D}^{N;a,b}\to \Sigma _{6D}^{N;b+1,a-1}$, $v_1\mapsto v_4$
        and this induces the permutation
	$(1\;2)$ on $(1-a,-b,a+b)$.
        Using this and the isomorphism in the proof of Lemma~\ref{lem:Hoq10_imm},
        we may assume that
        $$0\leq [1-a]\leq [-b]\leq [a+b]<N,$$
        where for any $x\in \ndZ $ we let $[x]\in \{0,1,\dots ,N-1\}$
        such that $[x]\equiv x\pmod{N}$.
	
	Suppose first that $a=1$ and $b=0$.  Then, using
	Lemma \ref{lem:Hoq10_imm} with $c=1-a$, we obtain $\imm(\Sigma)\le 1/3=\wg(\Sigma)$.

	Assume now that $a\equiv 1\pmod N$,
        $b\not\equiv 0\pmod{N}$ and $a+b\not\equiv0\pmod{N}$.
        Then $N\geq3$ by Lemma \ref{lem:Hoq10}.  Further, by \eqref{eq:Hoq10_v1v2}--\eqref{eq:Hoq10_v3v4},
	\[
		\wg(\Sigma)=\frac{(1/3)4N+(1/4)2N}{6N}=\frac{11}{36}.
	\]
        Let $n=|\langle -b\rangle |$. Then $2\le c\le N/2$ by Lemma~\ref{lem:Hoq10}.
        Hence Lemma~\ref{lem:Hoq10_imm} for $c=-b$ and Lemma~\ref{lem:Hoq10_ineq} imply
        that $\imm (\Sigma )\le (9N+6)/36N$ and hence $\imm(\Sigma)\le \wg (\Sigma )$ since $N\ge 3$.

	Finally we assume that $a\not\equiv 1\pmod N$, $b\not\equiv 0\pmod N$,
        and $a+b\not \equiv 0$.
        Let $c\in \{1-a,-b,a+b\}$ and $n=|\langle c\rangle |$.
        Because of Lemma~\ref{lem:Hoq10} we have $n\notin \{N,N/2\}$ and hence
        $n\le N/3$, $N-3n\ge 0$.
        Moreover, by choosing $c\in \{1-a,-b,a+b\}$ appropriately, we may assume that
        $n\ge 4$.        Since all vertices of $\Sigma $
	belong to cycles of length $\geq4$, we obtain that $\wg(\Sigma)=1/4$. Then Lemma
	\ref{lem:Hoq10_imm} implies that
	\[
		\imm(\Sigma)\leq\frac{N+N/n+n-1}{6N}
                =\frac{3Nn-(Nn-2N-2n^2+2n)}{12Nn} \leq\frac14
	\]
	because $Nn-2N-2n^2+2n=(N-2n)(n-4)+2(N-3n)\geq 0$.
\end{proof}	

\subsection{The graph $\Sigma_{7A}$}
\label{7A}

\begin{lem}
	\label{lem:Hoq11}
	Let $\Sigma $ be a covering of $\Sigma_{7A}$ in Figure \ref{fig:Hoq11} with simply intersecting
	cycles. Then $N=7$ and $(a,b,c)=(2,4,3)$. 
\end{lem}

\begin{proof}
	The $xy$ cycles and their labels are: $(v_1\;v_3\;v_6\;v_2)$ with label
	$a-c+1$ and $(v_5\;v_4\;v_7)$ with $b+c$. The $yx$-cycles are:
	$(v_1\;v_2\;v_5\;v_4)$ with $a-c+1$ and $(v_6\;v_7\;v_3)$ with $b+c$. Since
	$v_1$ and $v_2$ are on the same $xy$- and $yx$-cycles, $N>1$ by Lemma
	\ref{lem:trivial_covering}.  Lemma \ref{lem:loops} on $v_1$ and $v_7$ implies
	that $3a\equiv-1\pmod{N}$ and $2b\equiv1\pmod{N}$. By Lemma
	\ref{lem:xy_and_yx_cycles_with_loops} on $v_1$ and $v_7$ we obtain
	$b\equiv -c\pmod{N}$ and $a\equiv c-1\pmod{N}$.  From this the claim
	follows.
\end{proof}

\begin{figure}[ht]
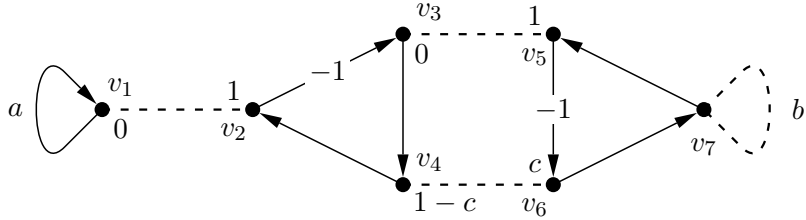

  \begin{graph}(10,4)
    \roundnode{n1}(1,2)
    \roundnode{n2}(3,2)
    \roundnode{n3}(5,3)
    \roundnode{n4}(5,1)
    \roundnode{n5}(7,3)
    \roundnode{n6}(7,1)
    \roundnode{n7}(9,2)
    \dirloopedge{n1}(-.5,-.5)(-.5,.5)
    \edge{n1}{n2}[\yedge ]
    \diredge{n2}{n3}
    \diredge{n3}{n4}
    \diredge{n4}{n2}
    \diredge{n5}{n6}
    \diredge{n6}{n7}
    \diredge{n7}{n5}
		\edgetext{n2}{n3}{$-1$}
		\edgetext{n5}{n6}{$-1$}
    \edge{n3}{n5}[\yedge ]
    \edge{n4}{n6}[\yedge ]
    \loopedge{n7}(.5,.5)(.5,-.5)[\yedge ]
		\freetext(-.15,2){$a$}
		\freetext(10.25,2){$b$}
		\nodetext{n1}(.25,.30){$v_1$}
		\nodetext{n2}(-.25,-.30){$v_2$}
		\nodetext{n3}(.35,.30){$v_3$}
		\nodetext{n4}(.35,.30){$v_4$}
		\nodetext{n3}(.25,-.25){$0$}
		\nodetext{n4}(.55,-.25){$1-c$}
		\nodetext{n5}(-.25,-.30){$v_5$}
		\nodetext{n5}(-.25,.25){$1$}
		\nodetext{n6}(-.25,-.30){$v_6$}
		\nodetext{n6}(-.25,.25){$c$}
		\nodetext{n7}(0,-.5){$v_7$}
		\nodetext{n1}(.25,-.25){$0$}
		\nodetext{n2}(-.25,.25){$1$}
  \end{graph}
	\caption{Schreier graph $\Sigma _{7A}$ and its coverings}
  \label{fig:Hoq11}
\end{figure}

\begin{lem} \label{lem:7A_imm}
	Let $\Sigma$ be a covering of $\Sigma_{7A}$ with simply intersecting
	cycles. Then $\imm(\Sigma)\leq\omega(\Sigma)$.
\end{lem}

\begin{proof}
	First we prove that $\imm(\Sigma)\leq13/49$.  By Lemma \ref{lem:Hoq11}, $N=7$
	and $(a,b,c)=(2,4,3)$. Let $I=\{0,\dots,5\}$.  We claim that $P=v_1[I]\cup
	v_3[*]$ is a plague of size $13$.
  Use the pivots $v_2[I]$, $v_1[(I-3)\cap (I+2)]$,
  $v_4[(I-3)\cap (I+2)]$, $v_5[I\cap (I-2)]$, $v_3[(I-3)\cap (I-1)\cap (I+2)]$
  and $v_7[(I-2)\cap I\cap (I+1)\cap (I+3)]$ to let $P$ spread to
  \begin{align*}
    v_4[I]&\cup v_2[(I-2)\cap (I+3)]\cup v_6[(I-2)\cap I]\\
    &\cup v_7[(I-2)\cap I]\cup v_5[\{0,1,3,5\}] \cup v_5[\{2,4,6\}].
  \end{align*}
  Then $P$ spreads to $v_5[*]$. With $v_3[I-1]$
  and with $v_1[3]$ as pivots, $P$ further spreads to $v_2[I]\cup v_1[6]$.
  Thus $P$ spreads to $v_1[*]\cup v_3[*]$.
  The above calculations with $I=\ZN $ prove
  that $P$ is a plague. Hence $\imm(\Sigma)\leq13/49$. 
  By Lemma~\ref{lem:Hoq11}, the cycle structure of $\Sigma$ is  
	\begin{align*}
		v_1[i]\text{ and }v_2[i]&:\text{ two $4$-cycles,}\\
		v_3[i],v_4[i],v_5[i],v_6[i]&:\text{ one $3$-cycle and one $4$-cycle,}\\
		v_7[i]&:\text{ two $3$-cycles,}
	\end{align*}
	for all $i\in\ZN$. Hence $\wg(\Sigma)=47/168>13/49\geq\imm(\Sigma)$.
\end{proof}

\subsection{The graph $\Sigma_{8A}$}
\label{8A}

\begin{lem} \label{lem:8A_imm}
	Let $\Sigma$ be the covering of $\Sigma_{8A}$ with simply intersecting
	cycles. Then $\imm(\Sigma)\leq\omega(\Sigma)$. 
\end{lem}

\begin{figure}[ht]
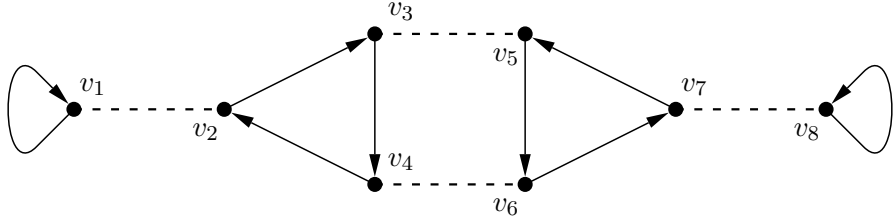

  \begin{graph}(12,4)
    \roundnode{n1}(1,2)
    \roundnode{n2}(3,2)
    \roundnode{n3}(5,3)
    \roundnode{n4}(5,1)
    \roundnode{n5}(7,3)
    \roundnode{n6}(7,1)
    \roundnode{n7}(9,2)
    \roundnode{n8}(11,2)
    \dirloopedge{n1}(-.5,-.5)(-.5,.5)
    \edge{n1}{n2}[\yedge ]
    \diredge{n2}{n3}
    \diredge{n3}{n4}
    \diredge{n4}{n2}
    \diredge{n5}{n6}
    \diredge{n6}{n7}
    \diredge{n7}{n5}
    \edge{n3}{n5}[\yedge ]
    \edge{n4}{n6}[\yedge ]
    \edge{n7}{n8}[\yedge ]
    \dirloopedge{n8}(.5,-.5)(.5,.5)
		\nodetext{n1}(.25,.30){$v_1$}
		\nodetext{n2}(-.25,-.30){$v_2$}
		\nodetext{n3}(.35,.30){$v_3$}
		\nodetext{n4}(.35,.30){$v_4$}
		\nodetext{n5}(-.25,-.30){$v_5$}
		\nodetext{n6}(-.25,-.30){$v_6$}
		\nodetext{n7}(.25,.30){$v_7$}
		\nodetext{n8}(-.25,-.30){$v_8$}
  \end{graph}
	\caption{Schreier graph $\Sigma _{8A}$ and its coverings}
  \label{fig:Hoq12}
\end{figure}
\begin{proof}
	A straightforward computation shows that $v_1[*]\cup v_3[*]$ is a plague of
	size $2N$ and hence $\imm(\Sigma)\leq1/4$.
	Since all cycles have length $\geq4$, the claim follows.
\end{proof}

\subsection{The graph $\Sigma_{9A}$}
\label{9A}

{
\setlength{\unitlength}{1.2cm}
\begin{figure}[h]
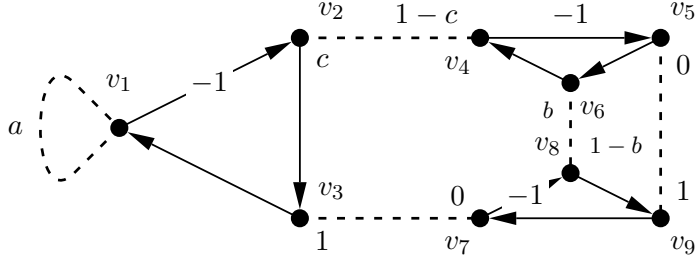

  \begin{graph}(8,3)
    \roundnode{n1}(1,1.5)
    \roundnode{n2}(3,2.5)
    \roundnode{n3}(3,.5)
    \roundnode{n4}(5,2.5)
    \roundnode{n6}(6,2)
    \roundnode{n5}(7,2.5)
    \roundnode{n7}(5,.5)
    \roundnode{n9}(7,.5)
    \roundnode{n8}(6,1)
    \diredge{n1}{n2}
    \diredge{n2}{n3}
    \diredge{n3}{n1}
    \diredge{n4}{n5}
    \diredge{n5}{n6}
    \diredge{n6}{n4}
    \diredge{n7}{n8}
    \diredge{n8}{n9}
    \diredge{n9}{n7}
    \loopedge{n1}(-.5,-.5)(-.5,.5)[\yedge]
    \edge{n2}{n4}[\yedge]
    \edge{n3}{n7}[\yedge]
    \edge{n5}{n9}[\yedge]
    \edge{n6}{n8}[\yedge]
		\edgetext{n1}{n2}{$-1$}
		\edgetext{n7}{n8}{$-1$}
		\freetext(6,2.75){$-1$}
		\freetext(-.15,1.5){$a$}
		\nodetext{n1}(0,.5){$v_1$}
		\nodetext{n2}(.35,.30){$v_2$}
		\nodetext{n2}(.25,-.25){$c$}
		\nodetext{n4}(-.25,-.30){$v_4$}
		\nodetext{n4}(-.6,.25){$1-c$}
		\nodetext{n3}(.35,.30){$v_3$}
		\nodetext{n3}(.25,-.25){$1$}
		\nodetext{n7}(-.25,-.30){$v_7$}
		\nodetext{n5}(.25,.30){$v_5$}
		\nodetext{n9}(.25,-.30){$v_9$}
		\nodetext{n6}(.25,-.30){$v_6$}
		\nodetext{n6}(-.25,-.30){\footnotesize{$b$}}
  	\nodetext{n8}(-.25,.30){$v_8$}
		\nodetext{n8}(.5,.30){\footnotesize{$1-b$}}
		\nodetext{n7}(-.25,.25){$0$}
		\nodetext{n5}(.25,-.30){$0$}
		\nodetext{n9}(.25,.30){$1$}
  \end{graph}
  \caption{Schreier graph $\Sigma _{9A}$ and its coverings}
  \label{fig:Hoq13}
\end{figure}
}

\begin{lem}
	\label{lem:Hoq13}
	Let $\Sigma $ be a covering of $\Sigma_{9A}$ in Figure \ref{fig:Hoq13}
  with simply intersecting cycles. Then $N>1$, $2a\equiv1\pmod{N}$,
  $N$ is odd, $c\equiv a+1\pmod{N}$ and
  $\langle b-1\rangle\cap\langle a+b\rangle=0$.
\end{lem}

\begin{proof}
	The $xy$-cycles and their labels are: $(v_1\;v_2\;v_5\;v_7)$ with label
	$a-c+1$, $(v_6\;v_9)$ with $-b+1$ and $(v_3\;v_8\;v_4)$ with $b+c-1$.  The
	$yx$-cycles are: $(v_1\;v_4\;v_9\;v_3)$ with $a-c+1$, $(v_2\;v_7\;v_6)$ with
	$b+c-1$ and $(v_5\;v_8)$ with $-b+1$.  Lemmas \ref{lem:loops} and
	\ref{lem:xy_and_yx_cycles_with_loops} on $v_1$ imply $2a\equiv1\pmod{N}$ and
	$c\equiv a+1\pmod{N}$. Since $v_2$ and $v_7$ are on the same $xy$-cycle and
	the same $yx$-cycle, we conclude that $N>1$ by
	Lemma~\ref{lem:trivial_covering}.  By Lemma \ref{lem:xy_and_yx_cycles} on
	$v_6$ the claim follows.
\end{proof}

\begin{lem}
	\label{lem:Hoq13_imm}
	Let $\Sigma$ be a covering of $\Sigma_{9A}$ with simply intersecting cycles.
	Then
	\begin{equation*}
		\imm(\Sigma)\leq
		\begin{cases}
			8/27 & \text{ if $a+b\equiv 0\pmod N$ and $b\equiv 1\pmod N$, }\\
			11/45 & \text{ if $a+b\not\equiv 0\pmod N$ and $b\not \equiv 1\pmod N$, }\\
			7/27 & \text{ otherwise.}
		\end{cases}
	\end{equation*}
\end{lem}

\begin{proof}
  There are four cases to consider. Assume first that $b\equiv 1\pmod N$
  and $a+b\equiv 0\pmod N$. Lemma~\ref{lem:Hoq13} implies then that $N=3$,
  $a\equiv 2\pmod N$, $b\equiv 1\pmod N$, and $c\equiv 0\pmod N$.  Let
	$I=\{0,1\}$. Then $P=v_1[I]\cup v_2[*]\cup
	v_5[*]$ is a plague of size $8$. Indeed, 
	\begin{center}
	\begin{tabular}{c|cccccc}
		pivot & $v_4[*]$ & $v_{1}[I+1]$ & $v_{3}[I-1]$ & $v_{8}[I+1]$ & $v_{2}[I\cap(I-1)]$ & $v_{5}[I+1]$ \tabularnewline
		\hline 
		& $v_6[*]$ & $v_{3}[I+1]$ & $v_{7}[I-1]$ & $v_{9}[I+2]$ & $v_{4}[I\cap(I+1)]$ & $v_{4}[I+2]$ \tabularnewline
	 \end{tabular}
  \end{center}
	and hence, since $(I\cap(I+1))\cup(I+2)=\ZN$, $P$ spreads to
        $v_4[*]$. From this the claim follows.

	Assume now that $b\not \equiv 1\pmod N$ and that $a+b\equiv 0\pmod N$.
        Then $|\langle b-1\rangle|\geq3$ because
	$N$ is odd. 
	Let $I$ be a set of representatives for $\ZN/\langle b-1\rangle$. We
	claim that $v_1[*]\cup v_2[*]\cup v_5[I]$ is a plague. We compute
	\begin{center}
	\begin{tabular}{c|ccccccc}
    pivot & $v_{1}[*]$ & $v_{3}[*]$ & $v_{2}[*]$ & $v_{4}[I]$ & $v_{5}[I-1]$ & $v_{7}[I]$ & $v_{6}[I+b-1]$ \tabularnewline
		 \hline
 	  & $v_{3}[*]$ & $v_{7}[*]$ & $v_{4}[*]$ & $v_{6}[I]$ & $v_{9}[I]$ & $v_{8}[I]$ & $v_{5}[I+b-1]$ \tabularnewline
		\end{tabular}
	\end{center}
	and the claim follows from Example \ref{exa:9A}. Therefore 
	\begin{equation}
		\label{eq:Hoq13_imm}
		\imm(\Sigma)\leq\frac{2N+|I|}{9N}\leq7/27
	\end{equation}
	since $|I|\leq N/3$. 

	Assume now that $b\equiv 1\pmod N$ and $a+b\not \equiv 0\pmod N$.
        Let $I$ be a set of representatives for
	$\ZN/\langle a+b\rangle$. 
	Then $P=v_4[*]\cup v_8[*]\cup v_2[I]$ is a plague. Indeed, $P$ spreads
        according to the following table.
	\begin{center}
	\begin{tabular}{c|ccccc}
		pivot & $v_{6}[*]$ & $v_{9}[*]$ & $v_{4}[I-a-1]$ & $v_{8}[I-a-b-1]$ & $v_{7}[I-a-b]$ \tabularnewline
		 \hline
		 & $v_{5}[*]$ & $v_{7}[*]$ & $v_{6}[I-a-1]$ & $v_{9}[I-a-b]$ & $v_{3}[I-a-b+1]$ \tabularnewline
		\end{tabular}
	\end{center}
	Pivoting from $v_2[I-a-b]$ and $v_3[I-a-b]$, $P$ spreads further to $v_1[I-a-b+1]\cup v_2[I-a-b]$.
        As before, since $|I|\leq N/3$, and using Example
	\ref{exa:9A} with $\lambda=a+b$, we conclude that $P$ spreads to $v_2[*]$.
        Hence $P$ is a plague and $\imm(\Sigma)\leq7/27$. 

	Finally, assume that $b\not \equiv 1\pmod N$ and $a+b\not \equiv 0\pmod N$.
        From Lemma \ref{lem:Hoq13} we
	obtain that $N$ is odd,
        $\langle a+b\rangle\cap\langle b-1\rangle=0$,
        and $c\equiv a+1\pmod N$.
        If $|\langle a+b\rangle|=3$
	then $|\langle b-1\rangle|\ge 5$ and $v_1[*]\cup v_2[*]\cup v_5[I]$ is a
	plague, where $I$ is a set of representatives for $\ZN/\langle b-1\rangle$. 
        Indeed, the calculations
	\begin{center}
	\begin{tabular}{c|ccccccc}
		pivot & $v_{1}[*]$ & $v_{2}[*]$ & $v_{3}[*]$ & $v_{4}[I]$ & $v_{5}[I-1]$
                 & $v_8[I-1]$ & $v_4[I+b-1]$
                 \tabularnewline
		 \hline
		 & $v_{3}[*]$ & $v_{4}[*]$ & $v_{7}[*]$ & $v_{6}[I]$ & $v_{9}[I]$
                 & $v_6[I+b-1]$ & $v_5[I+b-1]$
                 \tabularnewline
		\end{tabular}
	\end{center}
        together with Example~\ref{exa:9A} with $\lambda =1-b$ show that
        $P$ is a plague if $P\cup v_5[*]$ is a plague. However, the latter is easy to check.
	On the other hand, if $|\langle a+b\rangle|\geq5$,
        then let $P=v_4[*]\cup v_8[*]\cup v_2[I]$,
        where $I$ is a set of representatives for $\ZN/\langle a+b\rangle$.
        Then $P$ is a plague. Indeed, the calculations
	\begin{center}
	\begin{tabular}{c|ccccccc}
		pivot & $v_{6}[*]$ & $v_{9}[*]$ & $v_{4}[I-a-1]$ & $v_{8}[I-a-b-1]$ & $v_{7}[I-a-b]$
                 \tabularnewline
		 \hline
		 & $v_{5}[*]$ & $v_{7}[*]$ & $v_{6}[I-a-1]$ & $v_{9}[I-a-b]$ & $v_{3}[I-a-b+1]$
                 \tabularnewline
                 pivot & & & & $v_2[I-a-b]$ & $v_3[I-a-b]$
                 \tabularnewline
		 \hline
                 & & & & $v_1[I-a-b+1]$ & $v_2[I-a-b]$
                 \tabularnewline
		\end{tabular}
	\end{center}
        together with Example~\ref{exa:9A} with $\lambda =a+b$ show that
        $P$ is a plague if $P\cup v_2[*]$ is a plague.
        Again, the latter is easy to check.
	Since $|I|\leq N/5$, in the last two cases we obtain that
	\[
		\imm(\Sigma)\leq\frac{2N+|I|}{9N}\leq\frac{2N+N/5}{9N}=11/45.
	\]
	This completes the proof.
\end{proof}

\begin{lem}
	\label{lem:9A_imm}
	Let $\Sigma$ be a covering of $\Sigma_{9A}$ with simply intersecting cycles.
	Then $\imm(\Sigma)\leq\wg(\Sigma)$.
\end{lem}

\begin{proof}
	The cycle structure at each vertex of $\Sigma$ is the following:
	\begin{align}
		\label{eq:Hoq13_v1}	v_1[i] & \text{ two $4$-cycles, }\\
		\label{eq:Hoq13_v2v3v4v7}	v_2[i],v_3[i],v_4[i],v_7[i] & \text{ one $4$-cycle and one cycle of length $3|\langle a+b\rangle|$,}\\
		\label{eq:Hoq13_v5v9}	v_5[i],v_9[i] & \text{ one $4$-cycle and one cycle of length $2|\langle b-1\rangle|$,}\\
		\label{eq:Hoq13_v6v8}	v_6[i],v_8[i] & \text{ cycles of length $3|\langle a+b\rangle|$ and $2|\langle b-1\rangle|$,}
	\end{align}
	for all $i\in\ZN$. 
	From \eqref{eq:Hoq13_v1}--\eqref{eq:Hoq13_v6v8} it is straightforward to compute
	the weight at every vertex of $\Sigma$ and $\wg(\Sigma)$:
	\[
		\wg(\Sigma)=\begin{cases}
			11/36 & \text{if $a\equiv -b\pmod N$, $b\equiv 1\pmod N$,}\\
			5/18 & \text{if $a\equiv -b\pmod N$, $b\not\equiv 1\pmod N$,}\\
			31/108 & \text{if $a\not\equiv -b\pmod N$, $b\equiv 1\pmod N$,}\\
			1/4 & \text{if $a\not\equiv -b\pmod N$, $b\not\equiv 1\pmod N$.}
		\end{cases}
	\]
	From this and from Lemma~\ref{lem:Hoq13_imm}
  the claim follows.
\end{proof}

\subsection{The graph $\Sigma_{12A}$}
\label{12A}

\begin{lem}
	\label{lem:12A_imm}
	Let $\Sigma$ be a covering of $\Sigma_{12A}$ with simply intersecting cycles.
	Then $\imm(\Sigma)\leq1/4=\wg(\Sigma)$.
\end{lem}

\begin{proof}
  In $\Sigma $, all cycles have length $\geq4$,
  hence $\wg(\Sigma )=1/4$.
	A straightforward computation shows that $v_1[*]\cup v_2[*]\cup v_5[*]$ is a
	plague of $\Sigma $ of size $3N$ and hence $\imm(\Sigma)\leq1/4=\wg(\Sigma)$.
\end{proof}

%

\begin{figure}[ht]
  \begin{graph}(12,3)
    \roundnode{n1}(1,1.5)
    \roundnode{n2}(3,2.5)
    \roundnode{n3}(3,.5)
    \roundnode{n4}(5,2.5)
    \roundnode{n5}(6,2)
    \roundnode{n6}(7,2.5)
    \roundnode{n7}(5,.5)
    \roundnode{n8}(7,.5)
    \roundnode{n9}(6,1)
    \roundnode{n10}(9,.5)
    \roundnode{n11}(9,2.5)
    \roundnode{n12}(11,1.5)
    \diredge{n1}{n2}
    \diredge{n2}{n3}
    \diredge{n3}{n1}
    \diredge{n4}{n5}
    \diredge{n5}{n6}
    \diredge{n6}{n4}
    \diredge{n7}{n8}
    \diredge{n8}{n9}
    \diredge{n9}{n7}
    \diredge{n10}{n11}
    \diredge{n11}{n12}
    \diredge{n12}{n10}
    \loopedge{n1}(-.5,-.5)(-.5,.5)[\yedge]
    \edge{n2}{n4}[\yedge]
    \edge{n3}{n7}[\yedge]
    \edge{n5}{n9}[\yedge]
    \edge{n6}{n11}[\yedge]
    \edge{n8}{n10}[\yedge]
    \loopedge{n12}(.5,.5)(.5,-.5)[\yedge]
		\nodetext{n1}(0,.5){$v_1$}
		\nodetext{n2}(.35,.30){$v_2$}
		\nodetext{n4}(-.25,-.30){$v_4$}
		\nodetext{n3}(.35,.30){$v_3$}
		\nodetext{n7}(-.25,-.30){$v_7$}
		\nodetext{n5}(.25,-.30){$v_5$}
		\nodetext{n9}(-.30,.30){$v_9$}
		\nodetext{n6}(.25,-.30){$v_6$}
		\nodetext{n8}(.25,-.30){$v_8$}
		\nodetext{n10}(-.35,.30){$v_{10}$}
		\nodetext{n11}(-.35,.30){$v_{11}$}
		\nodetext{n12}(0,.5){$v_{12}$}
  \end{graph}
  \caption{Schreier graph $\Sigma _{12A}$}
  \label{fig:Hoq14}
\end{figure}

\subsection{The graph $\Sigma_{12B}$}
\label{12B}

{
\setlength{\unitlength}{1.2cm}
\begin{figure}[h]
  \begin{graph}(8,3)
    \roundnode{n1}(1,2.5)
    \roundnode{n2}(3,2.5)
    \roundnode{n3}(2,2)
    \roundnode{n4}(1,.5)
    \roundnode{n5}(2,1)
    \roundnode{n6}(3,.5)
    \roundnode{n7}(5,2.5)
    \roundnode{n8}(7,2.5)
    \roundnode{n9}(6,2)
    \roundnode{n10}(5,.5)
    \roundnode{n11}(6,1)
    \roundnode{n12}(7,.5)
    \diredge{n1}{n2}
    \diredge{n2}{n3}
    \diredge{n3}{n1}
    \diredge{n4}{n5}
    \diredge{n5}{n6}
    \diredge{n6}{n4}
    \diredge{n7}{n8}
    \diredge{n8}{n9}
    \diredge{n9}{n7}
    \diredge{n10}{n11}
    \diredge{n11}{n12}
    \diredge{n12}{n10}
    \edge{n1}{n4}[\yedge]
    \edge{n2}{n7}[\yedge]
    \edge{n3}{n5}[\yedge]
    \edge{n6}{n10}[\yedge]
    \edge{n8}{n12}[\yedge]
    \edge{n9}{n11}[\yedge]
		\nodetext{n1}(-.30,.25){$v_1$}
		\nodetext{n1}(-.30,-.25){$a$}
		\nodetext{n2}(.30,-.30){$v_2$}
		\nodetext{n2}(.30,.30){$0$}
		\nodetext{n3}(-.30,-.25){$v_3$}
		\nodetext{n3}(.30,-.25){$1$}
	  \nodetext{n4}(-.30,-.25){$v_4$}
		\nodetext{n4}(-.50,.25){\footnotesize{$1-a$}}
		\nodetext{n5}(.30,.25){$v_5$}
		\nodetext{n5}(-.30,.25){$0$}
		\nodetext{n6}(.30,-.25){$v_6$}
		\nodetext{n6}(.30,.25){$0$}
 	  \nodetext{n7}(-.30,.25){$v_7$}
 	  \nodetext{n7}(-.30,-.25){$1$}
		\nodetext{n8}(.30,.25){$v_8$}
		\nodetext{n8}(.30,-.25){$c$}
		\nodetext{n9}(-.30,-.25){$v_9$}
		\nodetext{n9}(.50,-.25){\footnotesize{$1-b$}}
 	  \nodetext{n10}(-.30,.30){$v_{10}$}
 	  \nodetext{n10}(-.30,-.30){$1$}
		\nodetext{n11}(.40,.25){$v_{11}$}
		\nodetext{n11}(-.30,.25){$b$}
		\nodetext{n12}(.40,-.25){$v_{12}$}
		\nodetext{n12}(.50,.25){\footnotesize{$1-c$}}
 		\edgetext{n1}{n2}{$-1$}	
 		\edgetext{n4}{n5}{$-1$}	
 		\edgetext{n7}{n8}{$-1$}	
 		\edgetext{n10}{n11}{$-1$}	
  \end{graph}
  \caption{Schreier graph $\Sigma _{12B}$ and its coverings}
  \label{fig:Hoq15}
\end{figure}
}

\begin{lem}
	\label{lem:Hoq15}
	Let $\Sigma $ be a covering of $\Sigma_{12B}$ in Figure \ref{fig:Hoq15} with
  simply intersecting cycles. Then $N>1$ and the following conditions are satisfied:
	\begin{enumerate}
		\item\label{item:Hoq15_1} $\langle1-a\rangle\cap\langle a-c\rangle=0$,
		\item\label{item:Hoq15_2} $\langle1-b\rangle\cap\langle a-c\rangle=0$,
		\item\label{item:Hoq15_3} $\langle1-b\rangle\cap\langle 1-a\rangle=0$,
		\item\label{item:Hoq15_4} $\langle1-b\rangle\cap\langle b+c\rangle=0$,
		\item\label{item:Hoq15_5} $\langle a-c\rangle\cap\langle b+c\rangle=0$,
		\item\label{item:Hoq15_6} $(-c+\langle a-c\rangle)\cap\langle b-1\rangle=\emptyset$.
	\end{enumerate}
\end{lem}

\begin{proof}
	The $yx$-cycles and their labels are: $(v_3\;v_4)$ with $1-a$,
	$(v_8\;v_{11})$ with $b+c$, $(v_1\;v_7\;v_{12}\;v_6)$ with $a-c$ and
	$(v_2\;v_5\;v_{10}\;v_9)$ with $1-b$.  The $xy$-cycles are: $(v_1\;v_5)$ with
	$1-a$, $(v_9\;v_{12})$ with $b+c$, $(v_2\;v_8\;v_{10}\;v_4)$ with $a-c$ and
	$(v_3\;v_6\;v_{11}\;v_7)$ with $1-b$.
        Then Lemma~\ref{lem:xy_and_yx_cycles} implies \eqref{item:Hoq15_1}--\eqref{item:Hoq15_5}.
        Moreover, Lemma \ref{lem:trivial_covering} on $v_6$ and $v_7$ implies that $N>1$. 
	Finally, the claim \eqref{item:Hoq15_6} follows from
	Lemma \ref{lem:from_v_to_w} on $v_2$ and $v_{10}$. Indeed, the $xy$-path
        starting in $v_2[i]$ for some $i\in \ZN$, which has length $4k+2$ with $k\in \Z $,
        ends in $v_{10}[1-c+k(a-c)]$, and any $yx$-path of length $4l+2$
        with the same starting point ends in $v_{10}[1+l(1-b)]$.
\end{proof}

\begin{lem}
	\label{lem:Hoq15_imm}
	Let $\Sigma $ be a covering of $\Sigma_{12B}$ as in Figure \ref{fig:Hoq15} with
	simply intersecting cycles. Then 
	\[
	\imm(\Sigma)\leq\begin{cases}
		\frac{3N+1}{12N} & \text{if $a\equiv1\pmod N$ or $b+c\equiv 0\pmod N$, }\\
		1/4 & \text{otherwise.}
	\end{cases}
	\]
\end{lem}

\begin{proof}
	%
The cycle structure at each vertex of $\Sigma$ is:
	\begin{align}
		\label{eq:hoq15_v1v4}v_1[i],v_4[i]: & \text{ cycles of length $4|\langle a-c\rangle|$ and $2|\langle1-a\rangle|$},\\
		\label{eq:hoq15_v2v6v_7v10}v_2[i],v_6[i],v_7[i],v_{10}[i]: & \text{ cycles of length $4|\langle 1-b\rangle|$ and $4|\langle a-c\rangle|$},\\
		\label{eq:hoq15_v3v5}v_3[i],v_5[i]: & \text{ cycles of length $2|\langle1-a\rangle|$ and $4|\langle1-b\rangle|$},\\
		\label{eq:hoq15_v8v12}v_8[i],v_{12}[i]: & \text{ cycles of length $2|\langle b+c\rangle|$ and $4|\langle a-c\rangle|$},\\
		\label{eq:hoq15_v9v11}v_9[i],v_{11}[i]: & \text{ cycles of length $2|\langle b+c\rangle|$ and $4|\langle1-b\rangle|$},
	\end{align}
	for all $i\in\ZN$.

Assume first that $a\equiv 1\pmod N$ or $b+c\equiv 0\pmod N$.
We claim that there exists a subset $I\subseteq\ZN$ of size one such that the
set $P=v_2[*]\cup v_5[*]\cup v_9[*]\cup v_{10}[I]$ is a plague of size $3N+1$.
We compute
\begin{center}
\begin{tabular}{c|cccccc}
	pivot & $v_{3}[*]$ & $v_{4}[*]$ & $v_{7}[*]$ & $v_{11}[I-1]$ & $v_{10}[I]$ & $v_{12}[I]$\tabularnewline
	\hline 
	 & $v_{1}[*]$ & $v_{6}[*]$ & $v_{8}[*]$ & $v_{12}[I]$ & $v_{11}[I]$ & $v_{10}[I+1]$\tabularnewline
 \end{tabular}
\end{center}
and hence $P$ spreads to $v_{10}[*]$ by Example \ref{exa:9A}.
Thus $P$ is a plague since $P\cup v_{10}[*]$ is a plague.
In this case $\imm(\Sigma)\leq(3N+1)/12N$.  

Assume now that $a\not\equiv 1\pmod N$ and $b+c\not\equiv 0\pmod N$.
Let $I$ be a set of representatives for $\ZN/\langle a-1\rangle$ and
let $J$ be a set of representatives for $\ZN/\langle b+c\rangle$.
We claim that  $v_3[I]\cup v_8[J]\cup v_1[*]\cup v_2[*]$ is a plague.
We compute
\begin{center}
	\begin{tabular}{c|cccc}
		pivot & $v_{3}[*]$ & $v_{4}[*]$ & $v_{5}[I-1]$ & $v_1[I+a-1]$\tabularnewline
		\hline 
		 & $v_{5}[*]$ & $v_{6}[*]$ & $v_4[I]$ & $v_{3}[I+a-1]$\tabularnewline
	 \end{tabular}
 \end{center}
 and therefore $P$ spreads to $v_3[*]\cup v_4[*]$
 by Example \ref{exa:9A}. Now the calculations
	\begin{center}
	\begin{tabular}{c|cccc}
		pivot & $v_{2}[*]$ & $v_{6}[*]$ & $v_{9}[J]$ & $v_{12}[J+b]$\tabularnewline
		\hline 
		 & $v_{7}[*]$ & $v_{10}[*]$ & $v_{11}[J+b]$ & $v_{8}[I+b+c]$\tabularnewline
	 \end{tabular}
 \end{center}
 show that $P$ spreads to $v_8[*]$
 by Example \ref{exa:9A}. Then the claim
 follows and in this case $\imm(\Sigma)\leq1/4$ since $|I|+|J|\le N$.
\end{proof}

\begin{lem} \label{lem:12B_imm}
	Let $\Sigma$ be a covering of $\Sigma_{12B}$ with simply intersecting cycles.
	Then $\imm(\Sigma)\leq\wg(\Sigma)$.
\end{lem}

\begin{proof}
Using \eqref{eq:hoq15_v1v4}--\eqref{eq:hoq15_v9v11} we obtain
	\[
	\wg(\Sigma)=\begin{cases}
		11/36 & \text{ if $a\equiv 1\pmod N$ and $b+c\equiv 0\pmod N$, }\\
		5/18 & \text{ if $a\equiv 1\pmod N$ and $b+c\not \equiv 0\pmod N$, }\\
		5/18 & \text{ if $a\not\equiv 1\pmod N$ and $b+c\equiv 0\pmod N$, }\\
		1/4 & \text{ otherwise.}
	\end{cases}
	\]
	First notice that if $N=2$ then $a\equiv 1\pmod N$, $b\equiv 1\pmod N$,
        and $c\equiv 1\pmod N$ by Lemma~\ref{lem:Hoq15}(6).
	Then $\imm(\Sigma)\leq\frac{3N+1}{12N}=7/24<11/36=\wg(\Sigma)$
        by Lemma~\ref{lem:Hoq15_imm}, and the claim holds.
  So we may assume that $N>2$. 

	If $a\equiv1\pmod N$ or $b+c\equiv0\pmod N$ then Lemma \ref{lem:Hoq15_imm}
  and $N\ge 3$ imply that 
	\[
		\imm(\Sigma)\leq\frac{3N+1}{12N}\leq10/36\le \wg(\Sigma).
	\]
	Otherwise $\imm(\Sigma)\leq 1/4=\wg(\Sigma)$ by Lemma \ref{lem:Hoq15_imm}.
\end{proof}

\subsection{The graph $\Sigma_{12C}$}
\label{12C}

\setlength{\unitlength}{1.5cm}
\begin{figure}[ht]
  \begin{graph}(6,6)
    \roundnode{n1}(.5,4.5)
    \roundnode{n2}(1,5.5)
    \roundnode{n3}(1.5,4.5)
    \roundnode{n4}(2.5,3.5)
    \roundnode{n5}(3.5,3.5)
    \roundnode{n6}(3,2.5)
    \roundnode{n7}(2.5,.5)
    \roundnode{n8}(3,1.5)
    \roundnode{n9}(3.5,.5)
    \roundnode{n10}(4.5,4.5)
    \roundnode{n11}(5,5.5)
    \roundnode{n12}(5.5,4.5)
    \diredge{n1}{n2}
    \diredge{n2}{n3}
    \diredge{n3}{n1}
    \diredge{n4}{n5}
    \diredge{n5}{n6}
    \diredge{n6}{n4}
    \diredge{n7}{n8}
    \diredge{n8}{n9}
    \diredge{n9}{n7}
    \diredge{n10}{n11}
    \diredge{n11}{n12}
    \diredge{n12}{n10}
    \edge{n1}{n7}[\yedge]
    \edge{n2}{n11}[\yedge]
    \edge{n3}{n4}[\yedge]
    \edge{n5}{n10}[\yedge]
    \edge{n6}{n8}[\yedge]
    \edge{n9}{n12}[\yedge]
		\nodetext{n1}(-.35,0){$v_1$}
		\nodetext{n1}(0,-0.35){$a$}
		\nodetext{n2}(-.35,0){$v_2$}
		\nodetext{n2}(.30,.30){$0$}
		\nodetext{n11}(-.30,.30){$1$}
		\nodetext{n3}(.35,0){$v_3$}
		\nodetext{n3}(0,-.30){$0$}
		\nodetext{n4}(-.35,0){$v_4$}
		\nodetext{n4}(0,.30){$1$}
		\nodetext{n5}(0,.30){$b$}
		\nodetext{n5}(.35,0){$v_5$}
		\nodetext{n6}(-.30,-.30){$v_6$}
		\nodetext{n7}(-.30,-.30){$v_7$}
		\nodetext{n7}(-.5,0){$1-a$}
		\nodetext{n8}(.30,.30){$v_8$}
		\nodetext{n9}(.30,-.30){$v_9$}
		\nodetext{n9}(.30,0){$1$}
		\nodetext{n10}(-.45,0){$v_{10}$}
		\nodetext{n10}(.25,-.30){\footnotesize{$1-b$}}
		\nodetext{n11}(.45,0){$v_{11}$}
		\nodetext{n12}(.45,0){$v_{12}$}
		\nodetext{n12}(0,-.35){$0$}
		\nodetext{n6}(.40,-.30){\footnotesize{$1-c$}}
		\nodetext{n8}(-.30,.30){\footnotesize{$c$}}
		\edgetext{n1}{n2}{$-1$}
		\edgetext{n4}{n5}{$-1$}
		\edgetext{n10}{n11}{$-1$}
		\edgetext{n7}{n8}{$-1$}
  \end{graph}
  \caption{Schreier graph $\Sigma _{12C}$ and its coverings}
  \label{fig:Hoq16}
\end{figure}

\begin{lem}
	\label{lem:Hoq16}
  Let $\Sigma $ be a covering of $\Sigma_{12C}$ in Figure \ref{fig:Hoq16}
  with simply intersecting cycles. Then the following hold:
	\begin{enumerate}
		\item $\langle -a-c+1\rangle\cap\langle a+1 \rangle=0$,
		\item $\langle a+1\rangle\cap\langle -b\rangle=0$,
		\item $\langle -b\rangle\cap\langle -a-c+1\rangle=0$,
		\item $\langle -b\rangle\cap\langle b+c\rangle=0$,
		\item $\langle b+c\rangle\cap\langle -a-c+1\rangle=0$,
		\item $\langle b+c\rangle\cap\langle a+1\rangle=0$.
	\end{enumerate}
\end{lem}

\begin{proof}
	The $yx$-cycles and their labels are: $(v_1\;v_{11}\;v_9)$
	with label $a+1$, $(v_2\;v_4\;v_{10})$ with $-b$.  $(v_3\;v_7\;v_6)$ with
	$-a-c+1$ and $(v_5\;v_8\;v_{12})$ with $b+c$. The $xy$-cycles are:
	$(v_1\;v_8\;v_4)$ with $-a-c+1$, $(v_2\;v_{12}\;v_7)$ with $a+1$,
	$(v_{11}\;v_3\;v_5)$ with $-b$ and $(v_9\;v_{10}\;v_6)$ with $b+c$.  Then the
	claim follows from Lemma \ref{lem:xy_and_yx_cycles}.
\end{proof}

\begin{exa}
	\label{exa:12C}
	The trivial covering of $\Sigma_{12C}$ is isomorphic to the Hurwitz orbit of
	size $12$ of \cite[Figure 12]{MR2891215}. The covering
	$\Sigma_{12C}^{2;1,0,0}$ is isomorphic to the Hurwitz orbit of size $24$ of
	\cite[Figure 14]{MR2891215}. 
\end{exa}

\begin{lem}
	\label{lem:ineq_12C}
	 Assume that $N/n\geq6$ and $n\geq5$. Then 
	 \[
	 n\frac{N/n+1}{2}+\frac{N/n-1}{2}+\frac{N}{n}\leq N.
	 \]
 \end{lem}

 \begin{proof}
         After multiplication with $2n$ the inequality takes the form
         $Nn-3N+n-2n^2\geq 0$, or $(N-6n)(n-3)+4n(n-5)+3n\geq 0$.
 \end{proof}

\begin{lem}
	\label{lem:Hoq16_imm}
  Let $\Sigma $ be a covering of $\Sigma_{12C}$ in Figure \ref{fig:Hoq16}
	with simply intersecting cycles. Then 
	\[
	\imm(\Sigma)\leq\begin{cases}
		1/3 & \text{if $\Sigma$ is the trivial covering,}\\
		7/24 & \text{if $\Sigma=\Sigma_{12C}^{2;1,0,0}$,}\\
	  7/24 & \text{if $1+a\equiv 1-a-c\equiv b\equiv 0\pmod N$,}\\
          & \text{$b+c\not\equiv 0\pmod N$,}\\
	  5/18 & \text{if $1+a\equiv 1-a-c\equiv 0\pmod N$,}\\
               &  \text{$b,b+c\not\equiv 0\pmod N$,}\\
		1/4&\text{otherwise.}
	\end{cases}
	\]
\end{lem}

\begin{proof}
	The cycle structure at each vertex is the following:
	\begin{align}
		\label{eq:Hoq16_v1v7}	v_1[i],v_7[i]: & \text{ cycles of length $3|\langle1-a-c\rangle|$ and $3|\langle1+a\rangle|$, }\\
		\label{eq:Hoq16_v2v11}v_2[i],v_{11}[i]: & \text{ cycles of length $3|\langle1+a\rangle|$ and $3|\langle-b\rangle|$, }\\
		\label{eq:Hoq16_v3v4}	v_3[i],v_4[i]: & \text{ cycles of length $3|\langle-b\rangle|$ and $3|\langle1-a-c\rangle|$, }\\
		\label{eq:Hoq16_v5v10}	v_5[i],v_{10}[i]: & \text{ cycles of length $3|\langle b+c\rangle|$ and $3|\langle-b\rangle|$, }\\
		\label{eq:Hoq16_v6v8}	v_6[i],v_8[i]: & \text{ cycles of length $3|\langle b+c\rangle|$ and $3|\langle1-a-c\rangle|$, }\\
		\label{eq:Hoq16_v9v12}	v_9[i],v_{12}[i]: & \text{ cycles of length $3|\langle b+c\rangle|$ and $3|\langle1+a\rangle|$, }
	\end{align}
	for all $i\in\ZN$. Without loss of generality we may assume that
	\begin{equation}
		\label{eq:Hoq16_symmetry}
		|\langle 1+a\rangle|\leq|\langle1-a-c\rangle|\leq|\langle-b\rangle|\leq|\langle b+c\rangle|.
	\end{equation}

	We split the proof into five cases according to the number of trivial groups
	in \eqref{eq:Hoq16_symmetry}. However, two of these cases will be considered
        simultaneously.

        First assume that $1+a\equiv 1-a-c\equiv b\equiv b+c\equiv 0\pmod N$. Then
	$N=1$ (the trivial covering) or $N=2$ and $(a,b,c)=(1,0,0)$. Since these
	orbits appear as Hurwitz orbits of braided racks, see Example \ref{exa:12C},
	the claim follows from \cite[Proposition 10]{MR2891215}.

	Assume now that $1+a\equiv a+c-1\equiv b\equiv 0\pmod N$
        and $b+c\not\equiv 0\pmod N$. Then $a\equiv-1\pmod N$, $b\equiv 0\pmod N$,
        $c\equiv 2\pmod N$, and $N\geq3$.
        Let $I$ be a set of representatives for $\ZN/\langle b+c\rangle$.
  Then $|I|\le N/2$.
	We claim that $P=v_1[*]\cup v_2[*]\cup v_3[*]\cup v_5[I]$ is a plague. With the
	sequence of pivots 
	$v_1[*]$, $v_2[*]$, $v_3[*]$, $v_4[I]$, $v_5[I-1]$, $v_6[I]$,
	$v_7[I+c]$, $v_9[I+c]$, $v_{10}[I+c]$ we see that $P$ spreads to
  $$ v_7[*]\cup v_{11}[*]\cup v_4[*]\cup v_6[I]\cup v_{10}[I-b]
     \cup v_8[I+c]\cup v_9[I+c]\cup v_{12}[I+c]\cup v_5[I+b+c].$$
	Hence the claim follows from Example \ref{exa:9A} for the fiber over $v_5$.

	Now assume that exactly two groups of \eqref{eq:Hoq16_symmetry} are
	trivial, i.e., $1+a\equiv a+c-1\equiv 0\pmod N$, and $b,b+c\not\equiv 0\pmod N$.
        Then $a\equiv -1\pmod N$, $c\equiv 2\pmod N$, and
	$1<|\langle b\rangle |\le |\langle b+c\rangle |$.
  From Lemma \ref{lem:Hoq15} we
	obtain $\langle b\rangle\cap\langle b+c\rangle=0$ and hence 
	$|\langle b+c\rangle |\ge 3$.
  Therefore the previous computation with $|I|\le N/3$ yields the claim.

	Finally, assume that $1-a-c,b,b+c\not\equiv 0\pmod N$.
        We prove that $\imm(\Sigma)\leq1/4$. Let $n=|\langle b+c\rangle |$.
        Instead of $|\langle 1-a-c\rangle |\le |\langle b\rangle |\le |\langle b+c\rangle |$
        we assume that $|\langle 1-a-c\rangle |,|\langle b\rangle |\le n$
        and $n|\langle 1-a-c\rangle |\not=N$. This is possible since
        any subgroup of $\ZN$ is uniquely determined by its order and any two of these three subgroups
        have trivial intersection by Lemma~\ref{lem:Hoq16}. Observe that $n\ge 5$.
	We claim that there exist two subsets $I,J\subseteq \ZN $ with $|I|+|J|\le N$, such that
	\begin{equation}
		\label{eq:plague12C}
		P=v_1[*]\cup v_2[*]\cup v_3[I]\cup v_5[J]
	\end{equation}
	is a plague.  We split the proof into several steps. 

	\medskip
	\textbf{Step 1.}
	The following computation shows that if $x+a+c\in I$ and $x+b+c\in I\cap J$,
  then $P$ spreads to $\{v_5[x]\}$: 
	 {\footnotesize\begin{center}
		\begin{tabular}{c|c}
			pivot & \tabularnewline
			\hline 
			$v_{3}[*]$ & $v_{4}[*]$\tabularnewline
			$v_{1}[I]$ & $v_{7}[I-a+1]$\tabularnewline
			$v_{4}[I\cap J]$ & $v_{6}[I\cap J]$\tabularnewline
			$v_{5}[(I-1)\cap(J-1)]$ & $v_{10}[(I-b)\cap(J-b)]$\tabularnewline
			$v_{11}[(I-b-1)\cap(J-b-1)]$ & $v_{12}[(I-b)\cap(J-b)]$\tabularnewline
			$v_{9}[(I-a)\cap(I-b)\cap(J-b)]$ & $v_{8}[(I-a)\cap(I-b)\cap(J-b)]$\tabularnewline
			$v_{6}[(I-a-c)\cap(I-b-c)\cap(J-b-c)]$ & $v_{5}[(I-a-c)\cap(I-b-c)\cap(J-b-c)]$\tabularnewline
		 \end{tabular}
	 \end{center}}

	 \textbf{Step 2.}
	 We claim that if $x$ satisfies $x+a+c-1\in I$, $x+a+c\in I$ and $x+b+c\in
   I\cap J$, then $P$ spreads to $\{v_3[x],v_5[x]\}$.
   This follows from Step~1 and
	 the following fact: if $x+a+c-1\in I$ and $x\in J$, then $P$ spreads to
   $\{v_3[I]\}$. The proof is obtained from the following table:
	 \begin{center}
		 \begin{tabular}{c|c}
			 pivot & \tabularnewline
			 \hline 
			 $v_{3}[*]$ & $v_{4}[*]$\tabularnewline
			 $v_{6}[J]$ & $v_{8}[J+c]$\tabularnewline
			 $v_{7}[J+c]$ & $v_{9}[J+c]$\tabularnewline
			 $v_{1}[I]$ & $v_{7}[I+1-a]$\tabularnewline
			 $v_{8}[(I-a)\cap(J+c-1)]$ & $v_{6}[(I-a-c+1)\cap J]$\tabularnewline
			 $v_{4}[(I+1-a-c)\cap J]$ & $v_{3}[(I-a-c+1)\cap J]$\tabularnewline
			\end{tabular}
	 \end{center}

	 \textbf{Step 3.}
	 Since $\langle a+c-1\rangle\cap\langle b+c\rangle=0$ and
   $\langle b+c\rangle =\langle N/n\rangle $,
   it follows that $a+c-1\not\equiv0\pmod{N/n}$. The order
	 of $a+c-1$ modulo $N/n$ is the same as modulo $N$, as $\langle
	 a+c-1\rangle\cap\langle b+c\rangle=0$. Since $n|\langle a+c-1\rangle |\not=N$,
	 we obtain that $a+c-1\not\equiv-1\pmod{N/n}$.  Using the
	 prime factorization of $N$ and three nonzero parameters, from the conditions
	 on the parameters we also get that $N/n\geq6$.  (The generators of the
	 groups generated by the nonzero parameters have at least two distinct prime
	 factors.)

	 \textbf{Step 4.}
	 Let $J$ be a set of representatives for $\ZN/\langle b+c\rangle$.  Let
	 \[
		 I=J\cup\{i+\langle b+c\rangle \mid 0\leq i\leq(N/n-1)/2\}.
	 \]
	 Then $|J|=N/n$ and $|I|\leq n\frac{N/n+1}2+\frac{N/n-1}{2}$. By Lemma
	\ref{lem:ineq_12C}, $|I|+|J|\leq N$.  So it is enough to show
	that \eqref{eq:plague12C} is a plague.
	 For that purpose, we use Example \ref{exa:game12C} with $m=N/n$ and 
	 $\lambda=1-a-c$.
\end{proof}

\begin{lem} \label{lem:12C_imm}
	Let $\Sigma$ be a covering of $\Sigma_{12C}$ with simply intersecting cycles.
	Then $\imm(\Sigma)\leq\omega(\Sigma)$.
\end{lem}

\begin{proof}
	As in the proof of Lemma \ref{lem:Hoq16_imm}, without loss of generality we
	may assume that $|\langle
	1+a\rangle|\leq|\langle1-a-c\rangle|\leq|\langle-b\rangle|\leq|\langle
	b+c\rangle|$.  A straightforward computation shows that
	\begin{equation}
		\label{eq:Hoq16_weights}
		\wg(\Sigma)=\begin{cases}
			1/3 & \text{ if $\Sigma $ is the trivial covering of $\Sigma_{12C}$,}\\
			7/24 & \text{ if $1+a,1-a-c,b,b+c\equiv 0\pmod N$, $N=2$,}\\
			7/24 & \text{ if $1+a,1-a-c,b\equiv 0\pmod N$, $b+c\not \equiv 0\pmod N$,}\\
			41/144 & \text{ if $1+a,1-a-c\equiv 0\pmod N$, $b,b+c\not \equiv 0\pmod N$,}\\
			13/48 & \text{ if $1+a\equiv 0\pmod N$, $1-a-c,b,b+c\not \equiv 0\pmod N$,}\\
			1/4 & \text{ if $1+a,1-a-c,b,b+c\not \equiv 0\pmod N$.}
		\end{cases}
	\end{equation}
	Hence the claim follows from Lemma \ref{lem:Hoq16_imm}.
\end{proof}


\begin{lem}
	\label{lem:exception_12C}
	Let $X$ be an injective indecomposable rack. Assume that $X$ has at least one
	Hurwitz orbit $\mathcal{O}$ isomorphic to the covering
	$\Sigma_{12C}^{1;0,0,0}$.  Then $x\triangleright(x\triangleright y)=z$ for
	all $(x,y,z)\in\mathcal{O}$.
\end{lem}

\begin{proof}
	See the proof of \cite[Proposition 9]{MR2891215}.
\end{proof}


\subsection{The graph $\Sigma_{18A}$}
\label{18A}

\setlength{\unitlength}{1.5cm}
\begin{figure}[h]
  \begin{graph}(6,6)
    \roundnode{n1}(.5,4.5)
    \roundnode{n2}(1.5,5.5)
    \roundnode{n3}(1.5,4.5)
    \roundnode{n4}(4.5,5.5)
    \roundnode{n5}(5.5,4.5)
    \roundnode{n6}(4.5,4.5)
    \roundnode{n7}(2.5,4.5)
    \roundnode{n8}(3.5,4.5)
    \roundnode{n9}(3,3.5)
    \roundnode{n10}(3,2.5)
    \roundnode{n11}(3.5,1.5)
    \roundnode{n12}(2.5,1.5)
    \roundnode{n13}(.5,1.5)
    \roundnode{n14}(1.5,1.5)
    \roundnode{n15}(1.5,.5)
    \roundnode{n16}(4.5,1.5)
    \roundnode{n17}(5.5,1.5)
    \roundnode{n18}(4.5,.5)
    \diredge{n1}{n2}
    \diredge{n2}{n3}
    \diredge{n3}{n1}
    \diredge{n4}{n5}
    \diredge{n5}{n6}
    \diredge{n6}{n4}
    \diredge{n7}{n8}
    \diredge{n8}{n9}
    \diredge{n9}{n7}
    \diredge{n10}{n11}
    \diredge{n11}{n12}
    \diredge{n12}{n10}
    \diredge{n13}{n14}
    \diredge{n14}{n15}
    \diredge{n15}{n13}
    \diredge{n16}{n17}
    \diredge{n17}{n18}
    \diredge{n18}{n16}
    \edge{n1}{n13}[\yedge]
    \edge{n2}{n4}[\yedge]
    \edge{n3}{n7}[\yedge]
    \edge{n5}{n17}[\yedge]
    \edge{n6}{n8}[\yedge]
    \edge{n9}{n10}[\yedge]
    \edge{n11}{n16}[\yedge]
    \edge{n12}{n14}[\yedge]
    \edge{n15}{n18}[\yedge]
		\edgetext{n1}{n2}{$-1$}
		\edgetext{n4}{n5}{$-1$}
		\edgetext{n16}{n17}{$-1$}
		\edgetext{n13}{n14}{$-1$}
		\edgetext{n7}{n8}{$-1$}
		\edgetext{n10}{n11}{$-1$}
		\nodetext{n1}(-.35,0){$v_1$}
		\nodetext{n1}(-.35,-.25){$a$}
		\nodetext{n2}(-.35,0){$v_2$}
		\nodetext{n2}(.25,.25){$0$}
		\nodetext{n3}(0,-.25){$v_3$}
		\nodetext{n4}(.35,0){$v_4$}
		\nodetext{n4}(-.25,.25){$1$}
		\nodetext{n5}(.35,0){$v_5$}
		\nodetext{n5}(.35,-.25){$0$}
		\nodetext{n6}(0,-.30){$v_6$}
		\nodetext{n7}(0,.25){$v_7$}
		\nodetext{n8}(0,.25){$v_8$}
		\nodetext{n9}(.35,0){$v_9$}
		\nodetext{n10}(-.35,0){$v_{10}$}
		\nodetext{n11}(0,-.25){$v_{11}$}
		\nodetext{n12}(0,-.25){$v_{12}$}
		\nodetext{n13}(-.45,0){$v_{13}$}
		\nodetext{n13}(-.50,.25){$1-a$}
		\nodetext{n14}(0,.25){$v_{14}$}
		\nodetext{n15}(-.45,0){$v_{15}$}
		\nodetext{n15}(.35,-.25){$1$}
		\nodetext{n16}(0,.25){$v_{16}$}
		\nodetext{n17}(.45,0){$v_{17}$}
		\nodetext{n17}(.35,.25){$1$}
		\nodetext{n18}(.45,0){$v_{18}$}
		\nodetext{n18}(-.35,-.25){$0$}
		\nodetext{n3}(.25,.25){$c$}
		\nodetext{n7}(-.30,-.25){$1-c$}
		\nodetext{n8}(.25,-.25){$0$}
		\nodetext{n6}(-.25,.25){$1$}
		\nodetext{n9}(-.25,-.25){$d$}
		\nodetext{n10}(.35,.25){$1-d$}
		\nodetext{n14}(.25,-.25){$0$}
		\nodetext{n12}(-.25,.25){$1$}
		\nodetext{n11}(.25,.25){$b$}
		\nodetext{n16}(-.45,-.25){$1-b$}
  \end{graph}
  \caption{Schreier graph $\Sigma _{18A}$}
  \label{fig:Hoq17}
 \end{figure}

\begin{lem}
	\label{lem:Hoq17}
  Let $\Sigma $ be a covering of $\Sigma_{18A}$ in Figure \ref{fig:Hoq17}
  with simply intersecting cycles. Then
  \begin{enumerate}
    \item $\langle -a+c+d+1\rangle\cap\langle a+1\rangle=0$,
    \item $\langle -a+c+d+1\rangle\cap\langle -c+1\rangle=0$,
    \item $\langle -a+c+d+1\rangle\cap\langle -b-d\rangle=0$,
    \item $\langle -a+c+d+1\rangle\cap\langle b\rangle=0$,
    \item $\langle a+1\rangle\cap\langle -c+1\rangle=0$,
    \item $\langle a+1\rangle\cap\langle -b-d\rangle=0$,
    \item $\langle a+1\rangle\cap\langle b\rangle=0$,
    \item $\langle -c+1\rangle\cap\langle -b-d\rangle=0$,
    \item $\langle -b-d\rangle\cap\langle b\rangle=0$.
  \end{enumerate}
\end{lem}

\begin{proof}
	The $xy$-cycles and their labels are: $(v_1\;v_{14}\;v_{10}\;v_7)$ with
  $-a+c+d+1$, $(v_2\;v_5\;v_{18}\;v_{13})$ with $1+a$, $(v_3\;v_8\;v_4)$ with
	$1-c$, $(v_6\;v_9\;v_{11}\;v_{17})$ with $-b-d$ and
	$(v_{12}\;v_{15}\;v_{16})$ with $b$. The $yx$-cycles are:
	$(v_1\;v_4\;v_{17}\;v_{15})$ with $1+a$, $(v_2\;v_7\;v_6)$ with $1-c$,
	$(v_3\;v_{13}\;v_{12}\;v_9)$ with $-a+c+d+1$, $(v_5\;v_8\;v_{10}\;v_{16})$
	with $-b-d$ and $(v_{11}\;v_{14}\;v_{18})$ with $b$.  
	Lemma	\ref{lem:xy_and_yx_cycles} implies the claim.
\end{proof}

\begin{lem} \label{lem:18A_imm}
	Let $\Sigma$ be a covering of $\Sigma_{18A}$ with simply intersecting cycles.
	Then $\imm(\Sigma)\leq \wg(\Sigma )$.
\end{lem}

\begin{proof}
 Let $I$ be a set of representatives for $\ZN/\langle
	b\rangle$. We claim that \[
	P=v_1[*]\cup v_2[*]\cup v_3[*]\cup v_5[*]\cup
	v_{11}[I]
	\]
	is a plague. With the sequence of pivots
  $v_1[*]$, $v_3[*]$, $v_4[*]$,
	$v_8[*]$, $v_7[*]$, $v_6[*]$, $v_9[*]$ and $v_5[*]$ we obtain that
  $P$ spreads to
  $$ v_{13}[*] \cup v_7[*] \cup v_6[*] \cup v_9[*] \cup v_8[*]
     \cup v_4[*] \cup v_{10}[*] \cup v_{17}[*].$$
  Then with the sequence of pivots $v_{16}[I-b]$, $v_{15}[I-b]$, and
	$v_{12}[I-b]$ we see that $P$ spreads to
  $v_{18}[I-b] \cup v_{14}[I-b] \cup v_{11}[I-b]$.
  Hence the claim follows from Example \ref{exa:9A}.

The cycle structure of $\Sigma $ at the vertices is the following:
\begin{align}
  	\label{eq:Hoq17_v1v13}v_1[i], v_{13}[i]: & \text{ cycles of length $4|\langle -a+c+d+1\rangle|$ and $4|\langle 1+a\rangle|$,}\\
		\label{eq:Hoq17_v2v4}v_2[i], v_4[i]: & \text{ cycles of length $4|\langle 1+a\rangle|$ and $3|\langle 1-c\rangle|$,}\\
		\label{eq:Hoq17_v3v7}v_3[i], v_7[i]: & \text{ cycles of length $3|\langle
    1-c\rangle|$ and $4|\langle -a+c+d+1\rangle|$,}\\
		\label{eq:Hoq17_v5v17}v_5[i], v_{17}[i]: & \text{ cycles of length $4|\langle 1+a\rangle|$ and $4|\langle -b-d\rangle|$,}\\
		\label{eq:Hoq17_v6v8}v_6[i],v_8[i]: & \text{ cycles of length $4|\langle -b-d\rangle|$ and $3|\langle 1-c\rangle|$,}\\
		\label{eq:Hoq17_v9v10}v_9[i],v_{10}[i]: & \text{ cycles of length $4|\langle -b-d\rangle|$ and $4|\langle -a+c+d+1\rangle|$,}\\
		\label{eq:Hoq17_v11v16}v_{11}[i],v_{16}[i]: & \text{ cycles of length $4|\langle -b-d\rangle|$ and $3|\langle b\rangle|$,}\\
		\label{eq:Hoq17_v12v14}v_{12}[i],v_{14}[i]: & \text{ cycles of length $4|\langle -a+c+d+1\rangle|$ and $3|\langle b\rangle|$,}\\
		\label{eq:Hoq17_v15v18}v_{15}[i],v_{18}[i]: & \text{ cycles of length $4|\langle 1+a\rangle|$ and $3|\langle b\rangle|$,}
	\end{align}
	for all $i\in\ZN$. If $b\equiv 0\pmod N$ and $c\equiv1\pmod N$ then $\imm
  (\Sigma )\le 5/18=\wg (\Sigma )$. Otherwise, using a graph isomorphism,
  we may assume that $b\not\equiv 0\pmod N$. Then $\imm (\Sigma )\le
  (4N+N/2)/18N=1/4\le \wg (\Sigma )$.
\end{proof}

\subsection{The graph $\Sigma_{24A}$}
\label{24A}

\setlength{\unitlength}{1.5cm}
\begin{figure}[h]
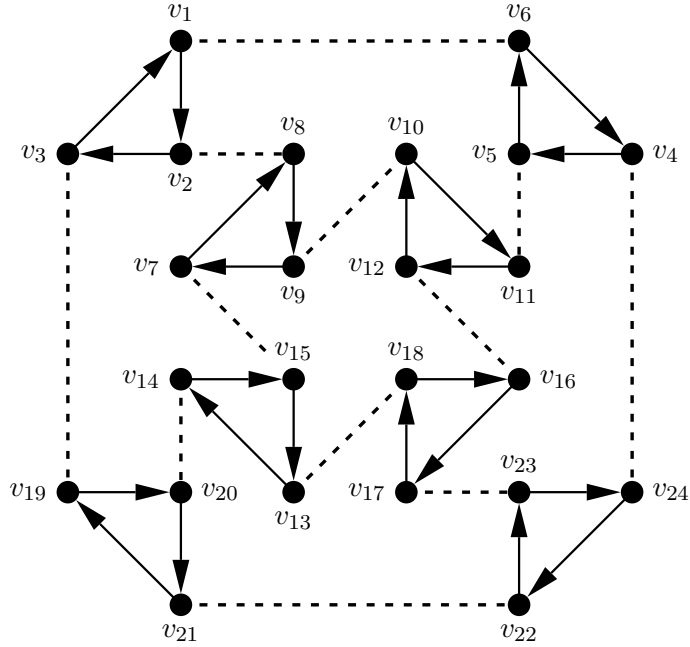

  \begin{graph}(6,6)
    \roundnode{n1}(1.5,5.5)
    \roundnode{n2}(1.5,4.5)
    \roundnode{n3}(.5,4.5)
    \roundnode{n4}(5.5,4.5)
    \roundnode{n5}(4.5,4.5)
    \roundnode{n6}(4.5,5.5)
    \roundnode{n7}(1.5,3.5)
    \roundnode{n8}(2.5,4.5)
    \roundnode{n9}(2.5,3.5)
    \roundnode{n10}(3.5,4.5)
    \roundnode{n11}(4.5,3.5)
    \roundnode{n12}(3.5,3.5)
    \roundnode{n13}(2.5,1.5)
    \roundnode{n14}(1.5,2.5)
    \roundnode{n15}(2.5,2.5)
    \roundnode{n16}(4.5,2.5)
    \roundnode{n17}(3.5,1.5)
    \roundnode{n18}(3.5,2.5)
    \roundnode{n19}(.5,1.5)
    \roundnode{n20}(1.5,1.5)
    \roundnode{n21}(1.5,.5)
    \roundnode{n22}(4.5,.5)
    \roundnode{n23}(4.5,1.5)
    \roundnode{n24}(5.5,1.5)
    \diredge{n1}{n2}
    \diredge{n2}{n3}
    \diredge{n3}{n1} 
    \diredge{n4}{n5}
    \diredge{n5}{n6}
    \diredge{n6}{n4} 
    \diredge{n7}{n8} 
    \diredge{n8}{n9}
    \diredge{n9}{n7}
    \diredge{n10}{n11} 
    \diredge{n11}{n12}
    \diredge{n12}{n10}
    \diredge{n13}{n14} 
    \diredge{n14}{n15}
    \diredge{n15}{n13}
    \diredge{n16}{n17} 
    \diredge{n17}{n18}
    \diredge{n18}{n16}
    \diredge{n19}{n20}
    \diredge{n20}{n21}
    \diredge{n21}{n19} 
    \diredge{n22}{n23}
    \diredge{n23}{n24}
    \diredge{n24}{n22} 
    \edge{n1}{n6}[\yedge]
    \edge{n2}{n8}[\yedge]
    \edge{n3}{n19}[\yedge]
    \edge{n4}{n24}[\yedge]
    \edge{n5}{n11}[\yedge]
    \edge{n7}{n15}[\yedge]
    \edge{n9}{n10}[\yedge]
    \edge{n12}{n16}[\yedge]
    \edge{n13}{n18}[\yedge]
    \edge{n14}{n20}[\yedge]
    \edge{n17}{n23}[\yedge]
    \edge{n21}{n22}[\yedge]
		\nodetext{n1}(0,.25){$v_1$}
		\nodetext{n2}(0,-.25){$v_2$}
		\nodetext{n3}(-.30,0){$v_3$}
		\nodetext{n4}(.30,0){$v_4$}
		\nodetext{n5}(-.30,0){$v_5$}
		\nodetext{n6}(0,.25){$v_6$}
		\nodetext{n7}(-.30,0){$v_7$}
		\nodetext{n8}(0,.25){$v_8$}
		\nodetext{n9}(0,-.25){$v_9$}
		\nodetext{n10}(0,.25){$v_{10}$}
		\nodetext{n11}(0,-.25){$v_{11}$}
		\nodetext{n12}(-.35,0){$v_{12}$}
		\nodetext{n13}(0,-.25){$v_{13}$}
		\nodetext{n14}(-.35,0){$v_{14}$}
		\nodetext{n15}(0,.25){$v_{15}$}
		\nodetext{n16}(.35,0){$v_{16}$}
		\nodetext{n17}(-.35,0){$v_{17}$}
		\nodetext{n18}(0,.25){$v_{18}$}
		\nodetext{n19}(-.35,0){$v_{19}$}
		\nodetext{n20}(.35,0){$v_{20}$}
		\nodetext{n21}(0,-.25){$v_{21}$}
		\nodetext{n22}(0,-.25){$v_{22}$}
		\nodetext{n23}(0,.25){$v_{23}$}
		\nodetext{n24}(.35,0){$v_{24}$}
  \end{graph}
  \caption{Schreier graph $\Sigma _{24A}$}
  \label{fig:Hoq18}
\end{figure}

\begin{lem} \label{lem:24A_imm}
	Let $\Sigma$ be a covering of $\Sigma_{24A}$ with simply intersecting cycles.
	Then $\imm(\Sigma)\leq1/4=\wg(\Sigma)$.
\end{lem}

\begin{proof}
	First we prove that $\imm(\Sigma)\leq1/4$.  A straightforward computation shows that 
	$v_1[*]\cup v_2[*]\cup v_3[*]\cup v_4[*]\cup v_7[*]\cup v_{13}[*]$
	is a plague. 
	To prove that $\wg(\Sigma)=1/4$ observe that in every covering all cycles
	have length $\geq4$. 
\end{proof}

\subsection{The proof of Theorem \ref{thm:percolation}}
    By Proposition \ref{pro:homogeneous_spaces}, the Schreier graphs to
    consider are those in Figures \ref{fig:Hoq1}--\ref{fig:Hoq18}. In this section
    we proved that only the Schreier graphs $\Sigma _{1A}$,
    $\Sigma _{3A}$, $\Sigma _{4A}$, $\Sigma _{6A}$, $\Sigma _{6D}$, $\Sigma
    _{7A}$, $\Sigma _{8A}$, $\Sigma _{9A}$, $\Sigma _{12A}$, $\Sigma _{12B}$,
    $\Sigma _{12C}$, $\Sigma _{18A}$, and $\Sigma _{24A}$ have coverings with
    simply intersecting cycles. For each of these coverings we determined an
    upper bound for the immunity and then we proved that this upper bound can be
    bounded from above by $\omega$. The corresponding claims are
    Lemmas~\ref{lem:1A_imm}, \ref{lem:3A_imm}, \ref{lem:4A_imm},
    \ref{lem:6A_imm}, \ref{lem:6D_imm}, \ref{lem:7A_imm}, \ref{lem:8A_imm},
    \ref{lem:9A_imm}, \ref{lem:12A_imm}, \ref{lem:12B_imm}, \ref{lem:12C_imm},
    \ref{lem:18A_imm}, and \ref{lem:24A_imm}.

\section{Nichols algebras with many cubic relations}
\label{section:manycubic}

We refer to \cite{MR1913436} 
for an introduction to Nichols algebras and
Yetter-Drinfeld modules. Some elementary facts can also be found in
Section~\ref{section:intro}.
Let $\fie$ be a field, $G$ be a group, and $V$ be an absolutely irreducible
finite-dimensional Yetter-Drinfeld module over the group algebra $\fie G$.
Recall that $V$ decomposes as $V=\oplus_{x\in G}V_x$, where $V_x=\{v\in
	V\mid\delta(v)=x\otimes v\}$ and $\delta :V\to \fie G\otimes V$ is the left
  coaction of $\fie G$ on $V$.  The \emph{support} of $V$ is the set 
\[
\mathrm{supp}(V)=\{x\in G\mid V_x\ne0\}.
\]
It is well-known that $(V,c)$ is a braided vector space,
where $c\in \mathrm{Aut}(V\otimes V)$ is defined by
$$ c(u\otimes v)=gv\otimes u\quad \text{for all $u\in V_g$,
   $g\in \supp V$, $v\in V$,} $$
and $c$ satisfies the braid relation on $V\otimes V\otimes V$.

\begin{defn}
Let $G,H$ be groups. We say
that two Yetter-Drinfeld modules $V\in \ydG $, $W\in \ydH $ are
\emph{bg-equivalent}
if there exists a bijection
$\varphi :\supp V\to \supp W$ and a
linear isomorphism $\psi :V\to W$ such that
$$\psi (V_g)=W_{\varphi (g)},\quad
\psi (gv)=\varphi(g)\psi (v)$$
for all $g,x\in \supp V$, $v\in V$. The pair $(\psi ,\varphi )$
is then called a \emph{bg-equivalence} between $V$ and $W$.
We also say that the Nichols algebras
$\NA (V)$ and $\NA (W)$ are \emph{bg-equivalent}.
\end{defn}

The reason for $b$ and $g$ in the definition is the following.

\begin{lem}
  Let $G,H$ be groups and let $V\in \ydG $, $W\in \ydH $.
  If $(\psi ,\varphi )$ is a bg-equivalence between $V$ and $W$,
  then $\psi :V\to W$ is an isomorphism of braided vector spaces,
  which maps the $G$-homogeneous components of $V$ to the
  $H$-homogeneous components of $W$.
\end{lem}

\begin{proof}
  Let $v\in V$, $g\in G$, and $u\in V_g$. Then $\psi (u)\in W_{\varphi (g)}$
  and hence
  \begin{align*}
    (\psi \otimes \psi ) (c(u\otimes v))=&\;\psi (gv)\otimes \psi (u)\\
    =&\;\varphi (g)\psi (v)\otimes \psi (u)\\
    =&\;c(\psi (u)\otimes \psi (v)).
  \end{align*}
  Thus $\psi $ is an isomorphism of braided vector spaces. The rest is clear
  from the assumptions on $\psi $.
\end{proof}

In particular, Nichols algebras of bg-equivalent Yetter-Drinfeld modules are
isomorphic as algebras and coalgebras.

Recall from \cite[Section 2.1]{MR2891215} that a Nichols
algebra $\toba(V)$ has \emph{many cubic relations} if 
\begin{equation}
	\label{eq:cubic_relations}
	\dim\ker(1+c_{12}+c_{12}c_{23})\geq\frac13\dim V\left((\dim V )^2-1\right).
\end{equation}

Let $X=\mathrm{supp}(V)$, $x\in X$ and $e=\dim V_x$.  Assume that $X$ is an
indecomposable rack of size $d>1$.  Since $V=\oplus_{x\in X}V_x$, we conclude
that $V^{\otimes3}=\bigoplus_{\mathcal{O}}V_{\mathcal{O}}^{\otimes3}$, where
the direct sum is taken over all Hurwitz orbits and
\[
	V_{\mathcal{O}}^{\otimes3}=\bigoplus_{(x,y,z)\in\mathcal{O}}V_x\otimes V_y\otimes V_z.
\]
Each $V_\mathcal{O}^{\otimes3}$ is invariant under
the map $1 + c_{12} + c_{12}c_{23}$. 

\begin{rem}
	By Theorem \ref{thm:percolation}, the inequality
	$\imm(\mathcal{O})\leq\omega(\mathcal{O})$ holds for $\mathcal{O}\subseteq
	X^3$ if the quotient $\mathcal{\overline{O}}$ does not have
	$xy$-cycles of length $\geq5$.
\end{rem}

\begin{lem} \label{lem:manycubic}
	Assume that $\toba(V)$ has many cubic relations. Then the following hold.
        \begin{enumerate}
          \item $\sum_{\mathcal{O}}\imm(\mathcal{O})\dim V_\mathcal{O}^{\otimes3}\geq\frac{de}{3}((de)^2-1)$,
  	    where the sum is taken over all Hurwitz orbits of $X^3$.
	  \item Assume that $\imm(\mathcal{O})\leq\omega(\mathcal{O})$
            for all $\mathcal{O}\subseteq X^3$.  Then 
	\[
	  \sum_{\mathcal{O}}\wg(\mathcal{O})\dim V_\mathcal{O}^{\otimes3}\geq\frac{de}{3}\left((de)^2-1\right).
	\]
        \end{enumerate}
\end{lem}

\begin{proof}
  (1) follows from \cite[Proposition 6]{MR2891215} and \eqref{eq:cubic_relations}.
  (2) follows from (1).
\end{proof}

%

\begin{lem}
	\label{lem:ineq}
	Assume that $\toba(V)$ has many cubic relations.  Also assume that
	$\imm(\mathcal{O})\leq\omega(\mathcal{O})$ for all $\mathcal{O}\subseteq
	X^3$. Then 
	\[
	\sum_{p,q\ge 1}k_pk_q(3\omega'_{p,q}-1)+\left(\frac1{10}+\frac{1}{8}\right)k_3+\frac1{4}k_4\geq -1.
	\]
	Further, if the covering $\Sigma_{6A}^{4;2,2}$ appears in $X^3$
  then $k_4\ge 4$, otherwise
	\[
	\sum_{p,q\ge 1}k_pk_q(3\omega'_{p,q}-1)+\left(\frac1{10}+\frac{1}{8}\right)k_3\geq -1.
	\]
\end{lem}

\begin{proof}
  The claim on $k_4$ holds by Lemma~\ref{lem:exception_6A}(1).
	We prove the first inequality of the lemma, the second is analogous.
  Since $\dim
	V_{\mathcal{O}}^{\otimes3}=|\mathcal{O}|e^3$ and $(de)^2-1\geq (de)^2-e^2$,
	we may assume that $e=1$. By Lemma \ref{lem:manycubic}(2), 
	\begin{equation}
		\label{eq:ineq_aux}
	\sum_{v\in X^3}\omega(v)=\sum_{\mathcal{O}}\omega(\mathcal{O})|\mathcal{O}|\geq d(d^2-1)/3.
	\end{equation}
	By Lemmas \ref{lem:exception_4A}, \ref{lem:exception_6A} and \ref{lem:exception_12C}, 
	\begin{align*}
	\left|\left\{v\in X^3:v\in\Sigma_{4A}^{5;3,2},\;v\in v_1[*]\right\}\right|&\leq dk_3,\\
	\left|\left\{v\in X^3:v\in\Sigma_{6A}^{4;2,2},\;v\in v_3[*]\right\}\right|&\leq dk_4,\\
	\left|\left\{v\in X^3:v\in\Sigma_{12C}^{1;0,0,0}\right\}\right|&\leq dk_3.
	\end{align*}
	By the definition of $\omega$ in \eqref{eq:omega}, 
	\eqref{eq:ineq_aux} and the last three inequalities imply that
	\[
	\sum_{p,q\ge 1}k_pk_q\omega'_{p,q}d+\frac{dk_3}{30}+\frac{dk_4}{12}+\frac{dk_3}{24}\geq\frac{d(d^2-1)}{3}.
	\]
  Since $\sum k_pk_q=d^2$, the lemma follows.
\end{proof}

%

\begin{cor}
	\label{cor:ineq}
	Assume that $\toba(V)$ has many cubic relations.  Also assume that
	$\imm(\mathcal{O})\leq\omega(\mathcal{O})$ for all $\mathcal{O}\subseteq
	X^3$.  Then 
	\begin{equation}
	\label{eq:all_exceptions}
		(k_3+k_3'-5)^2+\frac{11}{5}k_3+k_3'^2\leq49.
	\end{equation}
	Further, if the covering $\Sigma_{6A}^{4;2,2}$ appears in $X^3$ then $k_4\ge
  4$, otherwise
	\begin{equation}
	\label{eq:no_6A}
		(k_3+k_3'-4)^2+\frac{1}{5}k_3+k_3'^2\leq40.
	\end{equation}
\end{cor}

\begin{proof}
	The claim follows from Lemma \ref{lem:ineq} and the definition of $\omega '_{p,q}$.
	Note that $k_4\le k'_3$.
\end{proof}

\begin{rem}
	\label{rem:ineq}
	Recall that $k_n\ne1$ for all $n\geq3$.  In \cite{MR2891215}, the case where
	$X$ is braided was completely classified. Thus we may assume that $X$ is not
	braided. Hence $k_3'\geq2$.  Corollary \ref{cor:ineq} implies that
	$k'_3\leq 6$.
  Further, if $k_3'=2$ then $k_4\leq 2$ and $0\leq k_3\leq7$. If $k_3'=3$
	then $0\leq k_3\leq6$. If $k_3'=4$ then $0\leq k_3\leq5$. If $k_3'=5$ then
	$0\leq k_3\leq3$. Finally, if $k_3'=6$ then $k_3=0$.
  In particular, $k_3+k'_3\le 9$.
	If the covering $\Sigma_{6A}^{4;2,2}$ does not appear, then we get the same
	solutions except $(k_3,k'_3)=(5,4)$ and $(3,5)$.
\end{rem}

\begin{lem}
	\label{lem:bound_on_size}
	Assume that $\toba(V)$ has many cubic relations. Also assume that
	$\imm(\mathcal{O})\leq\omega(\mathcal{O})$ for all $\mathcal{O}\subseteq
	X^3$. Then $|X|\leq33$.
\end{lem}

\begin{proof}
	By Remark \ref{rem:ineq}, we may assume that $\varphi_x$ with $x\in X$
  moves at most $9$ elements of $X$.
  There are $29$ possible profiles for $\varphi _x$ if we ignore the number of
  fixed points.
	Using Proposition~\ref{pro:incommensurable_profiles} we can exclude $11$
  profiles and are left with the list $1^a 2$, $1^a 3$, $1^a 2^2$, $1^a 4$,
  $1^a 5$, $1^a 2^3$, $1^a 2 4$, $1^a 3^2$, $1^a 6$, $1^a 7$, $1^a 2^4$,
  $1^a 2^2 4$, $1^a 26$, $1^a 4^2$, $1^a 8$, $1^a 3^3$, $1^a 36$, $1^a 9$,
  where $a\ge 1$ is arbitrary.
  Theorem~\ref{thm:size_of_racks} tells that
  $$|X|\le \left(\sum _{j\ge 2}a_j\right)(k'_2-2)+k'_2+1.$$
  Since $k'_2\le 9$
  is the number of points moved by $\varphi _x$, $x\in X$, we conclude that
  $|X|\le 31$ if $\sum _{j\ge 2}a_j\le 3$. The only profile not satisfying
  the latter inequality is $1^a2^4$. For this we get $|X|\le 33$.
  This yields the claim.
\end{proof}

\begin{lem}
	\label{lem:only_aff5}
	Assume that $\toba(V)$ has many cubic relations, and the rack $X$ satisfies
	$k_3'(X)\ne0$.  Also assume that $\imm(\mathcal{O})\leq\omega(\mathcal{O})$
	for all $\mathcal{O}\subseteq X^3$.  Then $X$ is isomorphic to $\D_5$,
	$\D_7$, $\Aff(5,2)$,
	$\Aff(5,3)$, $\Aff(7,2)$ or $\Aff(7,4)$.
\end{lem}

\begin{rem}
	In fact, in Lemma \ref{lem:only_aff5} the rack $X$ has to be isomorphic to
	$\Aff(5,2)$ or $\Aff(5,3)$. The proof of this fact is more complicated and it
	is not needed here. 
\end{rem}

\begin{proof}
	By Lemma \ref{lem:bound_on_size}, $|X|\leq33$.  Since indecomposable quandles
  of size $\leq35$ were classified in \cite{MR2926571}, a straightforward computer
  calculation gives a complete list of indecomposable non-braided injective
  quandles satisfying \eqref{eq:all_exceptions}.  We obtain the list given in the lemma.
  The following table contains some properties of these racks.
	\begin{center}
	\begin{tabular}{c|c|c}
		Quandle & Sizes of orbits & Remark\tabularnewline
		\hline
		$\mathbb{D}_{5}$ & $1,24$ & \tabularnewline
		$\mathrm{Aff}(5,2)$ and $\mathrm{Aff}(5,3)$ & $1,24$ & \tabularnewline
		$\mathbb{D}_{7}$ & $1,48$ & $k_{7}=6$\tabularnewline
		$\mathrm{Aff}(7,2)$ and $\mathrm{Aff}(7,4)$ & $1,42,49$ & 
	\end{tabular}
	\end{center}
	This completes the proof.
\end{proof}

\section{Computation of the $2$-cocycles and the main theorem}
\label{section:2cocycles}

\begin{table}
\begin{center}
\begin{tabular}{|l|l|l|l|l|}
\hline 
Rack & Rank & Dimension & Hilbert series & Remark \tabularnewline
\hline 
$\{1\}$ & 1 & $n\ge 2$ & $(n)_t$ & \\
\hline 
$(12)^{\Sym_3}$ & 3 & $12$ & $(2)^2_t (3)_t$ & \\
\hline
$(12)^{\Sym_3}$ & 3 & $432$ & $(3)_t(4)_t(6)_t(6)_{t^2}$ & $\charf=2$\\
\hline
$(123)^{\Alt_3}$ & 4 & $36$ & $(2)^2_t (3)^2_t$ & $\charf=2$\\
\hline
$(123)^{\Alt_4}$ & 4 & $72$ & $(2)^2_t (3)_t (6)_{t}$ & $\charf\ne2$ \\
\hline
$(123)^{\Alt_4}$ & 4 & $5184$ & $(6)^4_t (2)^2_{t^2}$ & \\
\hline
$(12)^{\Sym_4}$ & 6 & $576$ & $(2)^2_t (3)^2_t (4)^2_t$ & \\
\hline
$(12)^{\Sym_4}$ & 6 & $576$ & $(2)^2_t (3)^2_t (4)^2_t$ & \\
\hline
$(1234)^{\Sym_4}$ & 6 & $576$ & $(2)^2_t (3)^2_t (4)^2_t$ & \\
\hline
$\mathrm{Aff}(5,2)$ & 5 & $1280$ & $(4)^4_{t} (5)_t$ & \\
\hline
$\mathrm{Aff}(5,3)$ & 5 & $1280$ & $(4)^4_{t} (5)_t$ & \\
\hline
$\mathrm{Aff}(7,3)$ & 7 & $326592$ & $(6)^6_{t} (7)_t$ & \\
\hline
$\mathrm{Aff}(7,5)$ & 7 & $326592$ & $(6)^6_{t} (7)_t$ & \\
\hline
$(12)^{\Sym_5}$ & 10 & $8294400$ & $(4)^4_t (5)^2_t (6)^4_t$ & \\
\hline
$(12)^{\Sym_5}$ & 10 & $8294400$ & $(4)^4_t (5)^2_t (6)^4_t$ & \\
\hline
\end{tabular}
\end{center}
\caption{Finite-dimensional elementary Nichols algebras}
\label{tab:nichols}
\end{table}

Let $X$ be one of the racks $\Aff(5,2)$ or $\Aff(5,3)$. 
Recall that $G_X$ is the enveloping group of $X$.  For all $i\in X$ let $x_i$
be the image of $i$ under the canonical map $X\to G_X$.

\begin{lem}{\cite[Lemma 5.4]{MR2803792}}
	\label{lem:centralizer}
	The centralizer of $x_1$ in $G_X$ is the cyclic group generated by $x_1$.	
\end{lem}

Let $\rho$ be an absolutely irreducible representation of $C_{G_X}(x_1)$.
By straightforward but lengthy calculations we obtain the following.

\begin{lem}
	\label{lem:dimker_24}
	Let $\mathcal{O}\subseteq X^3$ be a Hurwitz orbit of size $24$. Then
	\[
	\dim\ker(1+c_{12}+c_{12}c_{23})|_{V_\mathcal{O}^{\otimes3}}\leq\begin{cases}
 		8 & \text{if $\rho (x_1)=-1$,}\\ 
 		5 & \text{otherwise.}
	\end{cases}
	\]
\end{lem}

\begin{thm}
	\label{thm:main}
	Let $G$ be a group, $V$ be a finite-dimensional absolutely irreducible
	Yetter-Drinfeld module over $G$, and $X=\supp V$.  Assume that $X$ is an
	indecomposable rack such that 
	no quotient of a Hurwitz orbit in $X^3$
	contains an $xy$-cycle of length $\geq5$.  
	Also assume that $\toba(V)$ has many cubic relations.
  
  (1) The rack $X$ is isomorphic to one of the following:
	\begin{gather*}
	(12)^{\mathbb{S}_{n}}\text{ for }n\in\{3,4,5\},\\
	(1234)^{\mathbb{S}_{4}},
	(123)^{\mathbb{A}_{4}},\\
	\mathrm{Aff}(p,\alpha)\text{ for }(p,\alpha)\in\{(5,2),(5,3),(7,3),(7,5)\}.
	\end{gather*}

  (2) The Nichols algebra $\toba(V)$ is bg-equivalent to
  one of the Nichols algebras listed in
	Table~\ref{tab:nichols}.
  
  (3) The Hilbert series of $\NA (V)$ is $t$-integral of depth two.
\end{thm}

\begin{rem}
	Table~\ref{tab:nichols} does not contain the information on the
	$2$-cocycle of $X$ corresponding to $V$. For the racks $\Aff(5,2)$ and
	$\Aff(5,3)$ we use the constant $2$-cocycle $-1$, see the proof of Theorem
	\ref{thm:main}.  For the others, see \cite[Appendix A]{MR2891215}.
\end{rem}

\begin{proof}
	In \cite[Theorem 11]{MR2891215} the theorem was proved under the additional
  assumption that
        $X$ is braided. One obtains the same racks except 
        $\Aff(5,2)$ and $\Aff(5,3)$. Assume now that $X$ is not braided, i.e.,
	$k_3'(X)\ne0$. By Lemma \ref{lem:only_aff5}, $X$ is one of the racks
        $\Aff(5,2)$, $\Aff(5,3)$, $\D_5$, $\D_7$, $\Aff(7,2)$ and $\Aff(7,4)$.
        The racks $\D_5$,
	$\D_7$, $\Aff(7,2)$ and $\Aff(7,4)$ can be excluded
	since they have $xy$-cycles of length $\ge 5$ in some Hurwitz orbit quotients.

        Let now $X=\Aff(5,2)$ or $X=\Aff(5,3)$.
	Then $X^3$ decomposes into five Hurwitz orbits of size one and five of size
	$24$.  Since $\toba(V)$ has many cubic relations, \eqref{eq:cubic_relations}
	and Lemma \ref{lem:dimker_24} imply that $\rho (x_1)=-1$. Moreover, in this case
	\eqref{eq:cubic_relations} holds and the Nichols algebras
        of these $V$ are known to be finite-dimensional, see
        e.g.~\cite[Proposition 5.15]{MR2803792}. The claim on the Hilbert series is known for
        all Nichols algebras in Table~\ref{tab:nichols}.
\end{proof}

\begin{rem} \label{rem:weakerassumption}
	In Theorem \ref{thm:main} the assumption about the length of an $xy$-cycle in
	the quotient of a Hurwitz orbit in $X^3$ can be replaced by the following
	weaker assumption, see Theorem \ref{thm:percolation}: 
	\[
	\imm(\mathcal{O})\leq\omega(\mathcal{O})
	\]
	for all Hurwitz orbits $\mathcal{O}$ of $X^3$. To prove this, we exclude the
	racks $\D_5$, $\D_7$, $\Aff(7,2)$ and $\Aff(7,4)$. Straightforward
	calculations yield the information listed in Table~\ref{tab:small_racks}.
	Hence the inequality of Lemma \ref{lem:manycubic}(1) cannot be satisfied for
	the racks $\D_7$, $\Aff(7,2)$ and $\Aff(7,4)$. The rack $\D_5$ cannot be
	excluded with this argument. However, it can be excluded by computing 
	\[
	\dim\ker(1+c_{12}+c_{12}c_{23})|_{V_\mathcal{O}^{\otimes3}}, 
	\]
	where $\mathcal{O}$ is an orbit of size $24$. This is essentially what we did
	for the racks $\Aff(5,2)$ and $\Aff(5,3)$.
\end{rem}

	\begin{table}
	\begin{center}
		\begin{tabular}{|c|c|c|c|}
			\hline
			Rack & Size of orbits & Number & Bound for immunity\tabularnewline
			\hline
			$\mathbb{D}_{5}$ & $1$ & $5$ & $1$\tabularnewline
			 & $24$ & $5$ & $7/24$\tabularnewline
			\hline
			$\mathbb{D}_{7}$ & $1$ & $7$ & $1$\tabularnewline
			 & $48$ & $7$ & $14/48$\tabularnewline
			\hline
			 $\mathrm{Aff}(7,\alpha)$  & $1$ & $1$ & $1$\tabularnewline
			 $\alpha=2,4$ & $42$ & $1$ & $9/42$\tabularnewline
			  & $49$ & $6$ & $15/49$\tabularnewline
				\hline
			\end{tabular}
	\end{center}
	\caption{Some small racks}
	\label{tab:small_racks}
	\end{table}

We conclude the paper with the following conjecture.

\begin{conjecture} \label{conj:immbound}
	The inequality $\imm(\Sigma)\leq\omega(\Sigma)$ holds for all
	$\mathbb{B}_3$-spaces $\Sigma$ with simply intersecting cycles. 
\end{conjecture}

According to Remark~\ref{rem:weakerassumption}, Conjecture~\ref{conj:immbound}
implies Conjecture~\ref{conj:elemNA} formulated in the introduction.

\medskip
\textbf{Acknowledgement.}
We are greatful to the referee for pointing out some typos and for asking to
clarify some proofs.
Istv\'an Heckenberger was supported by German Research
Foundation via a Heisenberg professorship.  Leandro Vendramin was supported by
Conicet and the Alexander von Humboldt Foundation.
We used \textsf{GAP} (\verb+http://www.gap-system.org/+) 
and \textsf{Rig} (\verb+http://code.google.com/p/rig/+) 
for some 
computations.  
\bibliographystyle{abbrv}
\bibliography{refs}
\end{document}